\documentclass[a4paper,11pt]{amsart}

\usepackage{amssymb,amsmath, pinlabel, mathabx, amscd, tikz}

  \usepackage{enumitem}
 \usepackage{lineno} 
  \usepackage[normalem]{ulem} 
 
 \usepackage[pagebackref]{hyperref}
  \hypersetup{
colorlinks=true,
linkcolor=cyan,
citecolor=cyan
}

\usepackage{mathtools}

\input xy
\xyoption{all}

\usepackage{color}
\newcommand{\red}[1]{\textcolor{red}{#1}}

\newcommand{\blue}[1]{\textcolor{blue}{#1}}

\numberwithin{equation}{section}

\newtheorem{theorem}[equation]{Theorem}
\newtheorem{lemma}[equation]{Lemma}
\newtheorem{proposition}[equation]{Proposition}
\newtheorem{corollary}[equation]{Corollary}

\theoremstyle{definition}
\newtheorem{definition}[equation]{Definition}
\newtheorem{notation}[equation]{Notation}
\newtheorem{remark}[equation]{Remark}
\newtheorem{example}[equation]{Example}

\newcommand{\gras}[1]{{\mathbb #1}} 

\newcommand{\N}{\gras{N}}
\newcommand{\Z}{\gras{Z}} 
\newcommand{\Q}{\gras{Q}} 
\newcommand{\R}{\gras{R}} 
\newcommand{\C}{\gras{C}}  
\newcommand{\bS}{\gras{S}}
\newcommand{\bP}{\gras{P}}

\newcommand{\cO}{\mathcal{O}}

\newcommand{\cT}{\mathcal{T}}
\newcommand{\cV}{\mathcal{V}}

\newcommand{\de}{\mathbf{i}}
\newcommand{\ex}{\mathbf{e}}
\newcommand{\ic}{\mathbf{c}}

\newcommand{\mult}{\mathbf{m}}
\newcommand{\ld}{\mathbf{l}}
\newcommand{\si}{\mathbf{s}}

\newcommand{\Supp}{{\mathcal{S}}}

\newcommand{\bra}[1]{{\langle#1\rangle}}	

\def\elem(#1,#2){  \{ \frac{#1}  {\overline {\ #2\ }} \} }

\title[The valuative tree is the projective limit of Eggers-Wall trees]{The valuative tree \\ is the 
   projective limit of  Eggers-Wall trees}
\author{Evelia R. Garc\'{\i}a Barroso}
\address{Departamento de Matem\'aticas, Estad\'{\i}stica e I.O.
Secci\'on de Matem\'aticas, Universidad de La Laguna. Calle Astrof\'{\i}sico Francisco S\'anchez, 
La Laguna 38200, Tenerife, Espa\~na.}
   \email{ergarcia@ull.es}
\author{Pedro D. Gonz\'alez P\'erez} 
\address{Instituto de Ciencias Matem\'aticas (CSIC-UAM-UC3M-UCM), Dpto. de \'Algebra, Geometr\'\i a y Topolog\' \i a, Facultad de Ciencias Matem\'aticas,
Universidad Complutense de Madrid, Plaza de las Ciencias 3, Madrid 28070, Espa\~na.}
   \email{pgonzalez@mat.ucm.es}
\author{Patrick Popescu-Pampu}
   \address{Universit{\'e} Lille 1, UFR de Maths., B\^atiment M2\\
     Cit\'e Scientifique, 59655, Villeneuve d'Ascq Cedex, France.}
   \email{patrick.popescu@math.univ-lille1.fr}
\date{7 July 2018}


\begin{document}

\begin{abstract}
 Consider a germ $C$ of reduced curve on a smooth germ $S$ of complex analytic surface. 
    Assume that $C$ contains a smooth branch $L$. Using the Newton-Puiseux series  
    of $C$ relative to any coordinate system $(x,y)$ on $S$ such that $L$ is the $y$-axis, 
    one may define the {\em Eggers-Wall tree} 
    $\Theta_L(C)$ of $C$ relative to $L$.  Its ends are labeled by the branches of $C$ and 
    it is endowed with three natural functions measuring the characteristic exponents of the previous 
    Newton-Puiseux series, their denominators and contact orders.  
    The main objective of this paper is to embed canonically $\Theta_L(C)$ 
    into Favre and Jonsson's valuative tree $\bP(\cV)$ of 
    real-valued semivaluations of $S$ up to scalar multiplication, and to show that this 
    embedding identifies the three natural functions on $\Theta_L(C)$ as pullbacks of 
    other naturally defined functions on $\bP(\cV)$.  As a consequence, we prove an inversion 
    theorem generalizing the well-known Abhyankar-Zariski inversion theorem concerning one branch:  
    if $L'$ is a second smooth branch of $C$, then the valuative embeddings of the Eggers-Wall trees 
    $\Theta_{L'}(C)$ and  $\Theta_L(C)$ identify them canonically, 
    their associated triples of functions being easily expressible in terms of each other.  We prove also 
    that the space $\bP(\cV)$ is the projective limit of Eggers-Wall trees over all choices 
    of curves $C$. As a supplementary result, we explain how to pass from $\Theta_L(C)$ to 
    an associated splice diagram.
    
\end{abstract}

\maketitle

\tableofcontents

\newpage
\section{Introduction}
\label{sec:intro}

In their seminal 2004 book ``{\em The valuative tree}'' \cite{FJ 04}, Favre and Jonsson studied 
the space of real-valued semivaluations  $\mathcal{V}$ on a germ $S$ of smooth 
complex analytic surface. 
They proved that the projectivization $\mathbb{P}(\mathcal{V})$ 
of $\mathcal{V}$ is a compact real tree, called the \emph{valuative tree} 
of the surface singularity $S$. They gave several viewpoints on $\mathbb{P}(\mathcal{V})$: 
as a partially ordered set of normalized semivaluations, as a space of 
irreducible Weierstrass polynomials and as a universal dual graph of modifications of 
$S$. 

\medskip

The main objective of this paper is to present $\mathbb{P}(\mathcal{V})$ 
as a ``\emph{universal Eggers-Wall tree}'', 
relative to any \emph{smooth} reference branch (that is, germ of irreducible curve) $L$ on $S$.  
Namely, we show that $\mathbb{P}(\mathcal{V})$ is  
the projective limit of the \emph{Eggers-Wall trees} $\Theta_L(C)$ of the reduced germs of 
curves $C$ on $S$ which contain $L$. 

\medskip

Given such a germ $C$, let $(x,y)$ be a coordinate system verifying that $L$ is the $y$-axis. 
The tree $\Theta_L(C)$ 
is rooted at an end labeled by $L$ and its other ends are labeled by the remaining branches 
of $C$.  Consider the Newton-Puiseux series $(\eta_i(x))_i$ of these branches of $C$. 
The tree  $\Theta_L(C)$ 
has marked points corresponding to the characteristic exponents of the series 
$\eta_i(x)$ and it is endowed with three natural 
functions: the \emph{exponent} $\ex_L$,  the \emph{index} $\de_L$ and  
the \emph{contact complexity} $\ic_L$ (see Definitions \ref{def:EW} and \ref{Intcoefgen}). 
These functions determine the \textit{equisingularity class} of the 
germ $C$ with chosen branch $L$, that is, 
the oriented topological type of the triple $(S, C, L)$. In order to emphasize this property, 
we explain how to get from $\Theta_L(C)$ the minimal splice diagram 
of $C$ in the sense of Eisenbud and Neumann (see Section \ref{splicediag}).  

\medskip 

The branch $L$ may be seen as an {\em observer}, 
defining a coordinate system $(\ex_L, \de_L, \ic_L)$ on $\Theta_L(C)$. 
Analogously, an \textit{observer}  in the valuative tree $\bP(\mathcal{V})$ 
is either the special point of $S$, or a smooth branch $L$, identified with a suitable semivaluation on it. 
Each observer $R$ determines three functions on the valuative tree,   
the \emph{log-discrepancy} ${\ld}_R $,  
the \emph{self-interaction} ${\si}_R $ and the \emph{multiplicity} $\mult_R$ 
relative to $R$ (see Definitions \ref{logselfgen} 
and \ref{relmult}). 
If one identifies the valuative tree $\mathbb{P}(\mathcal{V})$ 
with the subspace of $\cV$ consisting of those semivaluations which take the value $1$  
on the ideal defining the observer $R$, then the functions $({\ld}_R,  \mult_R ,  {\si}_R )$ 
appear as restrictions of functions defined globally on the space of semivaluations. 

\medskip 

We describe an embedding of the Eggers-Wall tree $\Theta_L(C)$ 
inside the valuative tree $\bP(\mathcal{V})$.
This embedding transforms the exponent plus one $\ex_L +1$  into the log-discrepancy $\ld_L$, 
the index $\de_L$ into the multiplicity $\mult_L$ 
and the contact complexity $\ic_L$ into the self-interaction $\si_L$ (see Theorem \ref{embcurve}).
Our embedding is defined explicitly in terms of Newton-Puiseux series, and is similar to 
Berkovich's construction of seminorms on the polynomial ring $K[X]$ extending  
a given complete non-Archimedean absolute value on a field $K$,  
done by maximizing over closed balls of $K$ (see Remark \ref{Berkrem}). 
Theorem \ref{embcurve} generalizes a result of Favre and Jonsson, 
for a \emph{generic} Eggers-Wall tree relative to the  special point 
(see \cite[Prop. D1, page 223]{FJ 04}). 

\medskip

If the germ of curve $C$ is contained in another reduced germ $C'$, then we get 
a retraction from $\Theta_L(C')$ to $\Theta_L(C)$. 
These retractions provide an inverse system of continuous maps and 
we prove, as announced above,
that their projective limit is homeomorphic to the valuative tree $\bP(\mathcal{V})$ 
(see Theorem \ref{proj_lim}). 
This is the result alluded to in the title of the paper. 

\medskip

We study in which way the triple of functions $({\ld}_R, \si_R, \mult_R)$ changes when 
the observer $R$ is replaced by another one $R'$. We provide explicit formulas 
for this \emph{change of variables} in Propositions \ref{changels}, 
\red{\ref{changesm}} and \ref{changecrd}. 
As an application, we prove an \emph{inversion theorem}
which shows how to pass from the Eggers-Wall tree $\Theta_L (C)$ relative 
to a smooth branch $L$ of $C$
to the tree $\Theta_{L'} (C)$ relative to another smooth branch $L'$ of $C$. 
Our theorem means that the geometric realization of the Eggers-Wall tree, 
with the ends labeled by the branches of $C$, remains unchanged and that one only has to replace 
the triple of functions $(\ex_L, \ic_L, \de_L)$ by $(\ex_{L'}, \ic_{L'}, \de_{L'})$ (see Theorem \ref{muthm}). 
If $L$ and $L'$ are transversal, our result is a geometrization and generalization to
the case of several branches of the classical inversion theorem of Abhyankar and Zariski 
(see \cite{A 67}, \cite{Z 68}). In fact,  
Halphen \cite{H 76} and Stolz \cite{S 1879} already knew it 
in the years 1870, as explained in \cite{GBGPPP 17b}. This inversion theorem 
expresses the characteristic exponents with respect to a coordinate system $(y,x)$ 
in terms of those with respect to $(x,y)$.  
Our approach, passing by the embeddings of the Eggers-Wall trees in the space of 
valuations, provides a conceptual understanding of these results.

\medskip 

Let us describe briefly the structure of the paper. In Section 
\ref{fintrees}  we state the basic definitions and notions about 
finite trees and real trees used in the rest of the paper. 
In Section \ref{curves-EW} we introduce the definitions of 
the Eggers-Wall tree and of the exponent, index and contact complexity functions. 
In Section \ref{sec-inv} we give the statement of our inversion 
theorem for Eggers-Wall trees and we prove it using results of later sections.
In Section \ref{splicediag} we recall basic facts about splice diagrams of links 
in oriented integral homology spheres of dimension $3$ and we explain how 
to transform the Eggers-Wall tree $\Theta_L(C)$ into the minimal splice diagram 
of the link of $C$ inside the $3$-sphere. 
The spaces of valuations and semivaluations which play a relevant role in the paper   
are introduced in Section \ref{valspaces}. 
The multiplicity, the log-discrepancy and the self-interaction functions on the valuative tree  
are introduced in Section \ref{fundcoord}. 
In Section \ref{valemb}, we prove the embedding theorem of the Eggers-Wall 
tree in the valuative tree and we  deduce from it that the valuative 
tree is the projective limit of Eggers-Wall trees. 
Finally, in Section \ref{change-obs} we describe how the coordinate functions on the valuative tree 
vary when we change the observer.

\medskip

\medskip
{\bf Acknowledgements.} 
 This research was partially supported by the French grant ANR-12-JS01-0002-
125 01 SUSI and Labex CEMPI (ANR-11-LABX-0007-01), and also by the Spanish grants 
MTM2016-80659-P, MTM2016-76868-C2-1-P and SEV-2015-0554.

\section{Finite trees and $\R$-trees}       \label{fintrees}

In this section we introduce the basic vocabulary about {\em finite trees} used in the 
rest of the text. Then we define \emph{$\R$-trees}, which are more general than finite trees. 
Our main sources are \cite{FJ 04}, \cite{J 15} and \cite{N 14}, 
   even  if we do not follow exactly their terminology. We define attaching maps from 
   ambient $\R$-trees to subtrees (see Definition \ref{defatt}) 
   and we recall a criterion which allows to see a given 
   compact $\R$-tree as the projective limit of convenient families of finite subtrees, when they 
   are connected by the associated attaching maps (see Theorem \ref{R-theo}). 
   This criterion will be crucial in order 
   to prove in Section \ref{valemb} the theorem stated in the title of the paper.
\medskip

Intuitively, the finite trees are the connected finite graphs without circuits. 
As is the case also for graphs, 
the intuitive idea of tree gets incarnated in several categories: there are 
\emph{combinatorial}, \emph{(piecewise) affine} and  \emph{topological} trees, 
with or without a \emph{root}. Combinatorial trees are special types of 
abstract simplicial complexes:

     \begin{definition}  \label{combtree}
          A {\bf finite combinatorial tree} 
       $\cT$ is formed by a finite set $V(\cT)$ of {\bf vertices} 
     and a set $E(\cT)$ of subsets with two elements of $V(\cT)$, called {\bf edges},
     such that for any pair 
     of vertices, there exists a unique chain of pairwise distinct edges joining them.  
     The {\bf valency} 
     $v(P)$ of a vertex $P$  is the number of edges 
     containing it. A vertex $P$ is called a {\bf ramification point} of $\cT$ if $v(P) \geq 3$ 
     and an {\bf end vertex} (or simply an {\bf end}) if $v(P) = 1$.
   \end{definition}
   
   As a particular case of the general construction performed on any finite abstract simplicial    
   complex, each finite combinatorial tree has a unique geometric realization up to a unique 
   homeomorphism extending the identity on the set of vertices and 
   affine on the edges, which will be called a {\bf finite affine tree}.  
   If we consider an affine tree only up to homeomorphisms, 
    we get the notion of finite topological tree:
    
    \begin{definition} \label{fintree}
      A topological space homeomorphic to a finite affine tree is called a 
      {\bf finite topological tree}  or, simply, a {\bf finite tree}. The {\bf interior} 
      of a finite tree is the set of its points which are not ends. 
      A {\bf finite subtree}  of a given tree is a topological subspace 
      homeomorphic to a finite tree. 
   \end{definition}
   
   The simplest finite trees are reduced to points. Any finite tree is compact.   
   Only the ramification points and the end vertices are 
   determined by the underlying topology. One has to \emph{mark} as special 
   points the vertices of valency $2$ if one wants to remember them. 
   Therefore, we will speak in this case about  {\bf marked finite trees}, 
   in order to indicate that one gives also 
   the set of {\em vertices}, which contains, possibly in a strict way, 
   the set of ramification points and of ends. 
   By definition, a {\bf subtree} $\cT'$ of a \emph{marked} finite tree $\cT$ is 
   a finite subtree of the underlying topological space of $\cT$ such that its ends 
   are  marked points of $\cT$, and its marked points are the marked points of 
   $\cT$ belonging to $\cT'$.

    A {\bf (compact) segment} 
    in a finite tree is a connected subset which is 
   homeomorphic to a (compact) real interval. Each pair of points $P, Q  \in \cT$ 
  is the set of ends of exactly one compact 
   segment, denoted $[P,Q]= [Q,P]$. We speak also about the {\bf half-compact} and the  
   {\bf open} segments $(P,Q]$, $[P,Q)$, $(P,Q)$.

    We will often deal with sets equipped with a partial order, 
    which are usually called \textbf{posets}. 
    The next definition explains how the choice of a \emph{root}  
    for a tree endows it with a structure of poset:
    
   \begin{definition}  \label{finrtree}
      A {\bf finite rooted tree}  is a finite (affine or topological)  tree with a 
      marked vertex, called the {\bf root}. 
     In such a tree $\cT$, the ends which are 
   different from the root  are called the {\bf leaves} 
   of $\cT$. If the root is also an end, we say that $\cT$ is 
   {\bf end-rooted}. Each rooted tree with root $R$ may be 
   canonically endowed with a partial order $\preceq_R$ in the following way:
                  $$  P \preceq_R Q \: \Leftrightarrow \:   [R,P] \subseteq [R,Q].$$  
    \end{definition}

    Each finite marked rooted tree may be seen as a \emph{genealogical tree}, 
    the individuals with a common ancestor corresponding to the vertices, 
    the elementary filiations to the edges and the common ancestor to the root: 
  
\begin{definition} \label{predinit}
  Let $\cT$ be a marked finite rooted tree, with root $R$. 
  For each vertex $P$ of $\cT$ different from $R$,  its {\bf parent}
   $\mathtt{P}(P)$ is  the greatest vertex of $\cT$  
   on the segment $[RP)$. If we define $\mathtt{P}(R) =R$, 
   we get the {\bf parent map}
        $\mathtt{P} : V(\cT) \to V(\cT)$. 
 \end{definition}
      
 \medskip

   One may generalize the notion of finite rooted tree by keeping some of the properties 
   of the associated partial order relation:
   
 \begin{definition} \label{nonmetree}
    A {\bf rooted $\R$-tree}  is a poset $(\cT, \preceq)$ such that:
  \begin{enumerate}       
    \item There exists a unique smallest element $R \in \cT$ (called the {\em root}). 
       
       \item For any $P \in \mathcal{T}$, the set  $\{ Q \in \mathcal{T} \ | \  Q \preceq P \}$ 
         is isomorphic as a poset to a compact interval of $\R$ 
         (reduced to a point when $P = R$). 
       
       \item Any totally ordered convex subset of $\cT$ is isomorphic to an interval of $\R$ 
         (a subset $K$ of a poset $(P, \preceq)$ is called {\bf convex} 
         if $c\in K$ whenever 
        $a \preceq c \preceq b$ and $a,b \in K$).
       
       \item \label{critcond} Every non-empty subset $K$ of $\cT$ has an infimum, 
           denoted $\wedge_{P \in K} P$.   
    \end{enumerate}
    The rooted $\R$-tree $\cT$ is \textbf{complete} 
  if any increasing sequence has an upper bound.
\end{definition}

Every finite rooted tree $\cT$ is a complete rooted $\R$-tree, 
if one works with the partial order $\preceq_R$ defined by its root $R$.     

   \begin{remark} \label{notimpl} 
       We took Definition \ref{nonmetree} from Novacoski's paper \cite{N 14}, where 
       this notion is called instead \emph{rooted non-metric $\R$-tree}. In fact, Novacoski proved that 
       under the hypothesis that conditions (1) and (2) are both satisfied, the fourth one is equivalent 
       to the condition that any {\em two} elements have an infimum 
       (see \cite[Lemma 3.4]{N 14}). He emphasized the fact that 
       condition (4) is not implied by the previous ones, because of a possible phenomenon 
       of \emph{double point}. Glue for instance by the identity map 
       along $[0, 1)$ two copies of the segment $[0, 1]$, endowed with the usual order relation 
       on real numbers. One gets then a poset satisfying conditions (1)--(3) but not condition (4). 
       Indeed, the two images of the number $1$ do not have an infimum. This subtlety  
       was missed in the book \cite{FJ 04}, in which the previous notion 
       was defined (under the name \emph{rooted nonmetric tree}) only by the  conditions (1)--(3) 
       (see \cite[Definition 3.1]{FJ 04}). Property (4) 
       was nevertheless heavily used in the proofs of \cite{FJ 04}. 
       Happily, this does not invalidate some results of the book, 
       because Novacoski showed that the valuative trees studied by Favre and Jonsson satisfy also 
       the fourth condition (see \cite[Theorem 1.1]{N 14}).
   \end{remark}
   
\begin{figure}[h!] 
\vspace*{6mm}
\labellist \small\hair 2pt 
\pinlabel{$R$} at 110 5
\pinlabel{$P$} at 30 94
\pinlabel{$Q$} at 195 140
\pinlabel{$P \wedge Q$} at 120 40
\endlabellist 
\centering 
\includegraphics[scale=0.6]{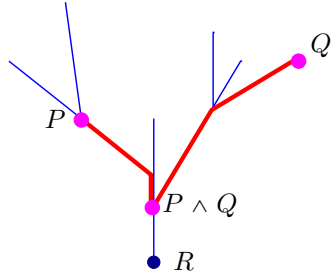} 
\caption{The infimum of two elements in a rooted tree} 
\label{fig:inftwo}
\end{figure}

   Let $\cT$ be a rooted $\R$-tree. If $P, Q$ are any two points on it and if 
   $P \wedge Q$ is their infimum (see Figure \ref{fig:inftwo}), denote 
   by $[P,Q]$ the {\bf compact segment} joining them, defined by:
     $$[P, Q] := \{ A \in \cT  \:  |   \: P \wedge Q \preceq A \preceq P \mbox{ or } 
                                                P \wedge Q \preceq A \preceq Q \}.$$ 
     Obviously, $[P,Q] $ is equal to $[Q, P]$. One defines then $[P, Q):= [P, Q] \: \setminus \:  \{Q\}$, etc. 
    
    \medskip
    
    In the same way as one speaks about affine spaces, which are vector spaces with 
    forgotten origin, we will need the notion of rooted tree with forgotten root:  
    
    \begin{definition}  \label{forgroot}
    An {\bf $\R$-tree} 
    is a rooted $\R$-tree with forgotten root. 
    That is, it is an equivalence class of structures of rooted $\R$-tree on a 
    fixed set, defining the same compact segments.    
    If $\cT$ is an $\R$-tree and $P \in \cT$ is an arbitrary point of it, a  {\bf direction} 
    at $P$ is an equivalence class of the following equivalence relation $\sim_P$ 
    on $\cT \setminus \{P\}$:
       \[Q_1 \sim_P Q_2 \: \Longleftrightarrow \: (P,  Q_1] \cap (P,  Q_2] \neq \emptyset.\]
    The {\bf weak topology} of   the $\R$-tree $\cT$ is the minimal one such 
    that all the directions at all points are open subsets of $\cT$.
    
    If $P \in \cT$, we define the partial order $\preceq_P$, as in Definition \ref{finrtree}. 
    This definition recovers  the \emph{rooted} $\R$-tree structure on the set $\cT$ with root at $P$.
        \end{definition}
      
\medskip
   
The number of directions at a point in a finite tree is equal to its valency.
    The notion of direction allows to extend to $\R$-trees $\cT$ the notion of 
    ramification point. Namely, a point $P \in \cT$ is a {\bf ramification point} 
    if there are at least three directions at $P$.  

\medskip

       \begin{remark} $\;$
       
          \begin{enumerate}
\item  Definition \ref{forgroot} is a reformulation of \cite[Definition 3.5]{FJ 04}. One may 
       define also a notion of \emph{complete} $\R$-tree as the equivalence class of 
       a complete rooted $\R$-tree. This last notion may be defined differently, emphasizing 
       the set of its compact segments (see Jonsson's \cite[Definition 2.2]{J 15}). 
      
  \item In \cite[Section 3.1.2]{FJ 04} the term \emph{tangent vector} is used instead of 
          \emph{direction}. We prefer this last term in order to emphasize the 
          analogy with the usual euclidean space, in which two points $Q_1$ and $Q_2$ 
          are said \emph{to be in the same direction} as seen from an observer $P$ if and only if the 
          segments $(P, Q_1]$ and  $(P, Q_2]$ are not disjoint. 
          
 \item Endowed with the weak topology, each $\R$-tree 
    $\cT$ is Hausdorff (see \cite[Lemma 7.2]{FJ 04}).  
    In that reference a few other tree topologies are defined and studied, but each time 
    starting from supplementary structures on the $\R$-tree, for instance metrics. 
    We will not need them in this paper. 
       \end{enumerate}      
    \end{remark}

    Let us illustrate the previous vocabulary by an example:

  \begin{example} \label{dblcomb}
     Consider the set $\cT := \R \times [0, \infty)$, endowed with the following partial order:
       $$(x_1, y_1) \preceq (x_2, y_2) \: \Longleftrightarrow \ 
                \left\{ \begin{array}{cl}
                              \mbox{ either } & x_1=0 \mbox{ and }   y_1 \leq y_2, \\
                               \mbox{ or } &    y_1 = y_2, |  x_1 |  \leq  | x_2 | \mbox{ and } 
                                                                   x_1 \cdot x_2 \geq 0.  
                          \end{array}   \right. $$
       Its structure is suggested in Figure \ref{fig:Doublecomb}.
       This partial order endows $\cT$ with a structure of rooted $\R$-tree. 
      Its root  is the point $(0,0)$. Notice that the segment $[(x_1, y_1),  (0,0)]$ of  $\cT$ is 
      the union of the segments $[(x_1, y_1),  (0,y_1)]$ and $[(0,y_1)], (0,0)]$.
      The set of ramification points is the vertical half-axis 
    $\{0\} \times [0, \infty)$. 
    At each point of it there are $4$ directions (up, down, right and left), 
    with the exception of $(0,0)$, at 
      which there are only $3$ of them (no down one). 
  \end{example}
   
 \begin{figure}[h!] 
 \vspace*{6mm}
 \centering 
 \includegraphics[scale=0.50]{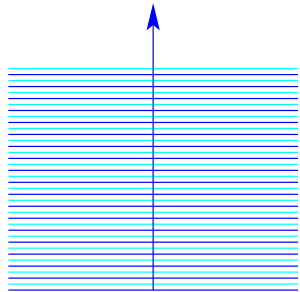} 
 \vspace*{1mm} 
 \caption{An example of $\R$-tree} 
 \label{fig:Doublecomb}
 \end{figure}

   \begin{lemma}\label{attp}
       Let $\cT$ be an $\R$-tree and let 
       $\cT'$ be a closed subtree of $\cT$, for the weak topology. 
       For any $P \in \cT$, there exists a unique point 
       $Q \in \cT'$ such that $[Q, P] \cap \cT' = \{ Q \}$. 
   \end{lemma}
   
   This lemma, whose proof is left to the reader, 
   says simply that if we take a point in a tree, then there is a unique minimal 
   segment joining it to a given closed subtree. Note that 
   $Q = P$ if and only if $P \in \cT'$.
   
   \begin{definition} \label{defatt}
      We call the point $Q$ characterized in  Lemma \ref{attp} {\bf the attaching point  
       of $P$ on $\cT'$} 
       and we denote it $\pi_{\cT'}(P)$. The map 
       $\pi_{\cT'} : \cT \to \cT$ is the {\bf attaching map} 
       of the closed subtree $\cT'$.
   \end{definition}

   Notice that the attaching map $\pi_{\cT'} : \cT \to \cT$ is {\em a retraction} onto $\cT'$. 
    Indeed: 
    $$\pi_{\cT'} \circ \pi_{\cT'} = \pi_{\cT'}  \mbox{ and } \mathrm{im} (\pi_{\cT'}) = \cT'.$$ 
   Sometimes we consider surjective attaching maps, by replacing the target 
   $\cT$ by $\mathrm{im} (\pi_{\cT'})$. The name we chose for $\pi_{\cT'}$ 
   is motivated by the fact that we think of 
   $\pi_{\cT'}(P)$ as the point where the smallest segment of $\cT$ (for the inclusion relation) 
   joining $P$ to $\cT'$ is {\em attached} to $\cT'$. 
 In the Figure \ref{fig:attach} is represented a tree $\cT$ and, 
   with heavier lines, a closed subtree $\cT'$. 
   We have also represented two points $A,B \in \cT$ and their attaching 
  points $\pi_{\cT'}(A), \pi_{\cT'}(B)$ on $\cT'$.

      \begin{figure}[h!] 

\vspace*{6mm}
\labellist \small\hair 2pt 
\pinlabel{$A$} at 58 144
\pinlabel{$B$} at 180 144
\pinlabel{$\pi_{\cT'}(A)$} at 30 72
\pinlabel{$\pi_{\cT'}(B)$} at 153 84
\endlabellist 
\centering 
\includegraphics[scale=0.70]{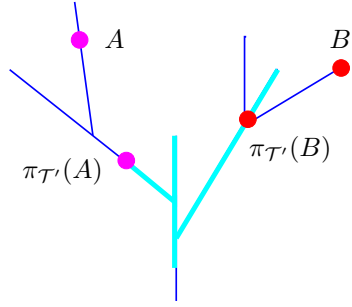} 
\caption{Attaching points on a subtree} 
\label{fig:attach}
\end{figure}

   One has the following property:
      \begin{lemma} \label{atinf}
           Let $\cT$ be an $\R$-tree. Then for any $A, B, C \in \cT$ one has: 
           \[ \pi_{[A, B]} ( C ) = \pi_{[B, C]}(A) = \pi_{[A, C]}(B). \]
           This point may also be characterized as the intersection 
           of the segments joining pairwise the points $A, B, C$. 
           If $\cT$ is rooted at $A$, then the previous point is equal to $B \wedge C$. 
      \end{lemma}
      
      \begin{proof}
      The constructions which allow to define the objects involved in this lemma 
      can be done inside the finite tree which is the union of the 
      segments $[A, B]$, $[B, C]$, and $[C, A]$. Generically, when no one of the three points lies 
      on the segment formed by the other two, this tree has the shape of a star with 
      three legs. Otherwise it is a segment. In any of these cases the assertion is clear. 
      \end{proof} 
      
      Let us introduce a standard name for the trees determined by three points:      

        \begin{definition}  \label{trip}
          If $A, B, C$ are three points of an $\R$-tree, then the union 
          of the segments $[A, B], [B, C], [C, A]$ is the {\bf tripod} generated 
          by them. Its {\bf center} $\langle A, B, C \rangle$ is the point 
          characterized in Lemma \ref{atinf}. 
      \end{definition}

   Notice that finite trees are compact for the weak topology. One has the following 
   characterization of the $\R$-trees which are also compact 
   when endowed with the weak topology (see \cite[Section 2.1]{J 15}):  

  \begin{theorem}  \label{R-theo}
      Let $\cT$ be an $\R$-tree. Let $(\cT_j)_{j \in J}$ be a (possibly infinite) collection 
      of finite subtrees of it. We assume that they form a projective system 
      for the inclusion partial order, 
      that is, for any $j, k \in J$, there exists $l \in J$ such that $\cT_j \subset \cT_l 
      \supset \cT_k$. When $\cT_j \subset \cT_l$, denote by $\pi_j^l: \cT_l \to \cT_j$ 
      the corresponding attaching map. Then:
        \begin{enumerate}
            \item the maps $\pi_j^l$ form a projective system of continuous maps; 
            \item their projective limit $\displaystyle{\lim_{\longleftarrow}} \: \cT_j$ is compact; 
            \item the attaching maps $\pi_j: \cT \to \cT_j$  glue into a continuous map 
               $\pi : \cT \to \displaystyle{\lim_{\longleftarrow}}  \: \cT_j$;
            \item \label{sufmany} 
                if for any two distinct points $A, B \in \cT$, there exists a tree $\cT_j$ 
                such that $\pi_j(A) \neq \pi_j(B)$, then the map $\pi$ is a homeomorphism 
                onto its image. 
            \item $\cT$ is compact if and only if $\pi$ is a homeomorphism onto 
                $\displaystyle{\lim_{\longleftarrow}} \: \cT_j$.
        \end{enumerate} 
  \end{theorem}

   This theorem shows also that compact $\R$-trees may be studied using \emph{sufficiently 
   many} (in the sense of condition (\ref{sufmany}))  of their finite subtrees. 
   
 We will use Theorem \ref{R-theo} in order to prove Theorem \ref{proj_lim}, 
      stated briefly in the title of this paper.

\section{Curve singularities and their Eggers-Wall trees } \label{curves-EW}

In this section we explain the basic notations and conventions used throughout the paper 
about reduced germs $C$ of curves on smooth surfaces.  Then we define the 
Eggers-Wall tree of such a germ relative to a smooth branch contained in it 
(see Definition \ref{def:EW}), as well as 
three natural real-valued functions defined on it, the exponent, the index and the contact 
complexity. We recall how this last function may be expressed in terms of the 
intersection numbers of the branches of $C$ (see Theorem \ref{intcomp}). 
Remark \ref{histcomm}  contains historical comments about the notion of Eggers-Wall tree.

\medskip

All over the text, $S$ denotes a smooth germ of complex algebraic or analytic surface and $O$ its 
special point. We denote by  $\mathcal{O}$  the formal local ring of $S$ at $O$ (the completion of the 
ring of germs at $O$ of holomorphic functions on $S$), by 
 $\mathcal{F}$ its field of fractions, and by  $\mathcal{M}$  its maximal ideal.  
 
A {\bf branch} on $S$ is a germ at $O$ of formal irreducible curve drawn on $S$. 
 A {\bf divisor} on $S$ is an element of the free abelian group generated by the branches on 
 $S$. A divisor is called {\bf effective} if it belongs to the free abelian monoid generated 
 by the branches. 

If $f \in \mathcal{F} \setminus \{ 0 \}$, we denote by $Z(f)$  its \textbf{divisor}. 
This divisor is effective if and only if $f \in \mathcal{O}$.
 If $D$ is an effective divisor through $O$, we denote by $\mathcal{O}(-D)$ the 
     ideal of $\mathcal{O}$ consisting of those functions which vanish along it.  
     As $S$ is smooth, this ideal is principal. Any generator of it is a {\bf defining function} of $D$.            
          The ring  $\mathcal{O}_D := \mathcal{O}/\mathcal{O}(-D) $ is  the local ring of $D$.

A \textbf{model} of $S$ is a proper birational morphism $\psi: (\Sigma, E) \to (S, O)$,
where $\Sigma$ is a \emph{smooth} surface and
the restriction $\psi_{| \Sigma \: \setminus \: E} \colon \Sigma \: \setminus \: E \to S \: \setminus \: \{ O \}$ 
is an isomorphism.  
The preimage $E= \psi^{-1} (O)$, seen as a reduced divisor on $\Sigma$, 
is the \textbf{exceptional curve} of the model $\Sigma$ (or of the morphism $\psi$). 
A point of $E$ is called an \textbf{infinitely near point} of $O$. 
By a theorem of Zariski, $\psi$ is a composition of blowing ups of points, thus
the irreducible components $(E_j)_{j \in J}$ 
of the {\bf exceptional curve}  $E$ are projective lines (see
\cite[Vol.1, Ch. IV.3.4, Thm.5]{S 94}).

\medskip

A {\bf local coordinate system} on $S$ is a pair $(x,y) \in \mathcal{O}$ 
establishing an isomorphism of $\C$-algebras,
$\cO \simeq \C[[x,y]]$, 
 where $\C[[x,y]]$ denotes the $\C$-algebra of formal power series in the variables 
 $x$ and $y$.  
 
 \medskip
 
 The $\C$-algebra $\C[[t]]$ of formal power series in a variable $t$ is endowed 
with the {\bf order} valuation $\nu_t$ which associates to every series the lowest exponent of its terms. 
This ring allows to \emph{parametrize} the branches on $S$:

\begin{definition}  \label{defparam}
   Let $C$ be a branch on $S$. A {\bf parametrization} of $C$ 
   is a germ of formal map $(\C,0) \to (S,O)$ whose image is $C$, 
   that is, algebraically speaking, a morphism $\mathcal{O} \to \C[[t]]$ 
   of $\C$-algebras whose kernel is the principal ideal $\mathcal{O}(-C)$. 
   The parametrization is called {\bf normal}
    if this map is a {\bf normalization} 
   of $C$, that is, if it is of degree one onto its image or, algebraically speaking,  
   if the associated map $\mathcal{O}_D \to \C[[t]]$ induces an 
   isomorphism at the level of fields of fractions. 
\end{definition}

 \begin{example}
   Assume that one works with local coordinates $(x,y)$. Then the branch $C= Z(y^2 - x^3)$ 
   may be parametrized by $(x= t^2, y= t^3)$ and also by $(x= t^4, y= t^6)$. Only 
   the first parametrization is normal. 
 \end{example}

Let $C$ be a \emph{reduced} germ of complex analytic 
curve at $O$, possibly having several branches $(C_i)_{i \in I}$, 
which are by definition the irreducible components of $C$. 
We think also about $C$ as an effective divisor, which allows us to write 
$C = \sum_{i \in I} C_i$. 
We write $C\subseteq D$ if $D$ is another reduced germ containing $C$. 
In such a case, $D-C$, thought as a difference of divisors,
denotes the union of the branches of $D$ which are not branches of $C$. 
We denote by $m_O(C)$ the {\bf multiplicity} of $C$ at $O$. 
If $C$ is defined by $f \in \mathcal{O}$, and if a local coordinate 
system $(x,y)$ is fixed, allowing to express $f$ as a formal 
series in $(x,y)$, then the multiplicity $m_O(C)$ is equal to 
the least total  degree of the monomials appearing in this series. 
One has  $m_O(C) = \sum_{i \in I} m_O(C_i).$

If $D_1$ and $D_2$ are two effective divisors through $O$, we denote by
 $(D_1\cdot D_2)$ 
their {\bf intersection number} at $O$ (also called intersection multiplicity). 
By definition, it is equal to $\infty$ if and only 
if the supports of $D_1$ and $D_2$ have a common branch. If 
$D_k = Z(f_k)$, for $k=1, 2$ then we have that
$(D_1 \cdot D_2) = \dim_{\C} \mathcal{O}/(f_1, f_2).$  
 If one of the two divisors $D_k$ is a branch, for instance $D_1$, then the 
 intersection multiplicity may be computed as the order $\nu_{t_1}(f_2\circ \phi_1)$ 
 in $t_1$ of the series 
 $f_2\circ \phi_1$, where $\phi_1: (\C_{t_1},0) \to (S,O)$ is a normal  
 parametrization of $D_1$ (see \cite[Proposition II.9.1]{BHPV 04})). 
 
 \begin{example}
     Assume that $D_1= Z(y^2 - x^3)$ and $D_2 = Z(y^2 - 2 x^3)$. Both are branches 
    and $(x= t_1^2, y= t_1^3)$ is  a normal parametrization of $D_1$. Therefore:
      $$(D_1 \cdot D_2) = \nu_{t_1}((t_1^3)^2 - 2 (t_1^2)^3)= \nu_{t_1}(- t_1^6) = 6.$$
  \end{example}

 Note that a pair $(x,y) \in \cO^2$ defines a local coordinate system on $S$ if and only if 
 the germs $Z(x)$ and $Z(y)$ are {\bf transversal} smooth branches, that is, if and only 
 if $(Z(x) \cdot Z(y)) =1$.

\medskip

One can study a reduced germ $C$, also called a {\bf plane curve singularity}, 
by using Newton-Puiseux series:

\begin{definition}   \label{def:NPseries}
    A {\bf Newton-Puiseux series} $\eta$ 
    in the variable $x$ is a power series of the form 
   $\psi(x^{1/n})$, where $\psi(t) \in \C[[t]]$  and   $n \in \N^*$. 
    For a fixed $n \in \N^*$, they form the ring $\C[[x^{1/n}]]$. 
    Its field of fractions is denoted $\C((x^{1/n}))$.
If $\eta \in \C[[x^{1/n}]]\:  \setminus \: \{0\}$, then its {\bf support} is the set $\Supp(\eta)$ 
    of exponents of $\eta$ with non-zero coefficient. 
  
   \noindent
  Denote by: 
   \[ \C[[ x^{1/ \N} ]]:= \bigcup_{n \in \N^*} \C[[ x^{1/n}]] \]
   the local $\C$-algebra of Newton-Puiseux series in the variable $x$. 
   The algebra  $\C[[ x^{1/ \N} ]] $  is endowed with the natural {\bf order} valuation:
    \[\nu_x :  \C[[ x^{1/ \N} ]] \longrightarrow \Q_+\cup \{ \infty \}\]
      which associates to each series $\eta = \psi(x^{1/n}) \in \C [[ x^{1/n} ]]$
       the minimum of its support.
\end{definition}

\medskip
 
 Assume that a coordinate system $(x,y)$ is fixed.
 Let $A$ be a branch on $S$ different 
 from $L =Z(x)$. Relative to the coordinate system $(x,y)$, it 
 may be defined by a Weierstrass polynomial 
 $f_A \in \C [[x]][y]$, which is monic, irreducible and of degree $d_A= (L \cdot A)$. 
 For simplicity, we mention only the dependency on $A$, not on the coordinate system $(x,y)$.

 By the Newton-Puiseux theorem, $f_A$ has $d_A$ roots inside $\C [[x^{1/ d_A}]]$. 
We denote by $\mathrm{Zer} (f_A)$ the set of these roots, which are called 
the {\bf Newton-Puiseux roots} of $A$ with respect to the coordinate system $(x,y)$.
These roots  can be obtained from a fixed one  $\eta= \psi(x^{1/ d_A})$ by 
 replacing $x^{1/ d_A}$ by $\gamma \cdot x^{1/ d_A}$, for  $\gamma$ running through  
the  $d_A$-th roots of $1$.

 \medskip

Therefore, all 
 the Newton-Puiseux roots of the branch $A$ have the same exponents. Some of those exponents 
 may be distinguished by looking at the differences of roots:

 \begin{definition}  \label{def:charexp}
     The {\bf characteristic exponents of the branch $A$ relative to $L$} are the $x$-orders 
     $\nu_x(\eta - \eta')$ 
     of the differences between distinct Newton-Puiseux roots $\eta, \eta'$ 
     of $A$ in the coordinate system $(x,y)$. 
 \end{definition}
 
 The fact that we mention only the dependency on $L$ and not on the full coordinate 
 system $(x,y)$ is explained by Proposition  \ref{smoothcv} below.
 The characteristic exponents may be read from a given Newton-Puiseux root 
 $\eta \in \C [[x^{1/ d_A}]]$ of $f_A$ by looking at the increasing sequence of exponents 
 appearing in $\eta$ and by keeping those which cannot be written as a quotient of 
 integers with the same smallest common denominator as the previous ones. In this 
 sequence, one starts from the first exponent which is not an integer. 
 
 One may find information about the history of the notion of \emph{characteristic exponent}  
 in \cite[Section 2]{GBGPPP 17a}. 
 
 \medskip 
We keep assuming that $A$ is a branch. 
 The \emph{Eggers-Wall segment of $A$ relative to $L$} is a geometrical way of encoding 
 the set of characteristic exponents, as well as the sequence of their successive 
 common denominators:

 \begin{definition} \label{def:EWirr}
     The {\bf Eggers-Wall segment} $\Theta_L(A)$ of the branch $A$ relative to $L$ 
     is a compact oriented segment endowed with the following 
     supplementary structures: 
         \begin{itemize} 
             \item  an increasing homeomorphism $\ex_{L,A} : \Theta_L(A) \to [0, \infty]$, 
                the {\bf exponent function}; 
             \item  marked points, which are by definition the points whose values by the exponent function 
                 are the characteristic exponents of $A$ relative to $L$, as well as the smallest 
                 end of $\Theta_L(A)$, labeled by $L$, and the greatest end, labeled by $A$. 
             \item an {\bf index function}  $\de_{L,A}: \Theta_L(A) \to \N$, which associates 
                 to each point $P \in \Theta_L(A)$ the index of $(\Z, +)$ in the 
                 subgroup of $(\Q, +)$ generated by $1$ and the characteristic exponents 
                 of $A$ which are strictly smaller than $\ex_{L,A}(P)$. 
          \end{itemize}
 \end{definition}
 
 The index $\de_{L, A}(P)$ may be also seen as the smallest common denominator 
 of the exponents of a Newton-Puiseux root of $f_A$ which are strictly less than 
 $\ex_{L,A}(P)$.

Let us consider now the case of a reduced curve with several branches. 
In this case, one may associate it an analog of the Eggers-Wall segment of one 
branch, its \emph{Eggers-Wall tree}. In order to construct this tree, one needs to know not only 
the characteristic exponents of its branches, but also the \emph{orders of coincidence} 
of its pairs of branches:

\begin{definition}  \label{def:ordcoin}
 If $A$ and $B$ are two distinct branches, which are also distinct from $L$, 
 then their {\bf order of coincidence} relative to $L$ is defined by:
      $$k_L(A, B):=\max  \{ \nu_x(\eta_A - \eta_B) \: | \: \eta_A  
         \in  \mathrm{Zer} (f_A),   \: \:    \eta_B  \in \mathrm{Zer} (f_B) \}  \in  \Q_+^* . $$
\end{definition}

Informally speaking, the order of coincidence is the greatest rational number $k$ for which 
one may find Newton-Puiseux roots of the two branches coinciding up to that 
number ($k$ excluded). 

Note that the order of coincidence is symmetric: $k_L(A,B) = k_L(B,A)$, similarly to 
the intersection 
number of the two branches. But, unlike the intersection number, it depends not only on the 
branches $A$ and $B$, but also on the choice of branch $L$. 
Nevertheless, the two numbers are related, 
as explained in Theorem \ref{intcomp} below.

 \begin{definition} \label{def:EW}
      Let $C$ be a reduced germ of curve on $S$. 
      Let us denote by $\mathcal{I}_C$ the set of irreducible components of $C$ 
      which are different from $L$. 
     The {\bf Eggers-Wall tree $\Theta_L(C )$ of $C$ relative to $L$} 
     is the rooted tree obtained 
     as the quotient of the disjoint union of the individual Eggers-Wall segments $\Theta_L(A)$, 
     $A \in \mathcal{I}_C$,  by the following equivalence relation. 
    If $A, B \in \mathcal{I}_C$, then the 
     gluing of $\Theta_L(A)$ with $\Theta_L(B)$ is done along the initial segments 
     $\ex_{L,A}^{-1}[0, k_L(A,B)]$ and $\ex_{L,B}^{-1}[0, k_L(A,B)]$ by:
      \[
      \ex_{L,A}^{-1}(\alpha) \sim \ex_{L,B}^{-1}(\alpha), \: \mbox{ for all } \:  
         \alpha \in [0, k_L(A,B)].
      \]
     One endows $\Theta_L(C)$ with the {\bf exponent function} 
     $\ex_L : \Theta_L(C) \to  [0, \infty]$ and the 
     {\bf index function} $\de_L: \Theta_L(C) \to \N$ obtained by 
     gluing the initial exponent 
     functions $\ex_{L, A}$ and $\de_{L, A}$ respectively, for 
     $A$ varying among the irreducible components of $C$ different from $L$. 
     If $L$ is an irreducible component of $C$, then the tree $\Theta_L(L )$ is 
     the trivial tree with  vertex set  a singleton, whose element is labelled by $L$. 
     The marked point $L \in \Theta_L (L)$ is  identified with the root of $\Theta_L(A)$
      for any $A \in \mathcal{I}_C$. 
  \end{definition}

The fact that in the previous notations $\Theta_L( C), \ex_L, \de_L$ we mentioned 
only the dependency on $L$, and not the whole coordinate system $(x,y)$, 
comes from the following fact (see \cite[Proposition 26]{GBGPPP 17b}):

\begin{proposition}  \label{smoothcv}
 The Eggers-Wall tree  $\Theta_L(C)$, 
 seen as a rooted tree endowed with the exponent function $\ex_L$ and the index  
 function $\de_L$, depends only on the pair $(C, L)$, where $L$ is defined 
 by $x=0$.
\end{proposition}

When $L$ is generic with respect to $C$, the Eggers-Wall tree $\Theta_L (C)$ 
is in fact independent of it (see \cite[Theorem 4.3.8]{W 04}).

\medskip

Note that the index function $\de_L$ is constant on each segment 
$(\mathtt{P}(V) V]$  of $\Theta_L(C)$, where $\mathtt{P}$ denotes the parent map
introduced in Definition \ref{predinit}. Here $V$ denotes any 
vertex of the marked tree $\Theta(C)$ which is different from the root $L$. 
Moreover, the set of marked points is determined by the topological structure 
of $\Theta_L(C)$ and by the knowledge of the index function, as the reader 
may easily verify:

\begin{lemma} \label{encod}
    The set of marked points of the Eggers-Wall tree $\Theta_L(C)$ is the 
    union of the following sets:
      \begin{itemize}
          \item the set of ends, consisting of the root $L$ and the leaves 
               $A \in \mathcal{I}_C \setminus \{ L \}$; 
          \item the set of ramification points;
          \item the set  of points of discontinuity  of the index function. 
       \end{itemize}
\end{lemma}

Any ramification point of $\Theta_L (C)$ is of the form $A \wedge_L B$ for $A, B \in \mathcal{I}_C$. 
Here, the point  $A \wedge_L B$, which has exponent equal to $k_L(A, B)$, 
  is the infimum 
  of the leaves of $\Theta_L(C)$ labeled by $A$ and $B$, 
  relative to the partial order on the set of vertices of $\Theta_L(C)$ 
  defined by the root $L$ (see Definition \ref{finrtree}). 
Note that the first set  in Lemma \ref{encod}
is disjoint from the two other ones, but that the second and the third 
one may have elements in common, as may be seen in Example \ref{EWmany}, 
in which $3$ of the $4$ ramification points are also points of discontinuity of the 
index function.

 \begin{remark} \label{encod2}
     By Lemma \ref{encod}, the Eggers-Wall tree $\Theta_L (C)$ is determined by its finite affine tree
   equipped with the exponent function and the index function (see Definition \ref{combtree}).
 \end{remark}

  \begin{example} \label{EWmany}
      Consider a plane curve singularity $C = \sum_{i=1}^5 C_i$ whose branches 
     $C_i$ are defined by the Newton-Puiseux series $\eta_i$, where:
      \[
       \eta_1 = x^2, \quad 
       \eta_2 =  x^{5/2} + x^{8/3}, \quad 
       \eta_3 = -x^{5/2} + x^{11/4}, \quad 
       \eta_4 = x^{7/2}  +  x^{17/4}, \quad 
       \eta_5 = x^{7/2} + 2 x^{17/4} + x^{14/3}.
       \]   
     We will denote simply $k$ instead of $k_L$, where $L = Z(x)$. 
      One has $k(C_1, C_2) = k(C_1, C_3) = k(C_1, C_4) = k(C_1, C_5) =2$, 
      $k(C_2, C_4) = k(C_2, C_5) = 5/2$, $k(C_2, C_3) = 8/3$, $k(C_3,C_4)=k(C_3,C_5)=5/2$, 
      $k(C_4, C_5) = 17/4$ 
    and the Eggers-Wall tree of $C$ relative to $L$ is drawn in Figure \ref{fig:EWfive}. 
    Observe that $C_3$ admits also as Newton-Puiseux series  
    $\tilde{\eta}_3 :=\eta_3(ix^{1/4})= x^{5/2}  - ix^{11/4}$ and that  
    $k(C_2, C_3) =\nu_x(\eta_2-\tilde{\eta}_3)>\nu_x(\eta_2-\eta_3)$.
 \end{example}

\begin{figure}[h!] 
\vspace*{6mm}
\labellist \small\hair 2pt 
\pinlabel{$L$} at 160 -10
\pinlabel{$C_1$} at 266 185
\pinlabel{$C_2$} at 300 315
\pinlabel{$C_3$} at 200 370
\pinlabel{$C_4$} at 110 370
\pinlabel{$C_5$} at 0 375

\pinlabel{$\mathbf{0}$} at 140 6
\pinlabel{$\mathbf{2}$} at 107 82
\pinlabel{$\mathbf{5/2}$} at 80 120
\pinlabel{$\mathbf{7/2}$} at 44 215
\pinlabel{$\mathbf{17/4}$} at 17 260
\pinlabel{$\mathbf{14/3}$} at -10 310
\pinlabel{$\mathbf{8/3}$} at 164 195
\pinlabel{$\mathbf{11/4}$} at 140 295

\pinlabel{$1$} at 150 50
\pinlabel{$1$} at 124 110
\pinlabel{$1$} at 200 125
\pinlabel{$1$} at 95 175

\pinlabel{$2$} at 132 165
\pinlabel{$2$} at 60 255
\pinlabel{$2$} at 165 245
\pinlabel{$4$} at 42 295
\pinlabel{$4$} at 80 300
\pinlabel{$4$} at 195 315
\pinlabel{$6$} at 220 245
\pinlabel{$12$} at 30 340

\endlabellist 
\centering 
\includegraphics[scale=0.50]{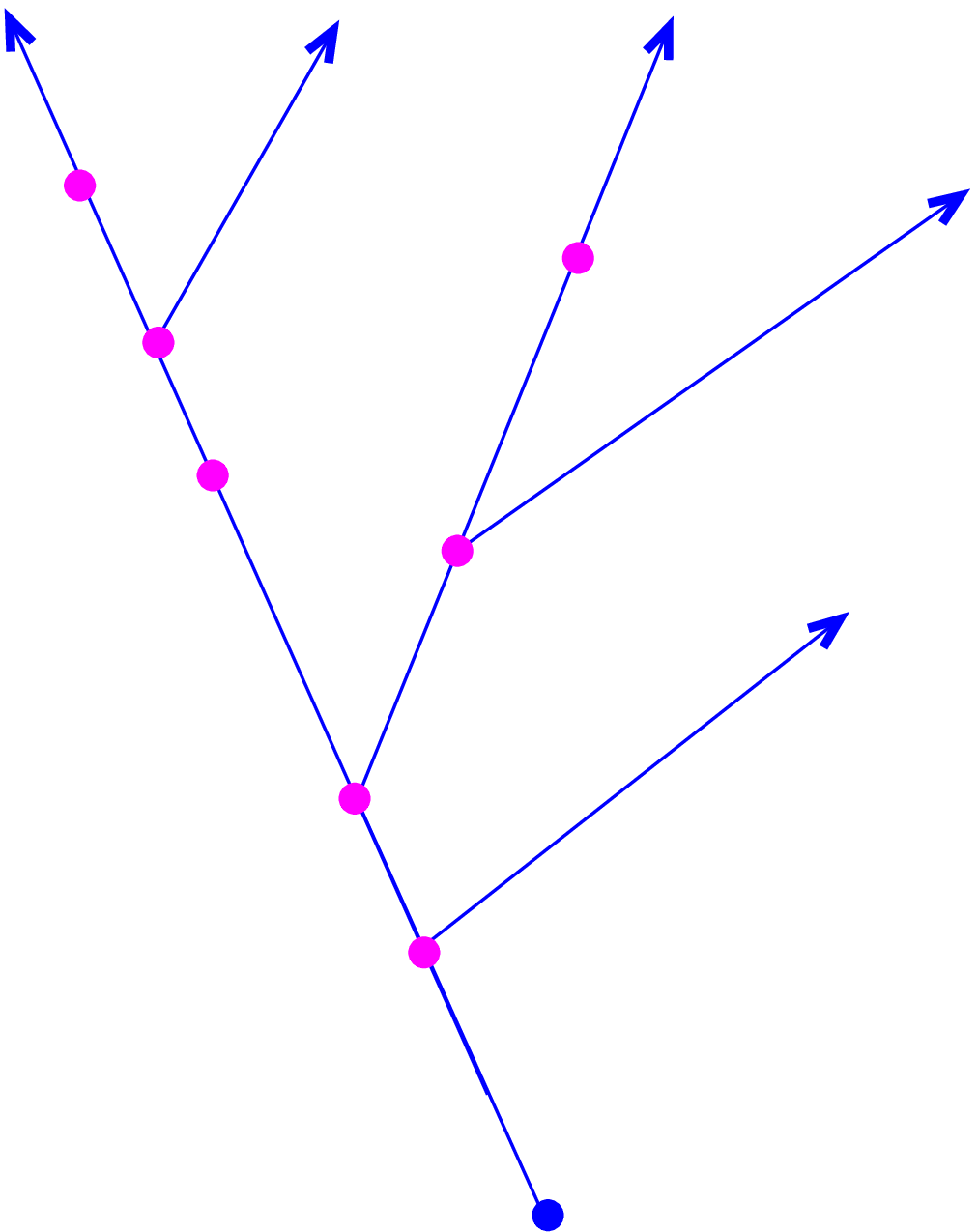} 
\vspace*{5mm} 
\caption{The Eggers-Wall tree of Example \ref{EWmany}} 
\label{fig:EWfive}
\end{figure}

\begin{remark}  \label{biggergerm}
  If one considers two reduced germs $C \subset C'$, then one has a unique 
  embedding $\Theta_L(C) \subset \Theta_L(C')$ such that the restrictions  
  to $\Theta_L(C)$ of the index and of the exponent function on $\Theta_L(C')$ 
  are equal to the corresponding functions on $\Theta_L(C)$. 
\end{remark}

\begin{remark}  \label{histcomm} $\,$
\begin{enumerate} 

\item
Eggers introduced in his 1983 paper  \cite{E 83} about the structure of polar curves 
of a possibly reducible plane curve singularity a slightly different notion of tree. 
Namely, given a reduced germ $C$, he considered only generic coordinate systems $(x,y)$, 
for which $L=Z(x)$ is transversal to all the branches of $C$. In terms 
of our notations, he rooted his tree at the minimal marked point different from the root $L$ 
of the Eggers-Wall tree. He considered only an analog of the exponent function, defined 
on the set of marked points of the tree. Eggers did not consider the index function. 
Instead, he used two colors for the edges of his tree, in order to remember for each 
branch of $C$ which marked points lying on it correspond to its characteristic exponents. 
Our notion of Eggers-Wall tree is based on Wall's 2003 paper \cite{W 03} (which 
circulated as a preprint since 2000), in which the  
functions $\ex_L, \de_L$ (with different notations) are used for computations adapted 
to the description of the polar curves of $C$. The name ``\emph{Eggers-Wall tree'' 
was introduced by the third author in \cite{PP 01}, to honor the previous works of 
Eggers and Wall.}

\item
In previous papers, versions of the notion of 
Eggers-Wall tree of $C$ with respect to the 
local coordinates $(x,y)$ were defined under the assumption that $L$ is not a component of $C$ 
(see \cite{E 83, G 96, GB 00, W 03, PP 01, PP 04, GB-GP 05, C 03, MN 05, GLM 09}). 
Allowing $L$ to be a branch of $C$ permits a very easy formulation of the 
inversion theorem for Eggers-Wall trees (see Theorem \ref{muthm}). 
Note that the third author's paper \cite{PP 04}, which presented some of the 
results of \cite{PP 01}, introduced an extension of the Eggers-Wall trees to quasi-ordinary 
power series in several variables, and applied them to the study of polar hypersurfaces 
of quasi-ordinary hypersurfaces. This study was continued by the first two authors 
in \cite{GB-GP 05}. 

\item Corral used in \cite{C 03} a version of the Eggers-Wall tree to 
describe the topology of a generic polar curve associated with a generalized curve foliation in 
$(\mathbb{C}^2,0)$, with non resonant logarithmic model.
 
\item The Eggers-Wall tree may be seen as a Galois quotient of a variant of the 
    tree constructed in 1977 by Kuo and Lu in \cite{KL 77} (see \cite[Remark 4.39]{GBGPPP 17a}, 
    as well as \cite[Section 2.5]{GLM 09}). 
    This variant is defined exactly as the Eggers-Wall tree, but using \emph{all} the Newton-Puiseux 
    roots of $C$, not only one root for each branch. Therefore, it has as many leaves 
    as the intersection number $(C \cdot L)$. A related construction was performed by Kapranov 
    in his 1993 papers \cite{K 93} and \cite{K 93bis}. He applied it to usual formal power 
    series with complex and real coefficients respectively and he called the resulting rooted 
    trees \emph{Bruhat-Tits trees}.

\end{enumerate}
\end{remark}

Let us introduce a third real-valued function ${\ic}_L$ defined on the Eggers-Wall tree.
It allows us to compute the pairwise intersection 
numbers of the branches of the given germ (see Theorem \ref{intcomp} below). 
It is determined by the knowledge 
of the exponent function ${\ex}_L$ and of the index function ${\de}_L$:

\begin{definition} \label{Intcoef}  
    Let $A$ be a branch  on $S$ with characteristic exponents 
    $\alpha_1<  \dots < \alpha_g$,  relative to the smooth germ $L$.
 We define conventionally $\alpha_0 = 0$ and $\alpha_{g+1} = \infty$.   
 Let us set $P_j =  \ex_L^{-1}(\alpha_j) $ for $j = 0, \dots, g+1$.
 We denote by ${\de}_{j} $ the value of the index function ${\de}_L$ 
 in restriction to the half-open segment  $ ( P_j , P_{j+1}]$. 
   If $P \in \Theta_L(A)$, then there exists $0 \leq l \leq g$ such that  $P\in [P_j, P_{j+1}]$.
 Then, the {\bf contact complexity}  ${\ic}_L(P)$ of the point $P$ is defined by: 
     \[ 
    {\ic}_L(P) :=   \left(
    \sum_{j =1}^{l} 
    \frac{\alpha_j - \alpha_{j-1}}
    {{\de}_{ j-1}}
    \right)  + 
    \frac{{\ex}_L(P)- \alpha_l}{{\de}_{ l}}.
     \]
\end{definition}

\begin{remark} The possibility  $\alpha_l=\alpha_0=0$ is  allowed in Definition \ref{Intcoef}. 
    The previous formula gives the same value to ${\ic}_L(P)$ when 
    ${\ex}_L(P) = \alpha_l$, if we compute it by looking at $\alpha_l$ 
    either as an element of $[ \alpha_{l-1}, \alpha_l]$ or as an element 
    of $[ \alpha_{l}, \alpha_{l+1}]$. 
\end{remark}

Note that the right-hand side of the formula defining ${\ic}_L(P)$ may be reinterpreted 
as an integral of the piecewise constant function $1/ {\de}_L$  along the segment 
$[L,P]$ of $\Theta_L(A)$, the measure being determined by the exponent 
function:

\begin{equation}  \label{intcoefint}
    {\ic}_L(P) = \int_L^{P} \frac{d \: {\ex}_L}{{\de}_L}. 
\end{equation}

\begin{remark}
Notice also that the knowledge of ${\ic}_L$ and ${\de}_L$ determines ${\ex}_L$:
\begin{equation}  \label{expfunctint}
    {\ex}_L(P) = \int_L^{P} {\de}_L \:  d \: {\ic}_L  . 
\end{equation}
Or, written in  a way which is analogous to the developed expression given 
in Definition \ref{Intcoef}, and keeping the notations of that definition:
  \begin{equation} \label{exp-ci}
  {\ex}_L(P) =   \left(\sum_{j =1}^{l} {\de}_{j-1} ({\ic}_j - {\ic}_{j-1})\right)  + 
               {\de}_l ({\ic}_L(P)- {\ic}_l),
  \end{equation}
  where ${\ic}_j := {\ic}_L(P_j)$ for every $j \in \{0, ..., g\}$. 
\end{remark}

\begin{remark}   \label{rem:integform}
   Formulae (\ref{intcoefint})  and (\ref{expfunctint}) are inspired by the 
   formulae (3.7) and (3.9) of Favre and Jonsson's book \cite{FJ 04}, 
   relating \emph{thinness} and \emph{skewness} 
   as functions on the valuative tree. See Section \ref{fundcoord} below.
\end{remark}

As the function $\de_L : \Theta_L(A) \to \N^*$ is increasing along the segment 
$\Theta_L(A)$, formulae (\ref{intcoefint})  and (\ref{expfunctint}) imply:

\begin{corollary} \label{convconc}
    Let $A$ be a branch on $S$ different from $L$.   
    The contact complexity function $\ic_L$ is an increasing homeomorphism from the Eggers-Wall 
    segment $\Theta_L(A)$ to $[0, \infty]$. Moreover, it is piecewise affine and 
    concave  in terms of the parameter $\ex_L$. Conversely, the function $\ex_L$ 
    is continuous piecewise affine and 
    convex in terms of the parameter $\ic_L$. 
\end{corollary}

Let us consider the case of a reduced germ $C$. 
As an easy consequence of Definition \ref{Intcoef}, we get:

\begin{lemma} \label{contint} $\,$ 
  The contact complexity functions of the branches of $C$ glue 
        into a continuous strictly increasing surjection $\ic_L : \Theta_L(C) \to [0, \infty]$.
\end{lemma}

This allows to formulate the following definition:

\begin{definition} \label{Intcoefgen}
    Let $C$ be a reduced germ of curve on the smooth surface $S$. If $L$ is a smooth branch 
     on $S$, then  the {\bf contact complexity
     ${\ic}_L : \Theta_L(C) \to [0, \infty]$ relative to $L$}
    is the function obtained by gluing the contact complexities  
    of the individual branches of $C$. 
\end{definition}

We chose the name of this function motivated by the following theorem, which shows 
that ${\ic}_L$ may be seen as a measure of the contact between the branches of $C$. 
In equivalent formulations, this theorem goes back 
at least to Smith \cite[Section 8]{S 75}, Stolz \cite[Section 9]{S 1879} and Max Noether \cite{N 90}. 
A  proof written in current mathematical language may be found in Wall \cite[Thm. 4.1.6]{W 04}:

\begin{theorem}  \label{intcomp}
    Let $C$ be a reduced germ and $L$ a smooth branch on $S$. 
     Let $C_i$ and $C_j$ be two distinct branches of $C$.     Let $P= \langle L , C_i, C_j \rangle$ be the center of the tripod determined by $L, C_i, C_j$ 
    in the Eggers-Wall  tree $\Theta_L(C)$ (see Definition \ref{trip}). Then:
      \begin{equation} \label{f-intcomp}
      {\ic}_L(P) = \frac{(C_i \cdot C_j)}{(L \cdot C_i) \cdot (L \cdot C_j)}.
      \end{equation}
      
\end{theorem}

Observe that Theorem  \ref{intcomp} also holds when $L$ coincides with $C_i$ or $C_j$ 
(using the convention that $a/\infty = 0$ for every $a \in (0, \infty)$).

\begin{remark}  In the paper \cite{P 85},  P\l oski proved a theorem which 
is equivalent to the fact that the function
\[ 
U_L (C_i, C_j) := \left\{
\begin{array}{lcl}
{\ic}_L( \langle L, C_i, C_j \rangle ) ^{-1}
& \mbox{ if }   &  C_i \ne C_j, 
\\
0  & \mbox{ if }  &  C_i = C_j,
\end{array}
\right.
\]
defines an ultrametric distance on the set of branches which are transversal to $L$. 
See \cite{GBGPPP 17b, GBGPPPR 18} for generalizations of this result to 
all normal surface singularities (in particular, it is proved there that, 
given a normal surface singularity $S$ and an arbitrary branch $L$ 
on it, the function $U_L$ is an ultrametric on the set of branches different from it 
if and only if $S$ is \emph{arborescent}, that is, the dual graphs of its good resolutions
are trees).
\end{remark}
\medskip

Note that the intersection number 
$(L \cdot C_i)$ is equal to  the maximum ${\de}_L (C_i)$ of the index function on the segment 
$[L, C_i]$. We deduce that:

\begin{corollary}  \label{intfromEW} {\bf (Tripod formula)}
   Assume that the Eggers-Wall tree $(\Theta_L(C), {\ex}_L, {\de}_L)$ of the 
   reduced germ $C = \sum_{i \in I} C_i$ relative to $L$ is known. Then the pairwise 
   intersection numbers of its branches are determined by:
     $$  (C_i \cdot C_j) = \de_L(C_i) \cdot  \de_L(C_j) \cdot  {\ic}_L(\langle L, C_i, C_j \rangle). $$
\end{corollary}

The previous equality shows that the intersection number of two branches of $C$ 
is determined by the indices of the two corresponding leaves and by 
the contact complexity of the center of the tripod  
formed by the root of the tree $\Theta_L(C)$ and the two leaves.  
That is why we call it the {\em tripod formula}. 

\medskip 

Corollary \ref{intfromEW} admits an extension for semivaluations 
(see Proposition \ref{tripodform} below).

\begin{example} \label{Exintcoef}
    Consider again the curve singularity of Example \ref{EWmany}. 
    Then the contact complexities of the marked points of its 
    Eggers-Wall tree with respect to the given coordinate system 
    are as indicated in Figure \ref{fig:Intfive}. For instance, the 
    contact function of the highest point on the geodesic 
    going from $L$ to $C_5$ is computed in the following way 
    using Definition \ref{Intcoef}:
      $$  \frac{7}{2}+ \frac{1}{2}\left(\frac{17}{4} - 
               \frac{7}{2}\right) +\frac{1}{4}\left(\frac{14}{3} - \frac{17}{4}\right) =  \frac{191}{48}.$$
    Using Theorem \ref{intcomp}, we deduce that 
    $(C_1 \cdot C_2) = 12,  (C_1 \cdot C_3) = (C_1 \cdot C_4) = 8$, 
    $(C_1 \cdot C_5) = 24$, $(C_2 \cdot C_3) = 62, (C_2 \cdot C_4) = 60, 
    (C_2 \cdot C_5) = 180, (C_3 \cdot C_4) = 40$, 
    $(C_3 \cdot C_5) = 120$, $(C_4 \cdot C_5) = 186.$
\end{example}

\begin{figure}[h!] 
\vspace*{6mm}
\labellist \small\hair 2pt 
\pinlabel{$L$} at 160 -10
\pinlabel{$C_1$} at 266 185
\pinlabel{$C_2$} at 300 315
\pinlabel{$C_3$} at 200 370
\pinlabel{$C_4$} at 110 370
\pinlabel{$C_5$} at 0 375

 \pinlabel{\blue{$\mathbf{0}$}} at 140 6
\pinlabel{\blue{$\mathbf{2}$}} at 111 82
\pinlabel{\blue{$\mathbf{5/2}$}} at 80 120
\pinlabel{\blue{$\mathbf{7/2}$}} at 44 215
\pinlabel{\blue{$\mathbf{31/8}$}} at 17 260
\pinlabel{\blue{$\mathbf{191/48}$}} at -13 310
\pinlabel{\blue{$\mathbf{31/12}$}} at 166 195
\pinlabel{\blue{$\mathbf{21/8}$}} at 141 292

\endlabellist 
\centering 
\includegraphics[scale=0.55]{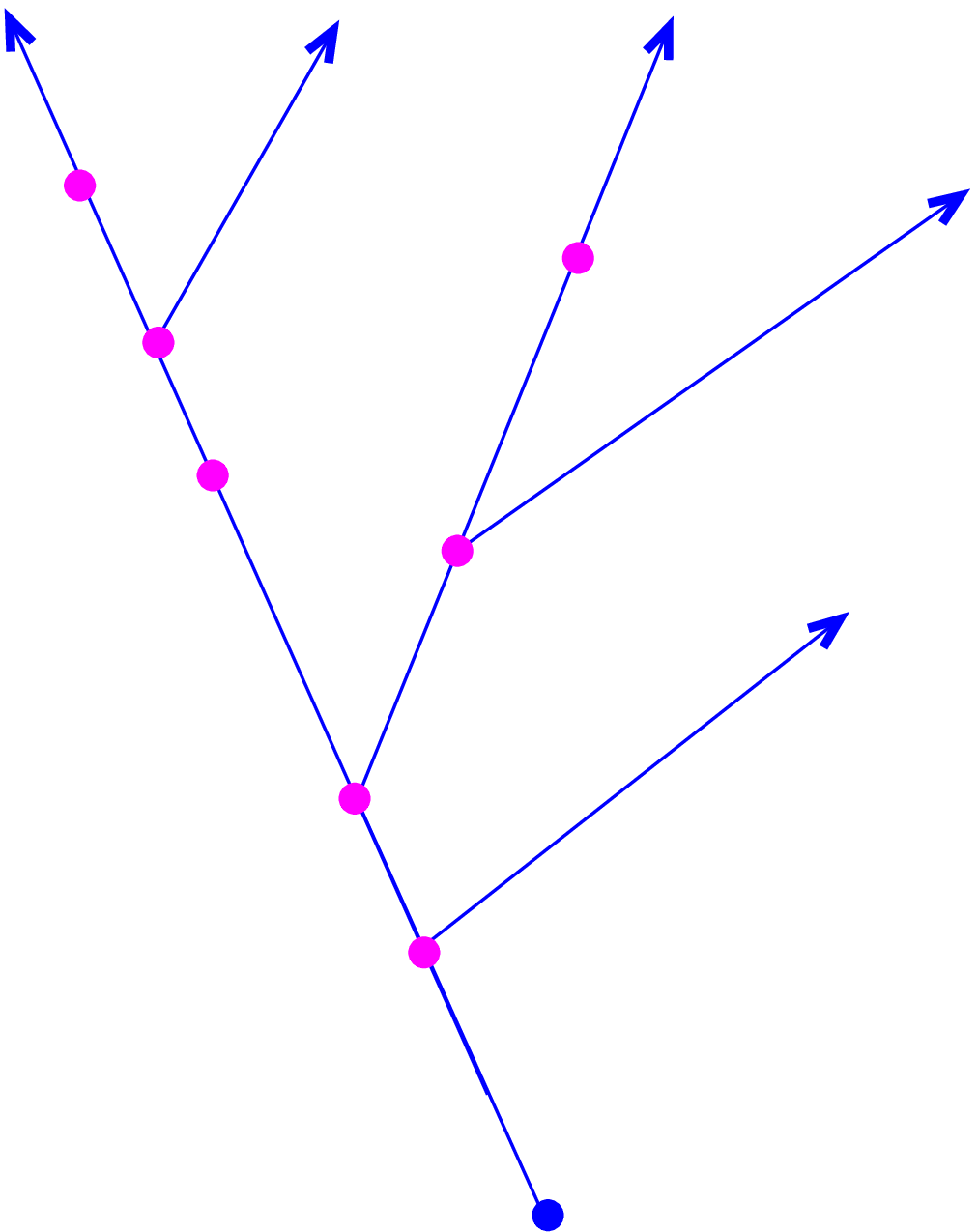} 
\caption{The values of the contact complexity  
    $\ic_L$  at the marked points of the tree of Example \ref{EWmany}} 
\label{fig:Intfive}
\end{figure}

\section{An inversion theorem for Eggers-Wall trees} \label{sec-inv}

Let $C$ be a reduced germ of formal curve on $S$ and let $L$ be a smooth branch. 
Assume that we know the Eggers-Wall tree $\Theta_L(C)$ of $C$ relative to $L$. How to pass to 
the Eggers-Wall tree of $C$ relative to another smooth branch $L'$? The answer 
is particularly simple when \emph{both $L$ and $L'$ are branches 
of $C$}. Indeed, in this case, we prove that the underlying topological space of the 
Eggers-Wall tree is unchanged: one has only to modify the exponent and index functions 
(see Theorem \ref{muthm}). 
This constitutes a geometrization and generalization to the case of several branches 
of the classical inversion theorem of Abhyankar \cite{A 67}, which can be traced back in fact 
to Halphen \cite{H 76} and Stolz \cite{S 1879} 
in the years 1870, as explained in \cite{GBGPPP 17b}.

\medskip 

Before stating our inversion theorem, we need some definitions and properties of the 
Eggers-Wall segments of smooth branches and of their attaching points on Eggers-Wall trees 
of germs not containing them, in the sense of Definition \ref{defatt}.

\begin{definition} \label{atpoint}
  Let $C$ be a reduced germ of formal curve on $S$ and let $L$ be a smooth branch. 
The  \textbf{unit subtree} $\Theta_L(C)_1$  of $\Theta_L (C)$ 
consists of its points of index $1$, equipped with the restriction of the exponent function $\ex_L$. 
The  {\bf unit point} of the tree $\Theta_L(C)$ is the 
 attaching point of a {\em generic} smooth branch through $O$. 
\end{definition}

The unit point is independent of the choice of generic smooth branch through $O$, as it may be  characterized by the following lemma:

\begin{lemma} \label{charcentre}
    The unit point  of $\Theta_L(C)$ is: 
    \begin{itemize} 
        \item  the highest end of $\Theta_L(C)_1$,  when 
             the exponent function takes only values $<1$ in restriction to $\Theta_L(C)_1$ (case 
             in which $\Theta_L(C)_1$ is a segment); 
        
        \item the unique point of $\Theta_L(C)_1$ of exponent $1$, otherwise.         
    \end{itemize}
\end{lemma}

  \begin{proof}
Consider a smooth branch $L'$ transversal both to $L$ and to the branches of $C$. Work then 
     in a coordinate system $(x,y)$ such that $L = Z(x)$ and $L'= Z(y)$. Therefore $L'$ has $0 
      \in  \C[[ x^{1/ \N} ]]$ as only Newton-Puiseux series. Our transversality hypothesis 
     implies that for any branch $A$ of $C$, its Newton-Puiseux series $\eta$ satisfy 
      $\nu_x(\eta) \in (0, 1]$. But one has that $  \nu_x(\eta) = \nu_x(\eta -0) = k_L(A, L')$. This implies 
     immediately our statements. We are in the first case if $\nu_x(\eta) <1$ for all the branches 
     of $C$ and in the second one otherwise.
  \end{proof}

 \begin{example}
      In Figure \ref{fig:unitex} are represented the unit subtree and the unit point $U$ 
      of the Eggers-Wall  tree of Figure \ref{fig:EWfive}. 
 \end{example}

 \begin{figure}[h!] 
 \vspace*{1mm}
 \labellist \small\hair 2pt 
 \pinlabel{$L$} at 160 -10
 \pinlabel{$C_1$} at 266 185
 \pinlabel{$C_2$} at 300 315
 \pinlabel{$C_3$} at 200 370
 \pinlabel{$C_4$} at 110 370
 \pinlabel{$C_5$} at 0 375

 \pinlabel{$\mathbf{1}$} at 125 45
 \pinlabel{$\mathbf{2}$} at 107 82
 \pinlabel{$\mathbf{5/2}$} at 80 120
 \pinlabel{$\mathbf{7/2}$} at 44 215
 \pinlabel{$\mathbf{17/4}$} at 17 260
 \pinlabel{$\mathbf{14/3}$} at -10 310
 \pinlabel{$\mathbf{8/3}$} at 164 195
 \pinlabel{$\mathbf{11/4}$} at 140 295

 \pinlabel{$1$} at 160 30
 \pinlabel{$1$} at 124 110
 \pinlabel{$1$} at 200 125
 \pinlabel{$1$} at 95 175
 \pinlabel{$\mathbf{U}$} at 170 60

 \pinlabel{$2$} at 132 165
 \pinlabel{$2$} at 60 255
 \pinlabel{$2$} at 165 245
 \pinlabel{$4$} at 42 295
 \pinlabel{$4$} at 80 300
 \pinlabel{$4$} at 195 315
 \pinlabel{$6$} at 220 245
 \pinlabel{$12$} at 30 340
 \endlabellist 
 \centering 
 \includegraphics[scale=0.50]{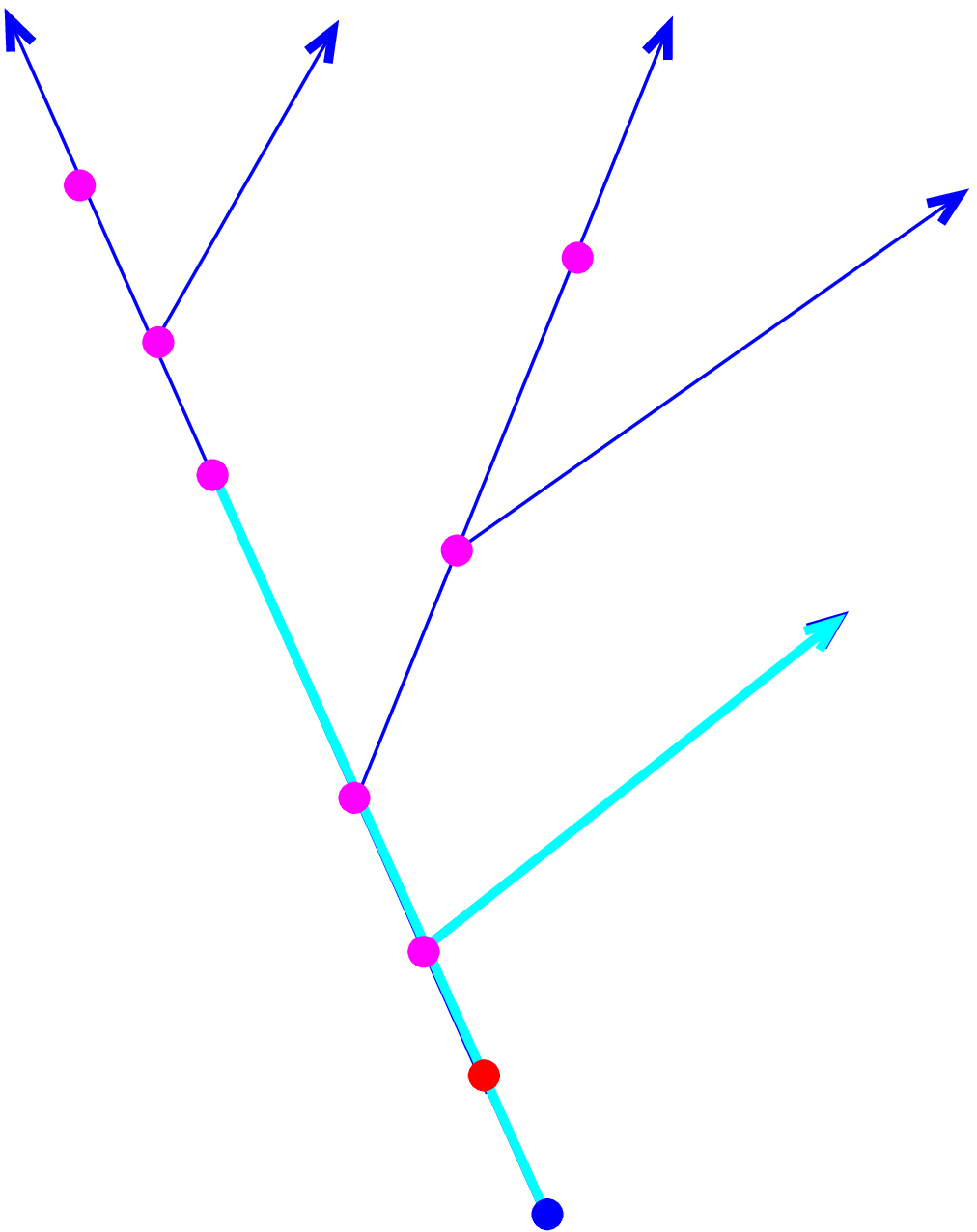} 
 \vspace*{1mm} \caption{The unit subtree (in heavier lines) 
      and the unit point of the Eggers-Wall tree  
    (labelled by $\mathbf{U}$)} 
 \label{fig:unitex} 
 \end{figure}

\begin{figure}[h!] 
\labellist \small\hair 2pt 
\pinlabel{$L$} at 40 4
\pinlabel{$\mathbf{1/n}$} at 40 80
 \pinlabel{$C$} at 40 250
 
  \pinlabel{$1$} at -10 45
  \pinlabel{$n$} at -10 160
\endlabellist 
\centering 
\includegraphics[scale=0.30]{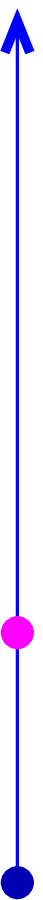} 
\caption{A smooth Eggers-Wall segment with  unit point of exponent $1/n$ (see Definition \ref{smtr}). }
\label{fig:EWsmooth}
\end{figure}

Let us introduce now special names for the Eggers-Wall segments of smooth branches 
with respect to a  smooth branch $L$:

\begin{definition}  \label{smtr} 
      Let $C$ be a branch different from $L$. 
     The Eggers-Wall segment $\Theta_L (C)$ is \textbf{simple} if it has 
     no marked points in its interior. 
    It is called \textbf{smooth}  if  
    it is simple or if it is of the form indicated in Figure \ref{fig:EWsmooth}. 
    In this last case, the integer $n \geq 2$ is equal to the intersection number $(L \cdot C)$.
\end{definition}

The fact that the smooth Eggers-Wall trees are as indicated comes from the fact 
that there exists always a coordinate system $(x,y)$ in which the smooth branch 
$C$ is defined by  $y^n - x =0$ for  $n \geq 1$, while $L = Z(x)$. 
   
\medskip

        By Remark \ref{encod2}, the Eggers-Wall tree $\Theta_L (C)$ is 
        determined by its geometric realization equipped with the exponent function
        ${\ex}_L$ and the index function ${\de}_L$. Notice also that these two functions 
        determine ${\ic}_L$.  
       The following \textit{inversion theorem} proves that these functions determine 
       also the Eggers-Wall tree $\Theta_{L'}(C)$ (recall that $\de_L(L') = (L \cdot L')$): 
       
\begin{theorem} \label{muthm}
     Let $L$ and $L'$ be two smooth branches at $O$ which are components of the 
     reduced germ $C$. Let us denote  by $U$ the unit point of $\Theta_L(C)$ 
     in the sense of Definition \ref{atpoint} and by $\pi_{[L, L']}$ the attaching map 
     of the segment $[L, L']$ in the tree $\Theta_L(C)$ in the sense of Definition \ref{defatt}. 
      Then the finite affine trees associated with  $\Theta_{L'}(C)$ 
       and $\Theta_{L}(C)$ coincide  and the functions 
                 ${\ex}_{L'}$,  ${\ic}_{L'}$, ${\de}_{L'}$ 
     are determined by: 
                 \[  
               \begin{array}{lcl}
               {\ex}_{L'} +1 & =  & \dfrac{{\ex}_L+ 1}{(L \cdot L') \cdot  ({\ic}_L \circ \pi_{[L, L']}) } , 
               \end{array}         
                   \begin{array}{lcl}
                        {\ic}_{L'} & = &  \dfrac{{\ic}_L}{((L \cdot L') \cdot ({\ic}_L \circ \pi_{[L, L']}))^2 } , 
                   \end{array}
                \]
          
                \[
                \begin{array}{lcl}
                   {\de}_{L'}  & =  & \left\{ \begin{array}{ll}
                       1,  & \mbox{  on }  [\pi_{[L, L']}(U), \:  L'] , \; \\
                         \; & \\
                        (L \cdot L') , &  \mbox{  on  } 
                                 [L , \:  \pi_{[L, L']}( U)  ), \\
                                 \; & \\
                       (L \cdot L') \cdot   ({\ic}_L \circ \pi_{[L, L']})  \cdot   {\de}_L, &   \mbox{ otherwise}.
                                                     \end{array} \right. 

               \end{array}
               \]              
         Moreover, in restriction to the segment $[ L, L']$ we have: 
    $$ (L \cdot L') \cdot {\ic}_L = \left\{ \begin{array}{lll} 
                                (L \cdot L')\cdot {\ex}_L & \mbox{ on } & [L, \:  \pi_{[L, L']}( U) ], \\ 
                               {\ex}_L+ 1 - \dfrac{1}{(L \cdot L')}  & \mbox{ on } & [\pi_{[L, L']}( U), \: L'].
                            \end{array}   \right. $$        
 \end{theorem}

 \begin{proof}
       We use here several results developed later in this paper. The idea is to 
       embed the Eggers-Wall tree in the space $\mathbb{P}(\mathcal{V})$ 
       of semivaluations of $S$ and 
       to use formulae about the log-discrepancy, the multiplicity and the self-interaction 
       functions defined on that space. 
       
       Denote, as usual, by $C_i$ the branches of $C$. We will use the 
       valuative embeddings $\Psi_L$ and $\Psi_{L'}$ of Definition \ref{embranch}. 

       By the topological part of Theorem \ref{embcurve}, 
       the images of both embeddings $\Psi_L$ and $\Psi_{L'}$ 
       are the convex hulls of the ends $C_i$ inside the tree $\mathbb{P}(\mathcal{V})$. 
       Therefore, $\Psi_L$ and $\Psi_{L'}$ are homeomorphisms onto those convex hulls. 
    Consequently, the map:
   \begin{equation} \label{crucmap}
        \Psi^L_{L'} : = \Psi_{L'}^{-1} \circ \Psi_L : \Theta_L(C) \to \Theta_{L'}(C)
   \end{equation}
     is a homeomorphism. By construction, it sends each end $C_i$ of $\Theta_L(C)$ 
     to the end with the same label of $\Theta_{L'}(C)$. 

In order to compare $({\ex}_L, {\de}_L, {\ic}_L)$ with $({\ex}_{L'}, {\de}_{L'}, {\ic}_{L'})$, we use  
the part of Theorem \ref{embcurve} concerning the correspondence between functions, 
as well as Propositions \ref{changels}, \ref{changecrd}, \ref{changesm}.  
The statement of our theorem, as well as the one of its Corollary \ref{cormut} are immediate 
consequences of them (the last assertion of the theorem follows from  
Definition \ref{Intcoef}). 
\end{proof}

    Let us particularize this result to the situation where $L$ and $L'$ are \emph{transversal} 
    smooth branches on $S$, that is,  $(L \cdot L') =1$. Then,  the segment $[L, L'] \subset \Theta_L (L')$ 
    is a simple Eggers-Wall segment in the sense of Definition \ref{smtr}, and the unit point  is 
    the point $U$  such that $ {\ex}_L (U) = 1$.   
      
    \begin{corollary} \label{cormut}
    Let $L$ and $L'$ be two transversal smooth branches at $O$ which are components of the 
     reduced germ $C$. We have the relations:      
      \[  
               \begin{array}{lcl}
  		{\ex}_{L'} +1 & =  & \dfrac{{\ex}_L + 1}{{\ex}_L \circ \pi_{[L, L']} }  ,
		\end{array} 
		\quad 
  \begin{array}{lcl}
        {\ic}_{L'} & = &  \dfrac{{\ic}_L}{({\ex}_L\circ \pi_{[L, L']} )^2}. 
    \end{array}
    \]
    and 
      \[
      \begin{array}{lcl}
          {\de}_{L'}  & =  & \left\{ \begin{array}{ll}
                      1,  & \mbox{ \emph{on} }  [L, L'],                   \\
                    ({\ex}_L \circ \pi_{[L, L']})    \cdot 
                    {\de}_L, &   \mbox{ \emph{otherwise}} . \\
         \end{array} \right.  
     \end{array}
     \]  
  \end{corollary}

  \begin{remark}  \label{restransv}
     By combining formula (\ref{intcoefint}) with Corollary \ref{cormut}, we see 
     that in restriction to the segment $[L, L']$, one has the following equalities 
     in the transversal case:
        $$ \ic_{L'} = \ex_{L'} = \ex_L^{-1} = \ic_L^{-1}.$$
  \end{remark}

\begin{figure}[h!] 
\vspace*{6mm}
\labellist \small\hair 2pt 
\pinlabel{$L$} at 160 -10
\pinlabel{$C_1$} at 266 185
\pinlabel{$C_2$} at 300 315
\pinlabel{$C_3$} at 200 370
\pinlabel{$C_4$} at 110 370
\pinlabel{$C_5$} at 0 375

\pinlabel{$\mathbf{0}$} at 145 5

\pinlabel{$\mathbf{2}$} at 107 72
\pinlabel{$\mathbf{5/2}$} at 80 115
\pinlabel{$\mathbf{7/2}$} at 44 210
\pinlabel{$\mathbf{17/4}$} at 17 255
\pinlabel{$\mathbf{14/3}$} at -10 305
\pinlabel{$\mathbf{8/3}$} at 164 190
\pinlabel{$\mathbf{11/4}$} at 140 290

\pinlabel{$1$} at 160 30
\pinlabel{$1$} at 145 65
\pinlabel{$1$} at 124 105
\pinlabel{$1$} at 200 120
\pinlabel{$1$} at 95 170

\pinlabel{$2$} at 132 160
\pinlabel{$2$} at 60 250
\pinlabel{$2$} at 165 240
\pinlabel{$4$} at 42 290
\pinlabel{$4$} at 80 295
\pinlabel{$4$} at 195 310
\pinlabel{$6$} at 220 240
\pinlabel{$12$} at 30 335

\pinlabel{$\mathbf{\infty}$} at 266 165
\pinlabel{$\mathbf{\infty}$} at 300 295
\pinlabel{$\mathbf{\infty}$} at 220 350
\pinlabel{$\mathbf{\infty}$} at 120 350
\pinlabel{$\mathbf{\infty}$} at -20 350

\pinlabel{$\mathbf{1}$} at 125 40
\pinlabel{$U$} at 160 50

\pinlabel{$L$} at 540 -10
\pinlabel{$C_1$} at 643 185
\pinlabel{$C_2$} at 677 315
\pinlabel{$C_3$} at 577 370
\pinlabel{$C_4$} at 487 370
\pinlabel{$C_5$} at 377 375

\pinlabel{$\mathbf{0}$} at 635 165

\pinlabel{$\mathbf{1/2}$} at 484 72
\pinlabel{$\mathbf{3/4}$} at 457 115
\pinlabel{$\mathbf{5/4}$} at 421 210
\pinlabel{$\mathbf{13/8}$} at 394 255
\pinlabel{$\mathbf{11/6}$} at 367 305
\pinlabel{$\mathbf{5/6}$} at 541 190
\pinlabel{$\mathbf{7/8}$} at 517 290

\pinlabel{$1$} at 537 30
\pinlabel{$1$} at 525 65
\pinlabel{$2$} at 501 105
\pinlabel{$1$} at 577 120
\pinlabel{$2$} at 472 170

\pinlabel{$4$} at 509 160
\pinlabel{$4$} at 440 247
\pinlabel{$4$} at 542 240
\pinlabel{$8$} at 419 290
\pinlabel{$8$} at 460 295
\pinlabel{$8$} at 572 310
\pinlabel{$12$} at 597 240
\pinlabel{$24$} at 407 335

\pinlabel{$\mathbf{\infty}$} at 520 5
\pinlabel{$\mathbf{\infty}$} at 680 295
\pinlabel{$\mathbf{\infty}$} at 600 350
\pinlabel{$\mathbf{\infty}$} at 500 350
\pinlabel{$\mathbf{\infty}$} at 360 350

\pinlabel{$\mathbf{1}$} at 502 40
\pinlabel{$U$} at 537 50

\endlabellist 
\centering 
\includegraphics[scale=0.60]{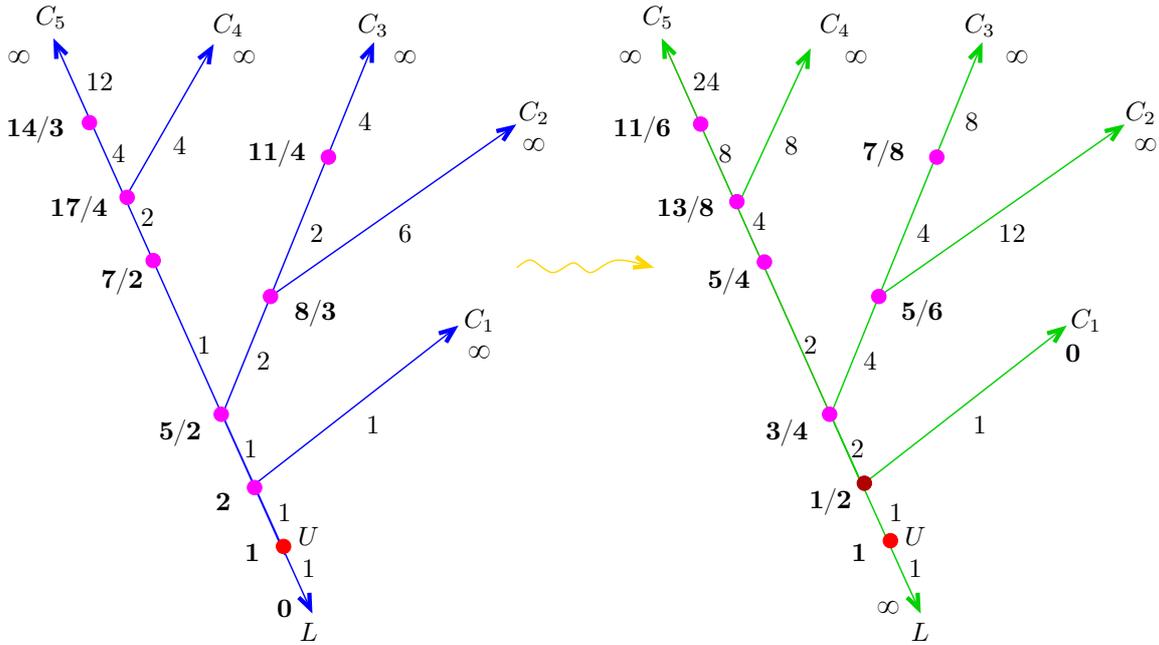} 
\vspace*{4mm}
\caption{The Eggers-Wall trees $\Theta_L (C)$ of Example \ref{Chroot} on the left,   compared with 
$\Theta_{C_1} (C)$ on the right.}
\label{fig:Chrootfive}
\end{figure}

\begin{remark}
    When applied to the case when $C$ is a branch,  
    Corollary \ref{cormut} is a reformulation in terms of Eggers-Wall trees of the 
    \emph{inversion formulae} of Abhyankar \cite{A 67} and Zariski \cite{Z 68}, 
    which express the characteristic 
   exponents with respect to a coordinate system $(y,x)$ in terms of those with 
   respect to $(x,y)$ (see \cite[Introduction]{GBGPPP 17a} for more details 
   about these formulae, and about their discovery and first proof by 
   Halphen \cite{H 76} and Stolz \cite{S 1879}). 
   One may try to prove this corollary directly from the Halphen-Stolz-Abhyankar-Zariski
   inversion formulae applied to the branches $A$ of $C$ different from $L$ and $L'$. 
   This provides the functions $\ex_{L', A}$ and $\de_{L', A}$ and $\ic_{L',A}$.
   Corollary \ref{intfromEW} allows to determine the value of $\ic_{L'}$ on 
   the points $ \langle L', A, B \rangle$,  for $A$ and $ B$ two distinct branches of $C$  
   (different from $L'$).  Formula  (\ref{exp-ci}) 
    determines the value of the exponent function $\ex_{L'}$ at the point 
    $ \langle L', A, B \rangle$, which is equal to $k_{L'} (A,B)$. 
   Then, it remains to prove that the geometric realizations of the trees 
   $\Theta_L (C)$ and $\Theta_{L'} (C)$ are isomorphic by an isomorphism 
   respecting the labellings of the ends by the branches of $C$. 
   Our approach, using the embeddings of the Eggers-Wall trees in the space of 
   valuations, provides a conceptual understanding of these combinatorial operations. 
\end{remark}

\begin{example} \label{Chroot}
    Consider again the Eggers-Wall tree of Example \ref{EWmany}. Now we assume 
    that $L$ is a component of $C$. We represent this in the left diagram of 
    Figure \ref{fig:Chrootfive} by 
     adding an arrow-head at the root. The Eggers-Wall tree $\Theta_{C_1}(C)$ 
    is represented in the right diagram of Figure \ref{fig:Chrootfive}. 
    In each one of the two diagrams, we have also 
   indicated the position of the unit point $U$ (which remains unchanged). 
   The roots may be recognized  as the only ends with vanishing exponent. 
   In our case, $L$ and $L':= C_1$ are transversal, which means that we may apply 
   Corollary \ref{cormut}. This implies that  ${\ex}_{L'} = \dfrac{1}{2}({\ex}_{L} -1)$ 
   on the union of the segments $[C_i, C_j]$, for $i, j \geq 1$, 
   since in restriction to them ${\ex}_L \circ \pi_{[L ,L']}=2$. 
\end{example}

\begin{remark}  \label{rem:inv}
When $L$ or $L'$ is not a branch of $C$, we determine the Eggers-Wall tree 
$\Theta_{L'}(C)$ from $\Theta_L(C)$ by constructing first $\Theta_L(C + L + L')$, 
by applying then Theorem \ref{muthm} to it in order to get $\Theta_{L'}(C + L + L')$, and 
by passing finally to the subtree $\Theta_{L'}(C)$. 
\end{remark}

\section{Eggers-Wall trees and splice diagrams} \label{splicediag}

In this section we recall from Eisenbud and Neumann's book \cite{EN 85} 
the topological operation of \emph{splicing} of two oriented links along 
a pair of their components inside \emph{oriented} integral homology spheres, 
as well as the associated encoding of \emph{graph links} by \emph{splice diagrams}. 
Then we particularize this construction to the links of curve singularities 
inside smooth complex surfaces and we explain 
how to pass from an Eggers-Wall diagram to a splice diagram (see Theorem \ref{EWsplice}). 
\medskip

A {\bf link} in a $3$-dimensional manifold is a closed $1$-dimensional 
submanifold. The link is called a {\bf knot} if it is moreover connected. 
The {\bf exterior} of a link is the complement of the interior of a 
compact tubular neighborhood of it in the ambient $3$-dimensional manifold.

In this section all ambient  $3$-dimensional manifolds and all the links considered inside them 
will be considered to be \emph{oriented}. For this reason, we will not mention this 
hypothesis anymore. 

\begin{definition} \label{ZHS} 
   An {\bf integral homology sphere} is a closed $3$-dimensional 
   manifold $\Sigma$ which has the same integral homology groups as 
   the $3$-dimensional sphere $\bS^3$. 
   Equivalently, it is connected and $H_1(\Sigma, \Z) =0$. 
\end{definition}

If $K_1$ and $K_2$ are two disjoint knots in an 
integral homology sphere $\Sigma$, then we denote by 
$lk_{\Sigma}(K_1, K_2) \in \Z$  their {\bf linking number}. Recall that:  
     $$lk_{\Sigma}(K_1, K_2) = lk_{\Sigma}(K_2, K_1).$$

\begin{definition}
Let $K$ be a knot inside a   
    $3$-dimensional integral homology sphere $\Sigma$. Denote by $U$ a compact 
    tubular neighborhood of $K$  and by $T$ 
    its boundary, which is a $2$-dimensional torus. A {\bf meridian} 
    of $K$ is an oriented simple closed curve $M$ on $T$ 
    which is non-trivial homologically in $T$ but becomes trivial in $U$, and 
    satisfies $lk_{M}(K, M) = 1$. 
    A {\bf longitude} of $K$ is an oriented simple closed curve $L$ on $T$ 
    which is homologous to $K$  in $U$ and 
    satisfies  $lk_{M}(K, L) = 0$. 
 \end{definition}
 
 Note that the constraint that $L$ be homologous to $K$ inside the solid torus $U$ 
 determines its orientation. The condition that $lk_{M}(K, L) = 0$ means intuitively that 
 $L$ does not spiral around $K$, seen from the global viewpoint of $M$. 
 A basic result of $3$-dimensional topology is that meridians and longitudes are 
 well-defined up to isotopy on $T$. 
 
 The following topological construction was described by Eisenbud and Neumann 
\cite[Chapter I.1]{EN 85}, inspired by previous work of Siebenmann \cite{S 79} 
and Bonahon and Siebenmann 
(by $ U^{\circ}$ we denote the \emph{interior} of the manifold with boundary $U$):

\begin{definition} \label{splicing}
    Let $\Lambda_1$ and $\Lambda_2$ be two links 
    inside the disjoint $3$-dimensional integral homology spheres $\Sigma_1$ and $\Sigma_2$ 
    respectively. 
    Let $K_j$ be a connected component of $\Lambda_j$, for each $j \in \{1, 2 \}$. 
    Denote by $U_j$ a compact tubular neighborhood of $K_j$,  
    disjoint from $\Lambda_j \:  \setminus \:  K_j$. We consider longitudes and 
    meridians of $K_j$ on the boundary $T_j$ of $U_j$. 
    The {\bf splice}  of $(\Sigma_1, \Lambda_1)$ and 
    $(\Sigma_2, \Lambda_2)$ along 
    $K_1$ and $K_2$ is the pair $(\Sigma, \Lambda)$ defined by: 
      \begin{itemize} 
          \item $\Sigma$ is the closed $3$-manifold obtained 
               from $\Sigma_1 \:  \setminus \:  U_1^{\circ}$ and $\Sigma_2 \:  \setminus \:  U_2^{\circ}$ by 
                identifying their boundaries $T_1$ and $T_2$ through a 
                diffeomorphism which permutes (oriented) meridians and longitudes. 
          \item $\Lambda$ is the link inside $\Sigma$ obtained by taking the union of the images of 
                 $\Lambda_1 \:  \setminus \:  K_1$ and $\Lambda_2 \: \setminus \: K_2$ inside $\Sigma$.
      \end{itemize}
\end{definition}

The basic result about this operation is (see \cite[Chapter I.1]{EN 85}):

\begin{proposition}
   The link $(\Sigma, \Lambda)$ is well-defined up to 
   an orientation-preserving  diffeomorphism which is unique up to 
   isotopy and $\Sigma$ is again an 
   integral homology sphere. 
\end{proposition}

Conversely, one may {\bf unsplice} an oriented link $(\Sigma, \Lambda)$  
inside an integral homology sphere $\Sigma$ 
by finding inside $\Sigma\:  \setminus \:  \Lambda$ an embedded $2$-torus $T$, 
then cutting $\Sigma$ along $T$ and filling the resulting two manifolds with boundary by 
solid tori in such a way as to get again integral homology spheres. Inside those 
two resulting homology spheres, one considers 
the links which are obtained from $\Lambda$ by adding central circles of the two solid 
tori used for performing the two fillings. 
Remark that the whole process is possible because the complement  
$\Sigma \:  \setminus \: T$ is disconnected,  
as a consequence of the hypothesis that $\Sigma$ is an integral homology sphere: 
otherwise, there would exist a simple closed curve intersecting transversely $T$ 
at one point, which would imply that this curve is not homologous to $0$ in $\Sigma$. 

One has the following result (see \cite[page 25]{EN 85}):

\begin{lemma} \label{gluest}
    Let $(\Sigma, \Lambda)$ be a link inside an integral 
    homology sphere and let $T$ be a $2$-torus inside $\Sigma \: \setminus \: \Lambda$. 
    Then $\Lambda$ is the result of a splicing operation along this torus, of two 
    links $(\Sigma_1, \Lambda_1)$ and $(\Sigma_2, \Lambda_2)$. 
    If $K_i$ denotes the component of $\Lambda_i$ along which this operation 
    is done, then the orientations of $K_1$ and $K_2$ are well-determined up to 
    a simultaneous reorientation. Moreover, if $\Sigma  \simeq \bS^3$, 
    then $\Sigma_1 \simeq \bS^3$, $\Sigma_2 \simeq \bS^3$ and the converse also holds.
\end{lemma}

In the sequel we will use integral homology spheres which are Seifert 
fibred and Seifert links inside them as building blocks in 
the splicing procedure. Let us start by defining the first notion 
(see Orlik's book \cite{O 72}): 

\begin{definition}  \label{def:seif}
   A {\bf Seifert fibration} on a compact $3$-manifold is a smooth foliation by circles, 
   such that each leaf has a saturated neighborhood (that is, a neighborhood obtained as a union 
   of fibres) which is diffeomorphic by a 
   leaf-preserving diffeomorphism to the quotient of the infinite cylinder 
   $\mathbb{D}^2 \times \R$ by the diffeomorphism:
     $$ (z,t) \to (e^{2i \pi q/p} z, t+1),$$
   where:
      \begin{itemize} 
           \item  $q$ and $p$ are coprime integers, with $p \in \N^*$; 
           \item the quotient is endowed with the projection of the foliation of 
                 $\mathbb{D}^2 \times \R$ by the translates of the second factor; 
           \item the initial leaf corresponds to the image of $0 \times \R$ 
                  by this quotient map. 
       \end{itemize} 
   When $p \geq 2$, one says that the initial leaf is {\bf singular} and that $p$ is 
   its {\bf multiplicity}. A saturated neighborhood of the previous kind is called a 
   {\bf model neighborhood}. The leaves of the foliation are called its {\bf fibers}.
 \end{definition}
 
Let us recall a homological interpretation of the multiplicity $p$ 
associated to a fiber $F_0$ of a Seifert 
fibration. Consider a model neighborhood $U$ of the chosen fiber. Orient all the fibers of this model  
in a continuous manner.  If $F$ is a fiber contained inside $U$ and different from $F_0$, then 
$H_1(U, \Z) = \Z [F_0]$, where $[F_0]$ denotes the homology class of $F_0$. 
Moreover, the homology class $[F]$ of $F$ in $H_1(U, \Z)$ is equal to $p [F_0]$. Note that 
this shows that $p$ is independent on the chosen orientations of $F_0$ and of the ambient 
manifold.

In order to get also the number $q$, one has to consider a meridian disk $D$ of $U$, 
whose boundary circle intersects transversally the foliation induced on the $2$-torus $\partial U$. 
Orient $D$ such that its orientation followed by the orientation of a fiber lying in $U$ gives 
the ambient orientation. This induces an orientation on $\partial D$. 
Consider a fiber $F$ lying on $\partial U$. It intersects $\partial D$ in $p$ points. Their set may be 
cyclically ordered by the orientation of $\partial D$, which allows to identify it canonically with 
the cyclic group $\Z/p \Z$. The first return map obtained by following $F$ along 
its chosen orientation is a translation of this group by one of its elements, which is precisely 
the image of $q$ in $\Z/p \Z$. This shows that $q$ is only well-defined modulo $p$ and that 
it is changed into its opposite when one changes the ambient orientation.

Having defined Seifert fibrations, we may define \emph{Seifert links} and the more general 
notion of \emph{graph links}:  

\begin{definition} \label{Seifert link}
  A {\bf Seifert link}  is a link whose 
  exterior admits a Seifert fibration. A {\bf graph link} is a link 
  whose exterior may be cut into Seifert fibred manifolds using a finite set 
  of pairwise disjoint tori. 
\end{definition}

The structure of any graph link inside an integral homology sphere 
may be expressed using a \emph{splice diagram}. This is a special kind 
of decorated tree: 

\begin{definition} \label{defspdiag}
    A {\bf splice diagram}
    is a marked finite forest (that is, a finite disjoint union of trees) 
    whose vertices are decorated with 
    the signs $\pm$ and whose germs of edges at each {\bf internal vertex} 
    (that is, a vertex which is not an end)  
    are decorated with \emph{pairwise coprime} integers. Some of its ends are 
    distinguished as {\bf arrowhead ends}. 
\end{definition}

Each splice diagram encodes up to orientation-preserving homeomorphisms 
a unique graph link inside an 
integral homology sphere. In order to understand this, we explain 
it first in the case in which the diagram is {\bf star-shaped}, that is, 
in which it has exactly one vertex which is not an end. Then the encoding 
is based on the following proposition (see \cite[Chapter II.7]{EN 85}):

\begin{proposition}  \label{basicihs}
     Let $n \geq 2$ and $\alpha_1, \dots , \alpha_n$ be $n$ pairwise 
     coprime  non-zero integers. 
    There exists a unique Seifert fibered oriented integral homology sphere 
    $\Sigma(\alpha_1, \dots, \alpha_n)$ endowed with an oriented link 
    $\Lambda := F_1 \cup \dots \cup F_n $ consisting of oriented fibers and with a choice 
    of continuous orientation of the fibers not belonging to $\Lambda$, such that:
      \begin{itemize}
           \item the link $\Lambda$ contains all singular fibers of the Seifert fibration; 
           \item for every $i \in \{1, \dots , n\}$, the multiplicity of $F_i$ is equal to $| \alpha_i |$; 
           \item the orientation of the generic fibers is chosen compatibly with the orientation 
              of $F_i$ if and only if $\alpha_i > 0$;
           \item for every distinct $i, j \in \{1, \dots, j\}$, 
               the linking number $lk_{\Sigma(\alpha_1, \dots, \alpha_n)}(F_i, F_j)$ 
              is equal to  the product: 
                   \[ \prod_{k \in \{1, \dots, n\} \: \setminus \:  \{i, j \}} \alpha_k.  \] 
       \end{itemize}
\end{proposition}

In fact, the previous definition may be extended to the situation where one of the 
integers $\alpha_i$ is $0$ (note that the coprimality condition prohibits having two 
of them vanishing simultaneously). In order to do this, one must allow still another 
kind of model neighborhood, in which the nearby fibers turn once around the central 
fiber (see \cite[Lemma 7.1]{EN 85}). The resulting manifold is still an integral homology 
sphere, but it is Seifert fibered only in the exterior of the link $\Lambda$. This 
explains the mention of such exteriors of links in Definition \ref{Seifert link}.

   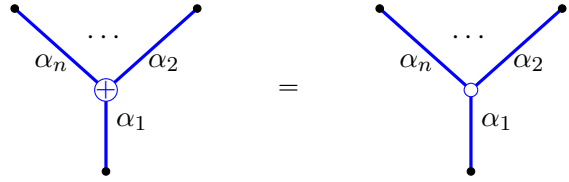
\begin{figure}
\centering
\begin{tikzpicture} [scale=0.6]
 \draw [color=blue, very thick](0,0) -- (0,-1.8) ;  
 \draw [color=blue, very thick] (0,0) -- (2,1.8);
 \draw [color=blue, very thick] (0,0) -- (-2,1.8) ;
 
 \draw (0.7, 0.6) node[right]{$\alpha_2$} ;
\draw (-0.6,0.6) node[left]{$\alpha_n$} ;
\draw (0,-0.7) node[right]{$\alpha_1$} ;
\draw (0, 0.8) node[above]{$\cdots$} ;

 \node[draw,circle,color=blue,inner sep=3pt,fill=white] at (0,0) {};
  \draw [color=blue] (0, 0) node{$+$} ;
 \node[draw,circle,inner sep=1pt,fill=black] at (2,1.8) {};
\node[draw,circle,inner sep=1pt,fill=black] at (0,-1.8) {};
\node[draw,circle,inner sep=1pt,fill=black] at (-2,1.8) {};

  \draw (4, 0) node{$=$} ;

\draw [color=blue, very thick] (8,0) -- (8,-1.8) ;  
 \draw [color=blue, very thick] (8,0) -- (10,1.8);
 \draw [color=blue, very thick] (8,0) -- (6,1.8) ;
 
 \draw (8.7, 0.6) node[right]{$\alpha_2$} ;
\draw (7.4,0.6) node[left]{$\alpha_n$} ;
\draw (8,-0.7) node[right]{$\alpha_1$} ;
\draw (8, 0.8) node[above]{$\cdots$} ;

 \node[draw,circle,color=blue,inner sep=1.8pt,fill=white] at (8,0) {};
 \node[draw,circle,inner sep=1pt,fill=black] at (10,1.8) {};
\node[draw,circle,inner sep=1pt,fill=black] at (8,-1.8) {};
\node[draw,circle,inner sep=1pt,fill=black] at (6,1.8) {};

    \end{tikzpicture}  
    \caption{The splice diagram of the oriented homology sphere 
        $\Sigma(\alpha_1, \dots, \alpha_n)$.}  
    \label{fig:Starshaped}
    \end{figure}

The Seifert fibered oriented homology sphere $\Sigma(\alpha_1, \dots, \alpha_n)$ may be 
represented by any of the two star-shaped diagrams of Figure \ref{fig:Starshaped}. The 
one on the left specifies the sign attributed to the central node, while that on the right does 
not mention any sign. This is a general rule:

\begin{remark}
  If the internal vertices of a splice diagram do not carry signs, this means 
    by convention that 
      they represent oriented Seifert-fibred homology spheres of the type 
       $\Sigma(\alpha_1, \dots,  \alpha_n)$ (see Proposition \ref{basicihs}). 
  
   If one replaces the $(+)$-sign in the diagram on the left of Figure \ref{fig:Starshaped} by 
    a $(-)$-sign, then one obtains by definition a representation of the oppositely oriented 
     manifold to $\Sigma(\alpha_1, \dots,  \alpha_n)$. Let us denote it simply by 
       $- \Sigma(\alpha_1, \dots, \alpha_n)$. 
\end{remark}

Each end of the splice diagrams of the oriented integral homology spheres 
 $\pm \Sigma(\alpha_1, \dots \alpha_n)$ represents by construction an oriented 
 knot in the corresponding manifold. Given two such knots  
 $(\epsilon_1 \Sigma(\alpha_1, \dots, \alpha_n), K_1)$ and 
 $(\epsilon_2 \Sigma(\alpha_1, \dots, \alpha_n), K_2)$ (where $\epsilon_i$ is a  sign and $K_i$ 
 is a knot corresponding to an end of the splice diagram of 
 $\epsilon_i \Sigma(\alpha_i, \dots, \alpha_n)$, then one may splice them as explained 
 in Definition \ref{splicing}. Graphically, one represents this operation by joining the corresponding 
 edges of the two diagrams. An example is shown in Figure \ref{fig:Splicestar}.

 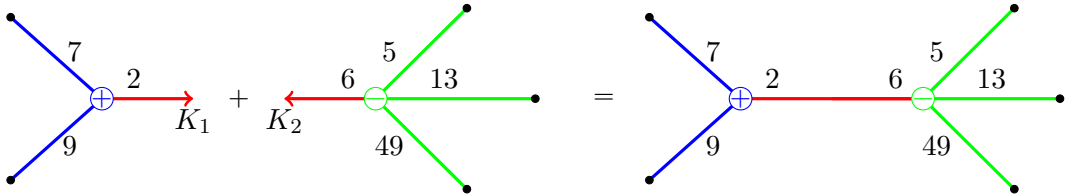
\begin{figure}
\centering
\begin{tikzpicture} [scale=0.6]
 \draw  [color=red, very thick] [->] (0,0) -- (2, 0) ;  
 \draw [color=blue, very thick] (0,0) -- (-2,1.8);
 \draw [color=blue, very thick] (0,0) -- (-2,-1.8) ;
 
 \draw (0.7, 0) node[above]{$2$} ;
\draw (-0.6,0.6) node[above]{$7$} ;
\draw (-0.7,-0.6) node[below]{$9$} ;
\draw (2, 0) node[below]{$K_1$} ;

 \node[draw,circle,color=blue,inner sep=3pt,fill=white] at (0,0) {};
      \draw [color=blue] (0, 0) node{$+$} ;
\node[draw,circle,inner sep=1pt,fill=black] at (-2,1.8) {};
\node[draw,circle,inner sep=1pt,fill=black] at (-2,-1.8) {};

   \draw (3, 0) node{$+$} ;
      
 \draw   [color=red, very thick] [->] (6,0) -- (4, 0) ;  
 \draw [color=green, very thick] (6,0) -- (8,2);
  \draw [color=green, very thick](6,0) -- (9.5, 0) ;
 \draw[color=green, very thick](6,0) -- (8,-2) ;

 \draw (5.4, 0) node[above]{$6$} ;
\draw (6.3,0.6) node[above]{$5$} ;
\draw (6.3,-0.6) node[below]{$49$} ;
\draw (7.5, 0) node[above]{$13$} ;
\draw (4, 0) node[below]{$K_2$} ;

\node[draw,circle,color=green,inner sep=3pt,fill=white] at (6,0) {};
      \draw [color=green](6, 0) node{$-$} ;
\node[draw,circle,inner sep=1pt,fill=black] at (8,2) {};
\node[draw,circle,inner sep=1pt,fill=black] at (9.5,0) {};
\node[draw,circle,inner sep=1pt,fill=black] at (8,-2) {};  
   
  \draw (11, 0) node{$=$} ;

  \draw   [color=red, very thick](14,0) -- (18, 0) ;  
 \draw [color=blue, very thick](14,0) -- (12,1.8);
 \draw [color=blue, very thick](14,0) -- (12,-1.8) ;

  \draw (14.7, 0) node[above]{$2$} ;
\draw (13.4,0.6) node[above]{$7$} ;
\draw (13.4,-0.6) node[below]{$9$} ;

 \node[draw,circle,color=blue,inner sep=3pt,fill=white] at (14,0) {};
      \draw [color=blue](14, 0) node{$+$} ;
\node[draw,circle,inner sep=1pt,fill=black] at (12,1.8) {};
\node[draw,circle,inner sep=1pt,fill=black] at (12,-1.8) {};
  
 \draw    [color=red, very thick](18,0) -- (16, 0) ;  
 \draw[color=green, very thick](18,0) -- (20,2);
  \draw[color=green, very thick](18,0) -- (21, 0) ;
 \draw[color=green, very thick](18,0) -- (20,-2) ;

 \draw (17.4, 0) node[above]{$6$} ;
\draw (18.3,0.6) node[above]{$5$} ;
\draw (18.3,-0.6) node[below]{$49$} ;
\draw (19.5, 0) node[above]{$13$} ;

\node[draw,circle,color=green,inner sep=3pt,fill=white] at (18,0) {};
      \draw [color=green](18, 0) node{$-$} ;
\node[draw,circle,inner sep=1pt,fill=black] at (20,2) {};
\node[draw,circle,inner sep=1pt,fill=black] at (21,0) {};
\node[draw,circle,inner sep=1pt,fill=black] at (20,-2) {};    

    \end{tikzpicture}  
    \caption{Splicing two star-shaped diagrams along the knots $K_1$ and $K_2$.}  
    \label{fig:Splicestar}
    \end{figure}

It is now easy to understand which integral homology sphere corresponds to a given 
  connected splice diagram. Indeed, it is enough to imagine it obtained by successive joining of simpler 
 diagrams along edges adjacent to ends. Then one performs the corresponding 
 splicing operations, taking into account the fact that the end vertices of a splice diagram represent 
 particular oriented knots in the corresponding oriented homology sphere. If one wants to 
 encode not only a manifold, but also a link inside it, then one marks some of the ends 
of the splice diagram as arrowheads.
 
 If the splice diagram is not connected, then by definition it encodes the connected 
 sum of the links corresponding to its connected components.

A given graph link in an integral homology sphere is representable by an infinite number of diagrams. 
Among them, one may define the following preferred ones (see \cite[Page 72]{EN 85}):

\begin{definition}  \label{mindiag}
    A splice diagram is called {\bf minimal} if it minimizes the number of edges 
    among the splice diagrams representing a given graph link. 
\end{definition}

A minimal splice diagram is unique for a given graph link with all fibers oriented compatibly 
outside the tori of the splice decomposition (see \cite[Corollary 8.3]{EN 85}). 
There is an algorithmic way to reduce any splice diagram to the minimal one representing 
the same link (see \cite[Theorems 8.1 and 8.2]{EN 85}).

The knowledge of a splice diagram of a graph link $\Lambda$ 
inside an oriented integral homology sphere $\Sigma$ allows to 
compute very easily the pairwise linking numbers of the 
components of $\Lambda$ (see Theorem \cite[10.1]{EN 85}):

\begin{proposition} \label{linksplice}
    Let $s(\Sigma, \Lambda)$ be a splice diagram for a graph link 
    $(\Sigma, \Lambda)$ inside an integral homology sphere $\Sigma$. 
    If $K_i, K_j$ are two distinct 
    components of $\Lambda$, then the linking number 
    $lk_{\Sigma}(K_i, K_j)$ is equal to the product of the weights 
    of the germs of edges adjacent to, but not included into the segment of 
    $s(\Sigma, \Lambda)$ which joins the arrowheads corresponding to 
    $K_i$ and $K_j$, multiplied by the product of the signs of the internal 
    vertices situated on this segment.
\end{proposition}

We restrict now to the splice diagrams of the links of reduced germs of curves  
inside smooth germs of complex surfaces (see \cite[Appendix to Chapter I]{EN 85}): 

\begin{theorem} \label{splcurve}
    Let $C$ be a germ of reduced holomorphic curve on the germ of complex analytic 
    smooth surface $S$. Then its oriented link $\Lambda(C)$
    inside the oriented boundary $\bS^3$ of $S$ is a graph link and it has a minimal 
    splice diagram whose vertex signs are all $+$ and whose edge decorations 
    are all strictly positive. 
\end{theorem}

As explained before, such a totally positive minimal splice diagram of 
$(\bS^3, \Lambda(C))$ is unique. We will call it the {\bf minimal splice diagram  of $C$}. 
The next theorem explains how to construct it from the Eggers-Wall tree of $C$ 
relative to a smooth branch $L$ which is transverse to it. It is a more graphical reformulation 
of Wall's \cite[Theorem 9.8.2]{W 04} (note that Wall spoke about {\em Eisenbud-Neumann 
diagrams} instead of {\em splice diagrams}). An advantage of speaking about the  
splice diagram of $C + L$ in the statement below allows a simpler
comparison of $\Theta_L(C)$ and of the minimal splice diagram of $C + L$ 
than in \cite{W 04}, avoiding special cases.

\begin{figure}[h!] 
\vspace*{1mm}
\labellist \small\hair 2pt 
  \pinlabel{$s$} at 67 495
  \pinlabel{$s$} at 67 287
  \pinlabel{$s$} at 67 87

  \pinlabel{$d$} at 96 460
  \pinlabel{$d$} at 112 513
  \pinlabel{$d$} at 40 513

  \pinlabel{$d$} at 96 256
  \pinlabel{$d$} at 128 293
  \pinlabel{$d'$} at 37 320
  \pinlabel{$d'$} at 120 330

  \pinlabel{$d$} at 96 56
  \pinlabel{$d'$} at 37 118
  \pinlabel{$d'$} at 120 118

  \pinlabel{$d^2 s$} at 370 460
  \pinlabel{$1$} at 382 513
  \pinlabel{$1$} at 310 513

  \pinlabel{$dd' s$} at 380 256
  \pinlabel{$d' / d$} at 405 290
  \pinlabel{$1$} at 307 320
  \pinlabel{$1$} at 390 330

  \pinlabel{$dd' s$} at 380 50
  \pinlabel{$1$} at 307 118
  \pinlabel{$1$} at 390 118
  \pinlabel{$d' / d$} at 390 82

\endlabellist 
\centering 
\includegraphics[scale=0.45]{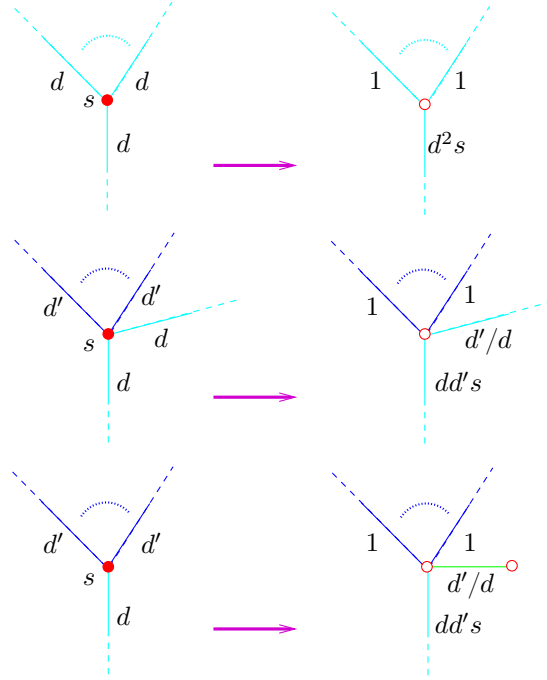} 
\vspace*{1mm} \caption{From the Eggers-Wall tree to the splice diagram}   
\label{fig:Changes} 
\end{figure}

\begin{theorem} \label{EWsplice}
   Let $C$ be a reduced germ of curve on the smooth germ of surface $S$ 
   and let $L$ be a smooth branch through $O$ such that $L$ is transversal to $C$. 
   Then the minimal splice diagram of $C + L$ may be obtained 
   from the Eggers-Wall tree $\Theta_L(C)$ decorated by the 
   contact function ${\ic}_L$ and the index functions  ${\de}_L$  
   by doing the local operations indicated in Figure \ref{fig:Changes}. 
\end{theorem}

\begin{proof}
     The topological type of $C + L$ is encoded by either of the following objects 
     (see Wall \cite[Proposition 4.3.9, Section 9.8]{W 04}):
        \begin{itemize}
              \item the collection of characteristic exponents of its branches and of 
                 intersection numbers between pairs of branches of $C + L$; 
              \item  the Eggers-Wall tree $\Theta_L(C + L)$; 
              \item the minimal splice diagram of $C + L$. 
        \end{itemize}
     Therefore, in order to prove the theorem it is enough to show that 
        the splice diagram obtained by our construction gives the same characteristic 
        exponents of individual branches and intersection numbers as the starting Eggers-Wall 
        tree. This verification may be done using the description from 
        \cite[Appendix to Chapter 1]{EN 85} of the way characteristic 
        exponents are encoded in the splice diagram of a branch and using 
        Proposition \ref{linksplice} for the way intersection numbers may be read 
        on a splice diagram of a germ with several branches. Here we use the fact that 
         the intersection number of two distinct branches on $S$ is equal to the linking 
         number of their associated knots in $\bS^3$.
       \medskip
       
    Let us give now a second proof 
        of the theorem, which furnishes a comparison with Wall's proof of \cite[Theorem 9.8.2]{W 04}.
    The  transversality hypothesis implies that the tree $\Theta_L(C)$ 
    contains no ramification point of exponent $<1$. 
We consider another smooth branch $L'$ transversal to the irreducible 
components of $C$ and to $L$. 
The attaching point of $L'$ on the tree $\Theta_{L} (C)$ is the unit point $U$  of this tree, 
which has exponent equal to $1$. 
By the inversion theorem  \ref{muthm}, the Eggers-Wall trees  $\Theta_{L'} (C + L)$  
and $\Theta_{L} (C + L')$ have the same exponent and index functions 
on the complement of the segment $[L, L']$. 
We apply the construction of  the splice diagram in 
\cite[Theorem 9.8.2]{W 04} to $\Theta_{L'} (C + L)$.  It  starts from 
the \emph{reduced} Eggers-Wall tree $\Theta^{red}_{L'} (C + L)$, which is obtained 
from $\Theta_{L'} (C + L)$ by removing the segment  $[L', U)$ and by 
unmarking the point $U$ in this tree if this point is not
a ramification point on the tree $\Theta_{L}(C)$ (this corresponds to (i) and (iv) in 
\cite[Theorem 9.8.2]{W 04}).

In order to make the comparison, Wall  considers the \emph{Herbrand function}  associated 
to a branch $B$ of $C + L$, which is a function $H_B: [0, \infty] \to [0, \infty]$ such that 
$H_B \circ \ex_{B,L'}  = \ic_{B,L'}$.

The first local operation in Figure \ref{fig:Changes} corresponds to point (iii) in 
Theorem 9.8.2 of \cite{W 04}, when 
the index function is continuous on the marked point $V$ considered. 
Wall considers a branch $B$ of $C + L$ 
through $V$ of multiplicity $m = \de_{L'} (B)$ and such that 
$P_q \prec_{L'}V \prec_{L'}  P_{q+1}$ where 
$P_j$ are the marked points of the tree $\Theta_{L'} (B)$. Then, the incoming edge 
at $V$ 
is marked by $(m^2 / e_q^2) \cdot H ( \ex_{L'} V )$, where 
$e_q = \de_{L'} (B) / \de_{L'} (V)$. 
We get the same decoration as in Figure \ref{fig:Changes} since: 
\[
\frac{m^2}{e_q^2} \cdot H_B ( \ex_{L'} (V )) =  (\de_{L'} (V))^2 \ic_{L'} (V) = d^2 \cdot s, 
\]
where we denote $d := \de_{L'} (V)$ and $s:=\ic_{L'} (V)$.

The second and third local operations in the figure below correspond to point (ii) in 
Theorem 9.8.2 of \cite{W 04},
when the index function is not continuous on the marked point $V$ considered.  
In the second case,  
there is a unique branch $B_{i_0}$ of $C$ passing through $V$ such that the 
index function restricted to this branch 
is continuous at $V$.  If $B_j$ is any other branch of $C$,  then $V$ is a marked point, 
say  $P_q$, of 
the tree $\Theta_{L'} (B_j)$.   In terms of Wall's notations, we have
 $e_q = \de_{L'} (B_j) / \de_{L'} (P_{q+1})$ and $e_{q-1} = \de_{L'} (B_j) / \de_{L'} (P_{q})$.

By \cite{W 04}, the outgoing segment at $V$ in the direction 
of a branch  $B_{i}$ is marked by
\[
\frac{e_{q-1}}{e_q} = 
\frac{\de_{L'} (P_{q+1}) }{ \de_{L'} (P_{q}) } = \frac{d'}{d}, 
\]
if $B_i= B_{i_0}$ and by $1$ otherwise (we denoted $d' := \de_{L'} (P_{q+1})$). 
Let us consider an auxiliary branch $K$ with $(q-1)$ characteristic exponents,  
having maximal contact with $B_j$. 
By definition, one has $(L' \cdot  K) = \de_{L'} (K) = d'$ and $(B_j \cdot  L') = e_q \cdot d$. 
The incoming edge 
at $V$ is marked by  $\bar{\beta}_q / e_q$, where $\{ \bar{\beta}_s \}_{s=0}^{g_j}$ 
denotes the sequence of minimal generators of the semigroup of the branch $B_j$. 
By Theorem \ref{intcomp},  
$s = \ic_{L'} (V) =  (B_j \cdot K) (K \cdot L')^{-1} (B_j \cdot L')^{-1}$,  
and thus we get the same decoration as in Figure \ref{fig:Changes} since: 
\[
d d' s = d d'  \frac{(B_j \cdot  K)}{ (K \cdot L') (B_j \cdot L')}  =  \frac{(B_j \cdot  K)} {e_q}= 
    \frac{\bar{\beta}_q}{e_q}.
\]

In the third case,  the index function is  not continuous on the marked point $V$ considered for 
all the branches of $C$ containing it. Then,   
we have to add a side at $V$ marked $d' / d$ to an end vertex, which is not arrow-headed.
\end{proof}

\begin{remark}

If $L$ is not transversal to $C$, the splice diagram associated to 
$C + L$  is obtained from the tree $\Theta_{L} (C + L)$ by doing the local operations 
indicated in Figure \ref{fig:Changes}, with respect to the values of the index 
and contact complexity functions on $\Theta_{L'} (C + L)$, where  $L'$ 
is a smooth branch transversal to $C + L$.  
\end{remark}

\begin{figure}[h!] 
\vspace*{2mm}
\labellist \small\hair 2pt 
  \pinlabel{$1$} at 156 46
  \pinlabel{$1$} at 125 110
  \pinlabel{$1$} at 93 184
  \pinlabel{$2$} at 64 250
  \pinlabel{$4$} at 43 295
  \pinlabel{$12$} at 23 340

  \pinlabel{$1$} at 200 124
  \pinlabel{$2$} at 130 160
  \pinlabel{$2$} at 165 245
  \pinlabel{$4$} at 197 320
  \pinlabel{$6$} at 224 248
  \pinlabel{$4$} at 82 309

  \pinlabel{$0$} at 140 4
  \pinlabel{$2$} at 105 80
  \pinlabel{$5/2$} at 77 130
  \pinlabel{$7/2$} at 36 225
  \pinlabel{$31/8$} at 20 263
  \pinlabel{$191/48$} at  -10 310
  \pinlabel{$31/12$} at 105 210
  \pinlabel{$21/8$} at 144 290

  \pinlabel{$L$} at 190 6
  \pinlabel{$C_1$} at 270 195
  \pinlabel{$C_2$} at 292 320
  \pinlabel{$C_3$} at 200 370
  \pinlabel{$C_4$} at 100 370 
  \pinlabel{$C_5$} at 0 375

  \pinlabel{$2$} at 495 57
  \pinlabel{$1$} at 480 96
  \pinlabel{$1$} at 524 96

  \pinlabel{$5$} at 488 118
  \pinlabel{$2$} at 456 143
  \pinlabel{$1$} at 492 146

  \pinlabel{$7$} at 450 210
  \pinlabel{$2$} at 418 210
  \pinlabel{$1$} at 418 235

  \pinlabel{$31$} at 435 254
  \pinlabel{$2$} at 402 253
  \pinlabel{$1$} at 400 278
  \pinlabel{$1$} at 440 278

  \pinlabel{$191$} at 415 300
  \pinlabel{$3$} at 383 298
  \pinlabel{$1$} at 400 330

  \pinlabel{$31$} at 510 180
  \pinlabel{$3$} at 498 226
  \pinlabel{$1$} at 546 217

  \pinlabel{$21$} at 545 267
  \pinlabel{$2$} at 518 284
  \pinlabel{$1$} at 560 305

  \pinlabel{$L$} at 560 7
  \pinlabel{$C_1$} at 634 200
  \pinlabel{$C_2$} at 660 323
  \pinlabel{$C_3$} at 570 370
  \pinlabel{$C_4$} at 470 370 
  \pinlabel{$C_5$} at 370 375

\endlabellist 
\centering 
\includegraphics[scale=0.65]{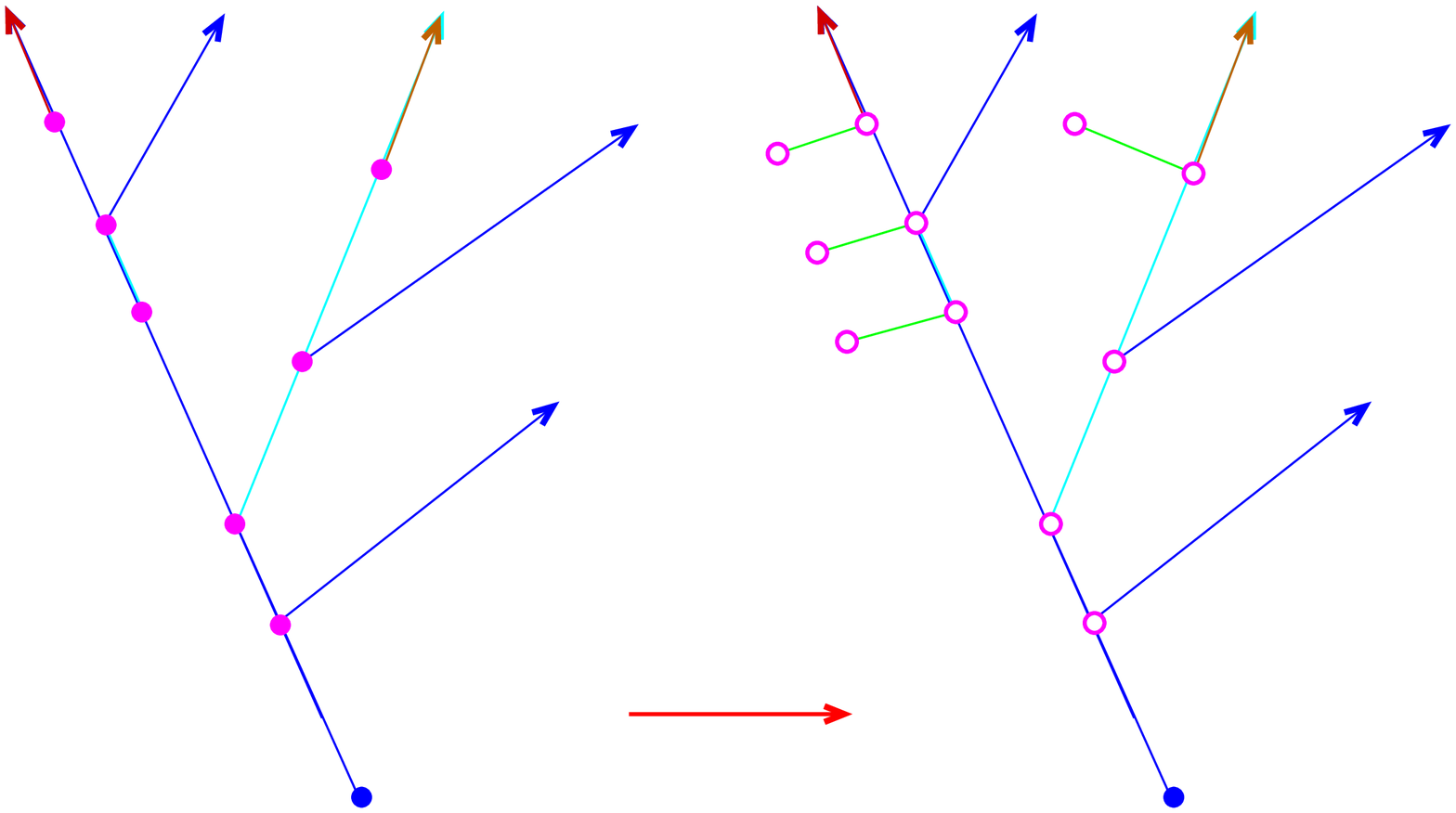} 
\vspace*{1mm} \caption{The splice diagram associated to our recurrent example}   
\label{fig:ExEWtosplice1} 
\end{figure}

\begin{example}
   Consider again our recurrent Example \ref{EWmany}. Recall that 
   the values of 
   the contact complexity function and of the  index function are 
   represented in Figure \ref{fig:Intfive}.
   The result of applying the previous theorem 
   is indicated in Figure \ref{fig:ExEWtosplice1}. One may verify that the application of 
   Proposition \ref{linksplice} gives the same values of the intersection numbers 
   $(C_i \cdot C_j)$ as those computed in Example \ref{Exintcoef}. 
\end{example}

\section{Semivaluation spaces} \label{valspaces}

In this section we define the spaces of valuations and semivaluations 
of  $\mathcal{O}$ which will be used in the sequel: the space $\cV$ of all real-valued semivaluations 
(see Definition \ref{semivalspace}), 
its projectivization $\bP(\cV)$ (see Definition \ref{projsvsp}) 
and the sets of normalized semivaluations relative either to the base point $O$ of $S$ 
or to a smooth branch $L$ on $S$ (see Definition \ref{normaliz}).
We describe also the types of semivaluations used in the next sections: the \emph{multiplicity valuations}, 
the \emph{intersection semivaluations} and the \emph{vanishing order valuations} 
(see Definition \ref{def:val}). 
\medskip

Recall that we denote by $\mathcal{O}$ the formal local ring of $S$ at $O$, by  
$\mathcal{F}$ its field of fractions and by $\mathcal{M}$ the maximal ideal of $\mathcal{O}$. 

\begin{definition} \label{valdef}
  Extend the usual total order relation of  $\R$ to $\R\cup
  \{\infty\}$ by the convention that $\infty > \lambda$, for all $\lambda \in
  \R$.   A {\bf semivaluation of $\mathcal{O}$} is a function 
   $\nu:  \mathcal{O} \rightarrow [0, \infty]$ such that:
   
 \begin{enumerate}
    \item \label{add} $\nu(fg)=\nu(f) + \nu(g)$ for all $f,g \in \mathcal{O}$; 

    \item \label{ineq} $\nu(f+g) \geq \min(\nu(x), \nu(y))$  for all $f,g \in \mathcal{O}$;

    \item \label{const} $\nu(\lambda):= \left\{ \begin{array}{ll}
                                    0 & \mbox{ if } \lambda \in \C^* , \\
                                    \infty & \mbox{ if } \lambda=0.
                                  \end{array} \right. $
 \end{enumerate} 
A semivaluation $\nu$ of $\mathcal{O}$ is {\bf centered at $O$} if and only if one has moreover:
      $\nu(\mathcal{M}) \subset \R_+^* \cup \{\infty\}$.  
   The semivaluation $\nu$ is a {\bf valuation} if it takes the value $\infty$ 
   only at $0$. 
\end{definition}

\begin{remark}
    If $\nu$ is a semivaluation, then the function 
     $|| \cdot || := e^{-\nu} : \mathcal{O} \to [0, 1]$ is a multiplicative 
     \emph{non-archimedean seminorm} 
    of the $\C$-algebra $\mathcal{O}$, that is: 
         \begin{enumerate}[label=({\alph*})]
    \item[(1)'] $|| xy ||=  || x || \cdot || y ||$ for all $x,y \in \mathcal{O}$; 

    \item[(2)']  $|| x+y  || \leq \max(|| x  ||,  ||y  ||)$  for all $x,y \in \mathcal{O}$;

    \item[(3)']  $|| \lambda || := \left\{ \begin{array}{ll}
                                    1 & \mbox{ if } \lambda \in \C^*,  \\
                                     0 & \mbox{ if } \lambda=0.
                                  \end{array} \right. $
 \end{enumerate} 
 The term {\em semivaluation} was introduced as an analog of the more standard term 
 {\em seminorm}. 
\end{remark}

If $f \in \mathcal{O}$ defines the germ of divisor $D$  and if $\nu$ 
is any semivaluation of $\mathcal{O}$,  we set: 
\[ \nu (D):= \nu ( f). \]
This definition is independent of the defining function $f \in \mathcal{O}$ of $D$. 
Indeed, any other such function is of the form $fu$, with $u$ a unit of $\mathcal{O}$. 
But then $\nu(u) + \nu(u^{-1}) = \nu(1) =0$, which implies that $\nu(u) =0$, 
as $\nu$ takes only non-negative values. Therefore one has also 
$\nu(fu) = \nu(f) + \nu(u) = \nu(f)$. 
More generally, if $\mathcal{I}$ is an arbitrary ideal of $\mathcal{O}$, we set:
   \[ \nu(\mathcal{I}) := \min \:  \{ \nu(f) \: | \: f \in \mathcal{I} \}. \]
This definition generalizes the previous one because the value $\nu(D)$ 
computed according to the first definition is equal to the value $\nu( \mathcal{O}(-D))$ 
computed according to the second one.

\begin{definition} \label{semivalspace}
    Denote by $\mathcal{V}$ the set of semivaluations of $\mathcal{O}$. We call it 
    the {\bf semivaluation space}
    of $\mathcal{O}$ or of the germ $S$. 
    We endow it with the topology of pointwise convergence, that is, with the 
    restriction of the product topology of $[0, \infty]^{\mathcal{O}}$. 
\end{definition}

The topological space $[0, \infty]^{\mathcal{O}}$ is compact as a product of 
compact spaces, by Tychonoff's theorem 
(see for instance \cite[Section 1-10]{HY 61}). 
The conditions defining semivaluations being closed, we see that: 

\begin{proposition} \label{compsv}
    The semivaluation space $\mathcal{V}$ is compact. 
\end{proposition}

\begin{remark}
  In contrast to the space $\mathcal{V}$ of semivaluations, the subspace 
   of valuations is not compact. This is the 
    main reason of the importance in our context  not only of valuations, but also 
    of semivaluations which are not valuations. 
 \end{remark}

 Let us define now the main types of semivaluations which we use in this paper: 
\medskip 

\noindent
\begin{definition} \label{def:val}
     The {\bf multiplicity valuation} at $O$, denoted by $I^O$, is defined by: 
        $$I^O(f) = \max\{n \in \N \ | \ f \in \mathcal{M}^n \}.$$
More generally, if $P$ is an infinitely near point of $O$, denoted by $I^P$, 
the associated {\bf multiplicity valuation} at $P$.
It may be defined in the following two equivalent ways, starting from a model  $(\Sigma,E)
\stackrel{\psi}{\rightarrow} (S,O)$ containing $P$:
\begin{itemize}
   \item If $f \in \mathcal{O}$, then $I^P(f)$ is the multiplicity of the 
            function $f\circ \psi$ at the point $P$ of the model $\Sigma$: 
             $$ I^P(f) : = m_P(f\circ \psi).$$

   \item If $f \in \mathcal{O}$, then $I^P(f)$ is the vanishing order 
            of  $f \circ \psi\circ \psi_P$ along $E_P$, where $\Sigma_P
           \stackrel{\psi_P}{\rightarrow} \Sigma$ is the blow up of $P$ in $\Sigma$ and 
           $E_P$ is the exceptional divisor created by it. That is, $I^P(f)$ is the coefficient 
         of $E_P$ in the divisor of $f \circ \psi\circ \psi_P$. 
\end{itemize}
Because of this second interpretation, we often denote: 
  \[ \mbox{ord}^{E_P} := I^P.\]

        \medskip
       
       \noindent
 Let $A$ be a branch at $O$. One has an associated {\bf intersection semivaluation} 
     $I^A$,  defined by: 
          $$I^A(f) := (A\cdot  Z(f)).$$
      Note that these are semivaluations which are not valuations, as $I^A(f) = \infty$ 
      precisely for the elements of the principal ideal $\mathcal{O}(-A)$ of the functions 
      vanishing identically on $A$. 
     
      \medskip
     
      \noindent
      All the previous examples of semivaluations are centered at $O$. To any branch $A$ 
         at $O$ is also associated a valuation which is not centered at $O$: 
         the {\bf vanishing order}  $\mbox{ord}^A$ along $A$:
            $$\mbox{ord}^A(f) := \mbox{ the coefficient of $A$ in the divisor of $f$}.$$

         \medskip
        
         \noindent
 If $V$ is a germ of irreducible subvariety of $S$ through $O$ (that is, 
      either the point $O$, or a branch $A$, or $S$ itself), the {\bf trivial semivaluation} 
       $\mbox{triv}^V$ associated to $V$ takes only two values:
      $$ \mathrm{triv}^V(f) := 
         \left\{ \begin{array}{cl}
                        \infty  & \mbox{ if }  f \in \mathcal{O} \mbox{ vanishes along } V, \\ 
                         0  & \mbox{ otherwise}. 
                  \end{array} \right. $$
          Among the trivial semivaluations, only $\mbox{triv}^S$ is a valuation. 
                    
\end{definition}

\begin{remark}  \label{remnot}
     We have denoted till now by $m_O(C)$ the multiplicity of a germ of curve $C$ at $O$. 
     We could have chosen to keep this notation, and to write $m_P$ instead 
     of $I^P$ when $P$ is infinitely near $O$. We have decided not to follow 
     this notational convention, because we will introduce in the next 
     section an invariant of semivaluations 
     called {\em multiplicity}, denoted by ${\mult}$, and we wanted to avoid 
     the notation ``${\mult}(m_P)$'' for the multiplicity of the valuation $m_P$. 
 \end{remark}

The multiplicative group 
$(\R^*_+, \cdot)$ acts on the semivaluation space $\mathcal{V}$ by scalar multiplication 
of the values. We denote by 
$t \nu \in \mathcal{V}$ the product of $t \in \R^*_+$ and $\nu \in \mathcal{V}$. 
One may show that this action is continuous. Its orbits allow to relate  
the three kinds of semivaluations $I^A$, $\mbox{ord}^A$ and $\mbox{triv}^A$ 
associated to a branch $A$ at $O$: 

\begin{proposition}  \label{relateval}
     Let $A$ be any branch through $O$. Then the orbit of the 
     vanishing order valuation $\mathrm{ord}^A$ 
     goes from $\mathrm{triv}^S$ to $\mathrm{triv}^A$ and the orbit of 
     the intersection semivaluation $I^A$ goes from 
     $\mathrm{triv}^A$ to $\mathrm{triv}^O$, that is (see Figure \ref{fig:Orbits}):
       \begin{itemize}
             \item $\displaystyle{\lim_{t \to 0} (t  \: \emph{ord}^A) } = 
                 \emph{triv}^S$ and   
                    $\displaystyle{\lim_{t \to \infty} (t  \: \emph{ord}^A) } = 
                 \emph{triv}^A$; 
             
             \item $\displaystyle{\lim_{t \to 0} (t  \: I^A) } = 
                 \emph{triv}^A$ and   
                    $\displaystyle{\lim_{t \to \infty} (t  \: I^A) } = 
                 \emph{triv}^O$.
       \end{itemize}
\end{proposition}

\begin{figure}[h!] 
\labellist \small\hair 2pt 
\pinlabel{$\mathrm{triv}^S$} at 5 4
\pinlabel{$\mathrm{ord}^A$} at 37 87
\pinlabel{$\mathrm{triv}^A$} at 147 70
\pinlabel{$I^A$} at 230 90
\pinlabel{$\mathrm{triv}^O$} at 275 -10

\endlabellist 
\centering 
\includegraphics[scale=0.50]{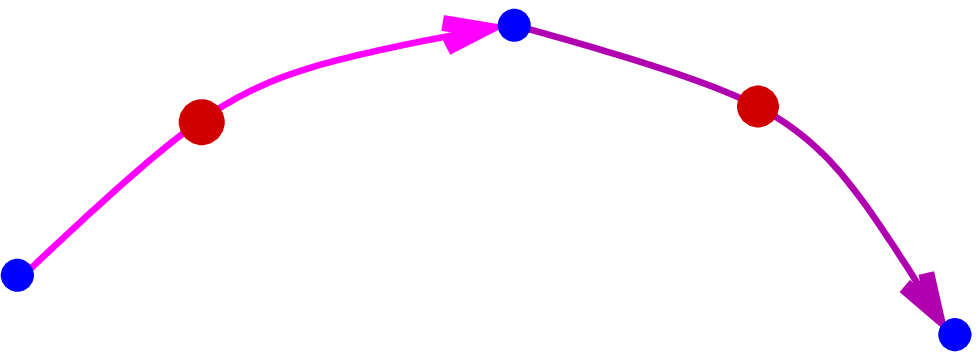} 
\caption{The orbits of $\mathrm{ord}^A$ and of $I^A$} 
\label{fig:Orbits}
\end{figure}

The previous proposition is in fact much more general, as shown by Proposition 
\ref{limpoints} below. Before stating it, let us introduce a new definition.

\begin{definition} \label{centsup}
   Assume that we work with an arbitrary 
   irreducible analytic or formal germ $\mathcal{X}$, with local ring $\mathcal{R}$. The {\bf center}  
   $C(\nu)$  of a semivaluation $\nu$ of $\mathcal{R}$ is the irreducible 
   subvariety of $\mathcal{X}$ defined by the functions $f\in \mathcal{R}$ such that $\nu(f) > 0$.   
   The {\bf support} $S(\nu)$ of $\nu$ is the irreducible subvariety 
   of $\mathcal{X}$ defined by those functions $f\in \mathcal{R}$ such that $\nu(f)  = \infty$. 
 \end{definition}
   
   Obviously, 
   $C(\nu) \subseteq S(\nu)$. The announced generalization of Proposition 
   \ref{relateval} is: 
   
   \begin{proposition} \label{limpoints}
    The orbit of $\nu$ under scalar multiplication by $t \in \R^*_+$ goes from 
   $\mathrm{triv}^{S(\nu)}$ to $\mathrm{triv}^{C(\nu)}$ 
   when $t$ goes from $0$ to $\infty$.
   \end{proposition}
   
   \begin{proof}
          Let $f \in \mathcal{R}$ 
          be arbitrary. We have the following possibilities:
              \begin{itemize}
                   \item   If $\nu(f) =0$, then $\displaystyle{\lim_{t \to 0} (t  \: \nu)(f) } = 
                                \displaystyle{\lim_{t \to \infty} (t  \: \nu)(f) } = 0$. 
                   
                   \item   If $\nu(f) \in (0, \infty)$, then $\displaystyle{\lim_{t \to 0} (t  \: \nu)(f) } = 0$ 
                                and $\displaystyle{\lim_{t \to \infty} (t  \: \nu)(f) } = \infty$. 
                   
                   \item   If $\nu(f) = \infty$, then $\displaystyle{\lim_{t \to 0} (t  \: \nu)(f) } = 
                                \displaystyle{\lim_{t \to \infty} (t  \: \nu)(f) } = \infty$. 
              \end{itemize}
            The conclusion follows readily from this. 
   \end{proof}
   
   Let us return to our germ $S$. 
   In fact, the semivaluations $I^A, \mathrm{ord}^A$ associated to the branches 
   $A$ on $S$ may be characterized, up to scalar multiplication, as the only 
   ones whose orbits do not connect $\mathrm{triv}^S$ to $\mathrm{triv}^O$:
   
   \begin{proposition}
      Let $\nu \in \mathcal{V}$. If the orbit of $\nu$ is not constant and does not go from 
      $\mathrm{triv}^S$ to $\mathrm{triv}^O$, then $\nu$ is proportional either to 
      $I^A$ (if $\displaystyle{\lim_{\lambda \to 0} (\lambda  \: \nu) } = 
                 \emph{triv}^A$) or to $\mathrm{ord}^A$ 
                 (if $\displaystyle{\lim_{\lambda \to \infty} (\lambda  \: \nu) } = 
                 \emph{triv}^A$), where $A$ denotes a branch on $S$. 
   \end{proposition} 
   
   \begin{proof}
      This comes from the fact that any irreducible subgerm of $S$ which is 
      distinct from $O$ and $S$ is necessarily a branch $A$, and that:
              
            --  a semivaluation 
                  whose center is $A$ is proportional to $\mathrm{ord}^A$;                  
             
             -- a semivaluation whose support is $A$ is proportional to $I^A$.
   \end{proof}
   
   The other types of semivaluations described in Definition \ref{semivalspace} 
   do not cover all of $\mathcal{V}$. One may find concrete descriptions of the remaining 
   possibilities in \cite[Sect.1.5]{FJ 04}.

\medskip

The previous considerations show 
that the quotient of $\mathcal{V}$ under the given action (that is, 
the space of orbits endowed with the quotient topology), is highly 
non-Hausdorff, because the closure of any point would contain either the image 
of $\mbox{triv}^S$ or of $\mbox{triv}^O$. A way to avoid this is to remove those two trivial 
semivaluations before doing the quotient. 
This does still not produce a Hausdorff quotient, because there exist 
sequences of orbits converging to the union of $\mbox{triv}^A$ and of 
the orbits of $I^A$ and of $\mbox{ord}^A$. But this is the only phenomenon which makes the space 
non-Hausdorff, and if one quotients more, by identifying those three orbits 
for each branch $A$, one gets a Hausdorff space:

  \begin{definition}   \label{projsvsp}
      The {\bf projective semivaluation space} 
      $\mathbb{P}(\mathcal{V})$  of $\mathcal{O}$ or of the germ $S$ is 
      the biggest Hausdorff quotient of $\mathcal{V}^* := 
      \mathcal{V} \:  \setminus \:  \{  \mbox{triv}^S, \mbox{triv}^O \}$ under the previous 
      action of $(\R^*_+, \cdot)$. Let:
          \begin{equation}  \label{quotmap}
                \pi : \mathcal{V}^* \to \mathbb{P}(\mathcal{V})
          \end{equation}
      be the associated continuous quotient map. 
      We say that an element of $\mathbb{P}(\mathcal{V})$ is a {\bf projective semivaluation} 
      of $\mathcal{O}$. 
  \end{definition}

The central Theorem 3.14 of \cite{FJ 04} implies that: 

\begin{theorem}
   $\mathbb{P}(\mathcal{V})$ is a compact $\R$-tree endowed with 
   its weak topology. 
\end{theorem}

In Section \ref{fintrees} we have not defined $\R$-trees directly as topological 
spaces, but as equivalence classes of special partial orders on a set, endowed 
with a canonically defined ``weak'' topology. In fact, Favre and Jonsson recognize 
the structure of $\R$-tree of $\mathbb{P}(\mathcal{V})$ 
in the same way, by defining first special partial orders on it. Those partial orders 
are not defined directly on $\mathbb{P}(\mathcal{V})$, but on sections of the 
projection $\pi$. Those sections are introduced using {\em normalization} rules 
relative either to $O$ or to a smooth branch $L$ through $O$:

\begin{definition} \label{normaliz} 
    A semivaluation $\nu \in \mathcal{V}$ is {\bf normalized relative to $O$} 
    if $\nu(\mathcal{M}) = 1$. Denote by $\mathcal{V}_O \subset \mathcal{V}$ 
    the subspace of semivaluations normalized  relative to $O$. If $\nu 
    \in \mathcal{V} \:  \setminus \:  \{\mathrm{triv}^O\}$ is centered 
    at $O$, we denote by $\nu_O\in \mathcal{V}_O$ the unique 
    semivaluation normalized relative to $O$ which is proportional to $\nu$. 
    
    Analogously, if $L$ is an arbitrary smooth branch, 
    we define the subspace $\mathcal{V}_L \subset \mathcal{V}$ 
    of semivaluations {\bf normalized relative to  $L$}  
    by the condition $\nu(L) = 1$, and if $\nu\in \mathcal{V}$ is not 
    supported by $L$, we denote by 
    $\nu_L$ the unique semivaluation in $\mathcal{V}_L$ which is 
    proportional to $\nu$.
\end{definition}

Notice that we have the following concrete descriptions of the normalizations of a given 
semivaluation $\nu$:
  \begin{equation} \label{normaliznot}
           \nu_O = \dfrac{\nu}{\nu(\mathcal{M})}, 
           \quad 
           \nu_L  = \dfrac{\nu}{\nu(L)}. 
  \end{equation}

Both subspaces $\mathcal{V}_O$ and $\mathcal{V}_L$ are closed inside $\mathcal{V}$, 
therefore compact, as $\mathcal{V}$ is compact. 
On each one of them,  one restricts the following partial order on $\mathcal{V}$:
   \begin{equation} \label{partord} 
         \nu_1 \preceq \nu_2 \:  \Leftrightarrow \: \nu_1(f) \leq \nu_2(f) 
       \mbox{ for any } f \in \mathcal{O}.
    \end{equation}
  Consider also the restrictions to them of the projection $\pi$: 
    \begin{equation} \label{restrOL}
            \pi_O : \mathcal{V}_O \to  \mathbb{P}(\mathcal{V}),  \  
               \pi_L : \mathcal{V}_L \to  \mathbb{P}(\mathcal{V})  .
     \end{equation}

What Favre and Jonsson prove in fact is:

\begin{theorem} \label{equitree}
   Endowed with the restrictions of the previous partial orders, 
   both $\mathcal{V}_O$ and $\mathcal{V}_L$ 
   are compact rooted $\R$-trees, their roots being 
   $I^O$ and $\mathrm{ord}^L$ respectively. The maps 
   $\pi_O$ and $\pi_L$ are both homeomorphisms, which induce 
   the same structure of (non-rooted) $\R$-tree on $\mathbb{P}(\mathcal{V})$. 
   The composed homeomorphism $\pi_L^{-1} \circ \pi_O : \mathcal{V}_O \to 
   \mathcal{V}_L$ sends $I^L$ to $\mathrm{ord}^L$.
\end{theorem}

Let us denote by $\preceq_O$ the partial order on $\mathbb{P}(\mathcal{V})$ 
induced from that of $\mathcal{V}_O$ and by $\preceq_L$ the one induced 
by that of $\mathcal{V}_L$. Those notations are motivated by the fact that 
they are the orders induced by the choice of the root at $\pi(I^O)$ and 
$\pi(\mathrm{ord}^L)$ respectively.

Favre and Jonsson prove in \cite{FJ 04} that the multiplicity valuations give by 
projectivization interior points 
of $\mathbb{P}(\mathcal{V})$ and that those points are dense inside any 
finite subtree. They may be characterized 
as being precisely the ramification points of the tree $\mathbb{P}(\mathcal{V})$. 
By contrast, the intersection semivaluations 
are end points. They are not the only ends, but they cannot be characterized purely in terms 
of the poset or topological structure of the tree $\mathbb{P}(\mathcal{V})$. 
One needs a supplementary structure on it, 
a \emph{multiplicity function}. It is one member of a triple of fundamental increasing functions 
defined on $(\mathbb{P}(\mathcal{V}), \preceq_{O}) $. The next section is dedicated to them.

\section{Multiplicities, log-discrepancies and self-interactions} \label{fundcoord}

Either the point $O$ or any smooth branch $L$ may be seen as an \emph{observer} 
of the projective semivaluation space  $\mathbb{P}(\mathcal{V})$. Namely, to each 
one of them is associated a \emph{coordinate system}, which is a triple of 
functions defined on $\mathbb{P}(\mathcal{V})$,  the {\em multiplicity}, 
the {\em log-discrepancy} and the {\em self-interaction} relative to that observer. 
We introduce those functions in Definitions \ref{logselfgen} and \ref{relmult}.   
In Proposition \ref{reldiffprop} we explain how to express each one of them 
in terms of the two other ones.  
Our presentation is a variation on those of Favre and Jonsson 
\cite[Sections 3.3.1, 3.4, 3.6]{FJ 04} and of Jonsson \cite[Section 7]{J 15}).
\medskip

If $E_i$ is a prime 
divisor over $O \in S$, recall that $\mbox{ord}^{E_i}$ denotes the associated vanishing 
order valuation. For such a divisor, 
consider an arbitrary model $\psi: (\Sigma, E) \to (S, O)$ containing it. We will denote 
by $(D \cdot D')_{\Sigma}$ the intersection number of two divisors on $\Sigma$ 
without common non-compact branches. Let 
$ \check{E}_i$ be the \emph{dual divisor} 
in this model, that is, the only divisor supported by $E$ 
such that $(\check{E}_i \cdot E_j)_{\Sigma}  = \delta_{i,j}$ for all the components $E_j$ of $E$. 

\begin{definition} \label{deflogself}
    The {\bf log-discrepancy} 
    ${\ld}(\mathrm{ord}^{E_i})$ and the {\bf self-interaction}  
    ${\si}(\mathrm{ord}^{E_i})$ 
    of the valuation $\mathrm{ord}^{E_i}$ are the positive integers defined by:
    
    \begin{itemize}
        \item   ${\ld}(\mathrm{ord}^{E_i}) := 1 + \mathrm{ord}^{E_i}(\psi^* \omega)$, where 
           $\omega$ is a non-vanishing holomorphic $2$-form on $S$ in the neighborhood of $O$. 
        
        \item  ${\si}(\mathrm{ord}^{E_i}) := - (\check{E}_i \cdot \check{E}_i)_{\Sigma} \geq 1$. 
    \end{itemize}    
\end{definition}

The previous definition is independent of the chosen model. This is clear for the log-discrepancy, 
but is a theorem for $(\check{E}_i \cdot \check{E}_i)_{\Sigma}$. This is the main reason of the 
importance of the dual divisors $ \check{E}_i$ in birational geometry over $S$. 
Indeed, the self-intersections $(E_i \cdot E_i)_{\Sigma}$ are not invariant 
under blow-ups of points of $E_i$. 

\begin{remark}
   We have chosen the letter ``${\ld}$'' as the initial of ``log-discrepancy'' and 
   the letter ``${\si}$'' as initial of ``self-interaction''. We think about a self-intersection 
   number as a measure of interaction of an object with itself. 
   See also Proposition \ref{selfcurv} for another interpretation of this measure of self-interaction. 
   In \cite{FJ 04}, ${\ld}$ is called ``\emph{thinness}'' and is denoted ``A'', while 
   ${\si}$ is called ``\emph{skewness}''  and is denoted ``$\alpha$''. In \cite{J 15}, those 
   names are not used any more, but the notations  ``$A$'' and ``$\alpha$'' remain, 
   ``$\alpha$'' being used with an opposite sign convention with respect to \cite{FJ 04}. 
\end{remark}

Recall that the notation $\mathcal{V}^*$ was introduced in Definition \ref{projsvsp}:

\begin{proposition}  \label{uniqext}
   There exist unique functions ${\ld}, {\si}  : \mathcal{V}^* \to (0, \infty]$ such that: 
       \begin{enumerate}
            \item In restriction to the  valuations $\emph{ord}^{E_i}$, one gets 
               the functions introduced in Definition \ref{deflogself}. 
            
            \item They are continuous in restriction to any subset of the form $\pi^{-1}(T)$, 
               where $\pi$ is the quotient map (\ref{quotmap}) and $T$ is a finite subtree of 
                $\mathbb{P}(\mathcal{V})$. 
            
            \item ${\ld}$ is homogeneous of degree $1$ and  ${\si}$ 
                is homogeneous of degree $2$ 
                relative to the action of $(\R^*_+, \cdot)$. 
       \end{enumerate}
\end{proposition}

\begin{definition}  \label{logselfgen}
    If $\nu \in \mathcal{V}^*$,  then  ${\ld}(\nu)$ is called the {\bf log-discrepancy} 
    of $\nu$ and ${\si}(\nu)$ is called its {\bf self-interaction}. 
\end{definition}

The self-interaction function may be seen as the quadratic function associated to 
the $(1,1)$-bihomogeneous function described by the following proposition, 
similar to Proposition \ref{uniqext}:

\begin{proposition} 
    There exists a unique function $\bra{ \cdot, \cdot}  : \mathcal{V}^* \times \mathcal{V}^*
         \to (0, \infty]$ such that: 
       \begin{enumerate}
            \item $\bra{\emph{ord}^{E_i} , \emph{ord}^{E_j}} = - (\check{E}_i \cdot \check{E}_j)_{\Sigma}$ 
                for any model $\psi: (\Sigma, E) \to (S, O)$ containing both $E_i$ and $E_j$. 
            
            \item It is continuous in restriction to any subset of the form 
                $\pi^{-1}(T) \times \pi^{-1}(T) $, where $T$ is a finite subtree of 
                $\mathbb{P}(\mathcal{V})$. 
            
            \item It is bihomogeneous of degree $(1,1)$ relative to the action of 
                $(\R^*_+, \cdot)$ on both entries. 
       \end{enumerate}
\end{proposition}

The following terminology is taken from \cite[Definition 1.6]{GBGPPPR 18}:

\begin{definition}
    If $\nu_1, \nu_2 \in \mathcal{V}^*$, we say that $\bra{\nu_1, \nu_2} \in \R$ is the 
    \bf bracket of $\nu_1$ and $\nu_2$.
\end{definition}

The bracket is obviously symmetric, and ${\si}(\nu) = \bra{\nu, \nu}$ for any 
$\nu \in \mathcal{V}^*$. The following proposition 
gives an alternative description of it for divisorial valuations:

\begin{proposition}   \label{selfcurv}
      Let $E_i$ and $E_j$ be two prime divisors over $O$, which are 
      not necessarily distinct and let 
      $\psi: (\Sigma, E) \to (S, O)$ be a model containing both of them. 
      Consider curvette $K_i$ and $K_j$ for $E_i$ and $E_j$ respectively in this model, 
      that is, germs of smooth curves transversal to $E$ at points of the corresponding 
      irreducible components of it. 
      If $E_i = E_j$, we assume that the two curvette do not pass through the same 
      point of $E_i$. Let $C_i$ and $C_j$ be their projections on $S$ by the 
      morphism $\psi$. 
      Then we have: 
          $$\bra{\mathrm{ord}^{E_i}, \mathrm{ord}^{E_j}} = (C_i \cdot C_j). $$
\end{proposition}

\begin{proof}
   As the intersection number of a compact divisor on a smooth surface with a principal one 
  is $0$, we have: $$(E_k \cdot \psi^* D)_{\Sigma} =0$$ for any component $E_k$ of $E$ and for any 
   effective divisor $D$ on $S$. Let us apply this fact to $D = C_i$. Denote by 
   $\tilde{E}_i$ the exceptional part of the divisor $\psi^* C_i$. We get:
  $$ 0 = (E_k \cdot \psi^* C_i)_{\Sigma} = (E_k \cdot (\tilde{E}_i + K_i))_{\Sigma} = 
       (E_k\cdot  \tilde{E}_i)_{\Sigma} + \delta_{k,i}.$$
   This equality being valid for all the components $E_k$ of $E$, we see that 
    $\tilde{E}_i= - \check{E}_i$. In particular:
    $$(K_i \cdot F)_{\Sigma} =  (\check{E}_i \cdot F)_{\Sigma}$$
    for any divisor $F$ on $\Sigma$ supported by $E$. Therefore:
     $$\begin{array}{ll}
      (C_i \cdot C_j) &  =  (\psi^* C_i \cdot \psi^* C_j)_{\Sigma} = (K_i \cdot \psi^* C_j)_{\Sigma} = 
              (K_i \cdot (-  \check{E}_j + K_j))_{\Sigma} = \\
           &    = -  (K_i \cdot    \check{E}_j )_{\Sigma} = - (\check{E}_i \cdot \check{E}_j )_{\Sigma} = 
                     (\mathrm{ord}^{E_i} \cdot \mathrm{ord}^{E_j}).
        \end{array}      $$

\end{proof}

There is also an alternative description in the case when one of the semivaluations is the 
intersection semivaluation of a branch or the multiplicity valuation at $O$:

\begin{proposition}  \label{geomint}
     Let $A$ be a branch on $S$ and $\nu \in \mathcal{V}^*$. Then:  
       $$ \bra{\nu, I^A} = \nu(A).$$
    In particular, if $A, B$ are distinct branches at $O$, one gets 
      $ \bra{I^{A}, I^{B }} = (A \cdot B)$. Analogously: 
         $$\bra{\nu, I^O} = \nu(\mathcal{M}).$$
\end{proposition}

The log-discrepancy ${\ld}$ and the self-interaction ${\si}$ are functions defined 
on $\mathcal{V}^*$. One may push them down to $\mathbb{P}(\mathcal{V})$ 
using images of sections of the quotient map $\pi: \mathcal{V}^* \to \mathbb{P}(\mathcal{V})$. 
As mentioned in Theorem \ref{equitree}, the maps 
$\pi_O :\mathcal{V}_O \to \mathbb{P}(\mathcal{V})$ and 
$\pi_L : \mathcal{V}_L \to \mathbb{P}(\mathcal{V})$ are homeomorphisms 
(where $L$ denotes an arbitrary smooth branch), which shows that 
$\mathcal{V}_O$ and $\mathcal{V}_L$ are such images. This motivates 
the following definition:

\begin{definition}  \label{relcoord}
    The functions ${\ld}_O, {\si}_O : \mathbb{P}(\mathcal{V}) \to [0, \infty]$ 
    and $ \bra{\cdot, \cdot}_O : \mathbb{P}(\mathcal{V}) \times \mathbb{P}(\mathcal{V})  
       \to (0, \infty]$ 
    are defined by:
       $${\ld}_O := {\ld} \circ \pi_O^{-1}, \:  
            \:  {\si}_O := {\si} \circ \pi_O^{-1}, \: 
             \:  \bra{\cdot, \cdot}_O := \bra{\pi_O^{-1}(\cdot), \pi_O^{-1}(\cdot)}.$$
      That is, they are the push-forwards of the functions ${\ld}, {\si}, \bra{\cdot, \cdot}$ 
      by the homeomorphism $\pi_O$. 
     They are called the {\bf log-discrepancy relative to $O$}, 
     the {\bf self-interaction relative to $O$} 
     and the {\bf bracket relative to $O$}. 
     One defines analogously three functions ${\ld}_L, {\si}_L, \bra{\cdot, \cdot}_L$ 
     {\bf relative to $L$}. 
\end{definition}

We will work also with a third kind of functions on 
$\mathbb{P}(\mathcal{V})$ relative to $O$ or to a smooth branch $L$, this time 
taking values in $\N^* \cup \infty$:

\begin{definition}  \label{relmult}
    Let $R$ denote either  $O$ or a smooth branch $L$. The 
       {\bf multiplicity relative to $R$} is the function denoted ${\mult}_R :  \mathbb{P}(\mathcal{V}) 
         \to \N^*\cup \{\infty\}$ and defined by:
      \begin{equation}   \label{relmultexpr}
          {\mult}_R(P) := \min \{ \bra{I^R, C} \ | \ P \preceq_R C \}.
      \end{equation}   
    Here $\preceq_R$ is the partial order relation defined on the tree $ \mathbb{P}(\mathcal{V}) $ 
    by choosing the root at $\pi(I^R)$ and $C$ denotes a branch on $S$. 
\end{definition}

We think of the irreducible subvariety $O$ or $L$ of $S$ as an  {\em observer} 
of the topological space $ \mathbb{P}(\mathcal{V})$, carrying with itself a 
coordinate system. In order to simplify notations, we will denote in the same way 
the corresponding point $\pi(I^R)$ of $ \mathbb{P}(\mathcal{V})$. That is:

\begin{definition}
 An {\bf observer}  of the projective semivaluation tree 
 $ \mathbb{P}(\mathcal{V})$ is either the point $O$ or a smooth branch $L$. 
 The set of observers is considered embedded inside $ \mathbb{P}(\mathcal{V})$ 
 through the map $R \to \pi(I^R)$, which will allow us to write simply $R$ 
 instead of $\pi(I^R)$. The triple $({\ld}_R, {\si}_R, {\mult}_R)$ is the  {\bf coordinate system} 
on the space $\mathbb{P}(\mathcal{V})$ determined by the observer $R$. 
\end{definition}

We list the essential properties of the coordinate system associated to any 
observer in the following 
three propositions (see \cite[Sections 3.3, 3.4, 3.6, 3.9]{FJ 04}):

\begin{proposition}
     Let $R$ be an observer of $\mathbb{P}(\mathcal{V})$. Consider 
     the tree $\mathbb{P}(\mathcal{V})$ as a poset with the order relation $\preceq_R$. 
     Then the following functions are increasing, surjective and continuous on finite subtrees:
          \begin{itemize}
              \item  ${\ld}_R : \mathbb{P}(\mathcal{V}) \to [\ld_R (R) , \infty]$, 
                  where $\ld_O (O) = 2$  and $\ld_R (R)=1$. 
              
              \item  ${\si}_R : \mathbb{P}(\mathcal{V}) \to [0, \infty]$.
          \end{itemize}
     The multiplicity function ${\mult}_R : \mathbb{P}(\mathcal{V}) \to \N^* \cup \{ \infty\}$ 
    is increasing, surjective and lower semi-continuous 
     when $\N^* \cup \{ \infty \}$ is 
     endowed with the divisibility order relation (in which, by definition, any positive integer 
     divides $\infty$). 
\end{proposition}

\begin{proposition}  \label{reldiffprop} 
One has  the following differential relation for $P \in \mathbb{P}(\mathcal{V})  \setminus \{ R \}$:
     \begin{equation}  \label{relderiv}
             {\mult}_R(P) = \lim_{P_- \to P,  \, P_- \prec_R P} \, \, \, 
                   \dfrac{{\ld}_R(P) - {\ld}_R(P_-)}{{\si}_R(P) - {\si}_R(P_-)}.    
    \end{equation}
  That is, one has in integral form:
    \begin{equation}  \label{relint1}
        {\ld}_R( P) - {\ld}_R(R) = \int_{[RP]} {\mult}_R(p) \: d \: {\si}_R(p), 
    \end{equation}
    
    \begin{equation}  \label{relint2}
         {\si}_R( P) - {\si}_R(R) = \int_{[RP]} \frac{1}{{\mult}_R(p)} \: d \: {\ld}_R(p).
    \end{equation}
\end{proposition}

\begin{remark} \label{concise}
  We could have written the relation (\ref{relderiv}) more concisely as:
      \begin{equation}  \label{reldiff}
            d\: {\ld}_R = {\mult}_R \:  d \: {\si}_R.
      \end{equation}
   We will write it sometimes in this way, even if this has, strictly speaking, no meaning 
   in the usual interpretation of differential geometry, as there is no differentiable 
   structure on $\mathbb{P}(\mathcal{V})$ for which ${\ld}_R$ and ${\si}_R$ are both 
    differentiable. 
\end{remark}

\begin{proposition} \label{tripodform} {\bf (Generalized tripod formulae)}
   Let $R$ be an observer for $\mathbb{P}(\mathcal{V})$ and $P, Q \in \mathbb{P}(\mathcal{V})$ 
   be arbitrary.  Recall that $\langle R, P, Q \rangle$ denotes the center of the tripod determined 
   by $R, P, Q$ in the tree $\mathbb{P}(\mathcal{V})$ (see Definition \ref{trip}).  Then:
         \begin{equation}  \label{trip1}
             {\si}_R(\langle R, P, Q \rangle) = \bra{P,Q}_R, 
         \end{equation}
     Equivalently:
          \begin{equation}  \label{trip2}
              {\si}_R(\langle R, P, Q \rangle) = \frac{\bra{\nu^P, \nu^Q} }{\bra{I^R, \nu^P} \bra{I^R, \nu^Q}},
          \end{equation}
        where $\nu^P, \nu^Q \in \mathcal{V}^*$ are arbitrary semivaluations representing 
        $P$ and $Q$ respectively. 
\end{proposition}

Proposition \ref{tripodform} generalizes the tripod formula of Proposition \ref{intfromEW}. This is 
not obvious, as that proposition dealt with contact complexities and the previous one deals 
with self-interactions. In fact, both functions ${\ic}_L$ and $ {\si}_L$ 
coincide if one embeds naturally the Eggers-Wall 
tree $\Theta_L$ in the space $\mathcal{V}_L$ of semivaluations normalized relative to $L$. 
This embedding is the subject of next section, the coincidence of the two functions 
being part of the content of its Theorem \ref{embcurve}.

\section{The valuative embedding of the Eggers-Wall tree} \label{valemb}

In this section we explain the construction 
and some properties of 
a canonical embedding of the Eggers-Wall tree $\Theta_L(C)$ 
into the projective semivaluation tree $\mathbb{P}(\mathcal{V})$ 
(see Definition \ref{defemb} and Theorem \ref{embcurve}). 
Then we prove the result announced in the title of the paper (see Theorem 
\ref{proj_lim}).
\medskip

As usual, $(x,y)$ is a coordinate system such that $Z(x) =L$. 
 If $\xi \in  \C[[ x^{1/ \N} ]]$ 
and $\alpha \in (0, \infty]$, consider the set of Newton-Puiseux series which coincide 
with $\xi$ up to the exponent $\alpha$ (but not including $\alpha$):
\begin{equation}
 \mathcal{NP}_x(\xi, \alpha) := \{ \eta \in \mathcal\C[[ x^{1/ \N} ]] \ \mid  \ 
            \nu_x(\eta - \xi)  \geq \alpha \}.
\end{equation}           

Let $\xi \in \mathcal\C[[ x^{1/ \N} ]]$ and $\alpha \in (0, \infty]$ be fixed. Define the map
$$ \nu^{\xi, \alpha} :  \mathcal{O}  \to  [0, \infty]$$ by:  
\begin{equation} \label{defsvser}
\nu^{\xi, \alpha}  (f) :=  \inf \{\nu_x(f(x, \eta)) \ \mid
               \ \eta \in \mathcal{NP}_x(\xi, \alpha) \}.
\end{equation}
Define also the map $ \nu^{\xi, 0} :  \mathcal{O}  \to  [0, \infty]$ by:
   \begin{equation} \label{defsvser0}
              \nu^{\xi, 0}  :=  \mathrm{ord}^L .
    \end{equation}

 \begin{remark}
    The infimum in the definition (\ref{defsvser}) is not always a minimum. For instance, if 
    $\xi = x$, $\alpha \in (0, 1)$ is irrational and $f(x,y) = y$, then 
    $\nu_x(f(x, \eta)) = \nu_x(\eta)$ may take any \emph{rational} value in the interval 
    $[\alpha, \infty]$ when $\eta$ varies in 
       $ \mathcal{NP}_x(\xi, \alpha) = \{ \eta \in \mathcal\C[[ x^{1/ \N} ]] \  \mid \  \nu_x(\eta) > \alpha \}$.
    In fact, as an immediate consequence  of Proposition 
    \ref{geomval} below, one may prove that the infimum is a minimum  precisely when 
    $\alpha$ \emph{is rational}. 
\end{remark}

\begin{remark}  \label{Berkrem}
      If one sets $|| \eta || := e^{- \nu_x(\eta)}$, one gets a multiplicative 
      non-archimedean norm on the $\C$-algebra $\C[[ x^{1/ \N} ]]$. 
      Then $ \mathcal{NP}_x(\xi, \alpha)$  is simply the closed ball of center 
      $\xi$ and radius  $e^{- \alpha}$ in this normed complex vector space.  
      The definition of the  function $ \mathcal{NP}_x(\xi, \alpha)$ parallels 
      Berkovich's construction of semi-norms on the $K$-algebra $K[X]$, 
      where $K$ is any non-archimedean 
      field, associating to each element of $K[X]$ its supremum on a given 
      closed ball of $K$  (see Berkovich \cite[Section 1.4.4]{B 90} and 
      Baker and Rumely \cite[Page xvi]{BR 10}). 
\end{remark}

We will see in Proposition \ref{nu-semi} 
that the map $\nu^{\xi, \alpha}$ is a semivaluation for any  choice of  $\xi$ and $\alpha$. 
Let us understand first in terms of Eggers-Wall trees what is the value $\nu^{\xi, \alpha}(f)$ 
and for which series $\eta \in  \mathcal{NP}_x(\xi, \alpha)$, the number 
$\nu_x(f(x, \eta)) \in [0, \infty]$ achieves it.  

\begin{notation}
         If $\eta \in \C [[ x^{1/ \N} ]]$ is a Newton-Puiseux series,  we denote by $C_\eta$ 
         the branch defined by the minimal polynomial of $\eta$ in $\C[[x]][y]$. 
         Recall from Definition \ref{trip} that $\langle L, C_{\eta}, Z(f) \rangle$ 
         denotes the center of the tripod generated by the ends $L, C_{\eta}, Z(f)$ 
         of the Eggers-Wall tree $\Theta_L(C_{\eta} + Z(f)) $. 
\end{notation}

\begin{lemma}  \label{geomval}
    Let $f \in \mathcal{O}$ be irreducible and $\eta \in \mathcal\C[[ x^{1/ \N} ]]$.  
    Then:  
       \begin{equation}
     \label{geomval2}
               \nu_x(f(x, \eta))  =  \left\{ 
                \begin{array}{ccl}
                       (L \cdot Z(f)) \cdot {\ic}_L ( \langle L, C_{\eta}, Z(f) \rangle ) & 
                          \mbox{ if } & Z(f) \neq L, \\
                       1 & 
                          \mbox{ if } & Z(f) = L.
                  \end{array}  \right. 
        \end{equation}
\end{lemma}  

\begin{proof}  The formula is clearly true when $Z(f) =L$.  

If  $Z(f) \neq L$, notice that 
     $(L \cdot C_{\eta}) = {\de}_L(C_{\eta})$, where 
      ${\de}_L$ denotes the index function on $\Theta_L(C_{\eta})$, and 
      $C_{\eta}$ is viewed as the leaf of this Eggers-Wall tree. One has 
      $\eta = \tilde{\eta}(x^{1/ {\de}_L(C_{\eta})})$, where $\tilde{\eta}(t) \in \C[[t]]$. Therefore: 
      $$\nu_x(f(x, \eta)) = \frac{1}{{\de}_L(C_{\eta})} \cdot \nu_t(f(t^{{\de}_L(C_{\eta})}, \tilde{\eta}(t)) )= 
               \frac{(Z(f) \cdot C_{\eta})}{(L \cdot C_{\eta})} = 
                 (L \cdot Z(f)) \cdot {\ic}_L ( \langle L, C_{\eta}, Z(f) \rangle ),$$
        the last equality being a consequence of Theorem \ref{intcomp}. The proof 
        is finished in all cases. 
        
        Note that when $Z(f) = C_{\eta}$, we have $f(x, \eta) =0$ and 
          $ \langle L, C_{\eta}, Z(f) \rangle = C_{\eta}$, which shows that both sides of 
          the equality (\ref{geomval2}) are $\infty$.
\end{proof}

\begin{proposition}  \label{valtree} 
    Let $f \in \mathcal{O}$ be irreducible, $\xi \in \mathcal\C[[ x^{1/ \N} ]]$ and $\alpha \in [0, \infty]$. 
    Denote by $P(\alpha) \in \Theta_L(C_{\xi})$ the unique point 
    with exponent $\alpha$.
   Then: 
              \begin{equation} \label{valtree2} 
            \nu^{\xi, \alpha}(f) =   \left\{ 
                \begin{array}{ccl}
            (L \cdot Z(f)) \cdot {\ic}_L( \min \{P(\alpha), \langle L, C_{\xi}, Z(f) \rangle 
                    \} ) & 
                          \mbox{ if } & Z(f) \neq L, \\
                       1 & 
                          \mbox{ if } & Z(f) = L.  
                  \end{array}  \right.  
             \end{equation}
              the minimum being taken with respect to the partial order 
                    $\preceq_L$ on the tree  $\Theta_L(C_{\xi} + Z(f)) $.
\end{proposition}

\begin{proof} If $Z(f) =L$, then the equality results from the fact that $\nu^{\xi, \alpha}(x) = 1$ 
    for all $\alpha \in [0,\infty]$. 
    
    We assume from now on that $Z(f)  \neq L$.
    
  \noindent   $\bullet$    {\bf Suppose first that $\alpha = 0$.}   Then, by definition, 
    $\nu^{\xi, 0} = \mathrm{ord}^L$. As we assumed that $Z(f)  \neq L$, 
    this implies that $\nu^{\xi, 0}(f) = 0$. But the right-hand side is also $0$, because 
    $\min \{P(0), \langle L, C_{\xi}, Z(f) \rangle \} = P(0) = L$,  ${\ic}_L(L) =0$, 
    and $(L \cdot Z(f)) < + \infty$. 

  \noindent  $\bullet$  {\bf Suppose now that $\alpha >0$.}  
        The condition $\eta \in \mathcal{NP}_x(\xi, \alpha)$ implies that 
        the attaching point $\pi_{[L, C_{\xi}]}(C_{\eta}) = \langle L, C_{\xi}, C_{\eta} \rangle$ 
        of $C_{\eta}$ in $\Theta_L(C_{\xi})$ belongs to the segment $[P(\alpha),  C_{\xi}]$. 
        We will consider two cases, according to the position of $\langle L, C_{\xi}, Z(f) \rangle$ 
        relative to $P(\alpha)$. 
        \medskip

              \noindent
             -- Assume that $\langle L, C_{\xi}, Z(f) \rangle \prec_L P(\alpha)$
                (see the tree on the left of Figure \ref{fig:Both1}). 
                
                \medskip 
                
            \noindent                
        This implies the equality $\langle L, C_{\eta}, Z(f) \rangle = 
                        \langle L, C_{\xi}, Z(f) \rangle $ for all 
                  $\eta \in  \mathcal{NP}_x(\xi, \alpha)$. 
                  We deduce the assertion from Formula (\ref{geomval2}) since: 
                    $$\nu^{\xi, \alpha}(f) = (L \cdot Z(f)) \cdot {\ic}_L(\langle L, C_{\xi}, Z(f) \rangle   ).$$

              \noindent
               -- Assume that $\langle L, C_{\xi}, Z(f) \rangle \succeq_L P(\alpha)$
                (see the tree on the right in Figure \ref{fig:Both1}). 
                
                \medskip 
                
                \noindent 
                  When $\eta$ varies in $\mathcal{NP}_x(\xi, \alpha)$, the point 
                  $\langle L, C_{\eta}, Z(f) \rangle$ varies surjectively in the set of rational points 
                  of the segment  $[P(\alpha), \langle L, C_{\xi}, Z(f) \rangle ] $. Since those points 
                 are dense in this segment, we deduce from  Formula (\ref{geomval2}) that: 
                          $$   \begin{array}{ll}
                          \nu^{\xi, \alpha}(f)  & = (L \cdot  Z(f)) \cdot \inf \{ {\ic}_L ( P) \ | \ P \in 
                                 [P(\alpha), \langle L, C_{\xi}, Z(f) \rangle ] \mbox{ is rational } \} = \\
                                   & = (L \cdot Z(f)) \cdot {\ic}_L(P(\alpha)).
                             \end{array}$$
              
              By combining the results of the two cases, we get the announced 
              conclusion  for $\alpha >0$.                  
\end{proof}

\begin{figure}[h!] 
\vspace*{6mm}
\labellist \small\hair 2pt 
\pinlabel{$L$} at 105 6
\pinlabel{$C_{\xi}$} at 30 256
\pinlabel{$\langle L ,C_\xi, Z(f) \rangle$} at 110 115
\pinlabel{$Z(f)$} at 140 260
\pinlabel{$P(\alpha)$} at 0 190

\pinlabel{$L$} at 312 6
\pinlabel{$C_{\xi}$} at 245 256
\pinlabel{$\langle L ,C_\xi, Z(f) \rangle$} at 320 115
\pinlabel{$Z(f)$} at 348 261
\pinlabel{$P(\alpha)$} at 240 60
\endlabellist 
\centering 
\includegraphics[scale=0.50]{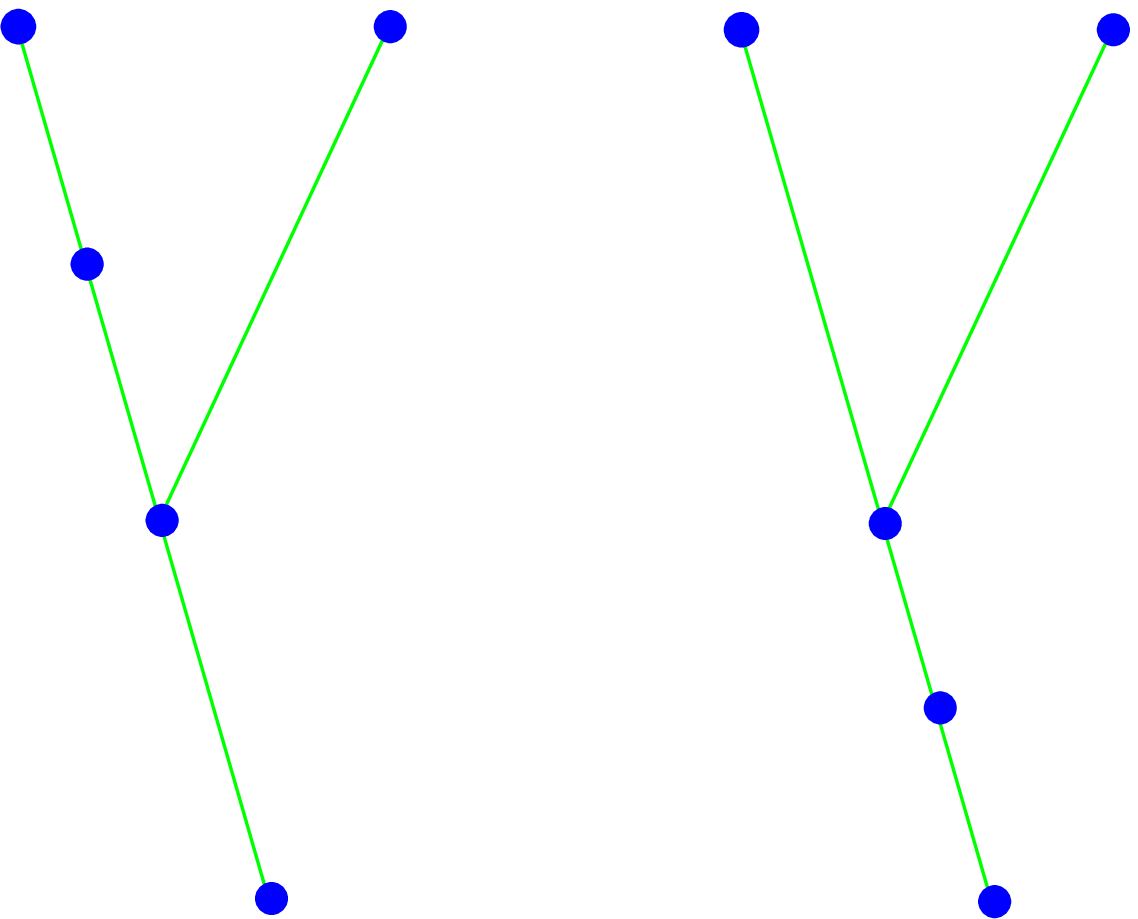} 
\caption{One has to compare $P(\alpha) $ and $\langle L ,C_\xi, Z(f) \rangle$}
\label{fig:Both1}
\end{figure} 

We need also the following lemma in order to prove in Proposition 
\ref{nu-semi} that the map $\nu^{\xi, \alpha}$ is a semi-valuation:  

\begin{lemma}  \label{corless}
      Let us fix $\xi \in \mathcal\C[[ x^{1/ \N} ]]$ and $\alpha \in (0, \infty]$. 
      If  $\eta_1, \eta_2 \in \mathcal{NP}_x(\xi, \alpha)$ and if  
      $f_1, f_2 \in \mathcal{O}$,      
      then  there exists $\eta \in \mathcal{NP}_x(\xi, \alpha)$ such that: 
    \[
   \nu_x(f_i(x, \eta))  \leq \nu_x(f_i (x, \eta_i)), \mbox{ for } i =1,2.  
    \]                                                                
\end{lemma}

\begin{proof} Let us denote by $\Theta$ the Eggers-Wall tree of the reduced effective 
divisor whose branches are $C_{\xi}$, $C_{\eta_1}$, $C_{\eta_2}$, $Z(f_1)$ and $Z(f_2)$. 
By definition, if $\eta_i \in \mathcal{NP}_x(\xi, \alpha)$, then the point  
$P_i = \langle L, C_\xi, C_{\eta_i} \rangle$ is 
$\succeq_L  \, P(\alpha)$ in the tree $\Theta$ for $i=1,2$.  
The segment $[P(\alpha), \min \{ P_1, P_2 \}]$ contains a rational point $P$ 
since its right hand extremity is rational. Let $C$ be a branch whose attaching point 
on the tree $\Theta$ is $P$. Since $P  \succeq_L  P(\alpha)$, 
there exists a Newton-Puiseux series $\eta$ of $C$ which belongs 
to $\mathcal{NP}_x(\xi, \alpha)$. Let us check that $\eta$ verifies the assertion.

If $Z(f_i)=L$ for some $i$ then the inequality of the statement trivially holds. 
Assume then that $f_i$ is irreducible  and $Z(f_i)\neq L$  for $i=1,2$.
Set $Q_i =  \langle L,C_{\xi}, Z(f_i) \rangle$.  We
get from the definition of the tree $\Theta$ that 
    $\langle L, C_{\eta_i}, Z(f_i) \rangle = \min \{ P_i, Q_i \}$. 
Similarly,  the attaching point 
$\pi_{[L,  Z(f_i) ]}   (C_{\eta})  = \langle L, C_{\eta}, Z(f_i) \rangle$ is equal to 
 $\min \{ P, Q_i \}$. 
By construction we obtain the inequality:
\[
\langle L, C_{\eta}, Z(f_i) \rangle  \, \preceq_L  \,  \langle L, C_{\eta_i}, Z(f_i) \rangle.
\]
In this case, the assertion follows from this and 
Formula (\ref{geomval2}), 
taking into account that the function ${\ic}_L$ is increasing. 

In the general case, the previous argument, applied to the irreducible components $f_{i,j} $ of 
$f_i = \prod_{j}  f_{i,j}$, shows that: 
\[
 \nu_x(f_{i,j} (x, \eta))  \leq \nu_x(f_{i,j} (x, \eta_i)).
\]
Since $\nu_x$ is a valuation we get: 
\[
\nu_x(f_i  (x, \eta)) = \sum_{j} \nu_x(f_{i,j} (x, \eta)) \leq \sum_{j} \nu_x(f_{i,j} (x, \eta_i)) 
                               = \nu_x(f_i  (x, \eta_i)).
\]
\end{proof}

\begin{proposition} \label{nu-semi}
     The map  $\nu^{\xi, \alpha}$ belongs to the set $\mathcal{V}_L$ of  
      semivaluations normalized relative to $L = Z(x)$. 
\end{proposition}

\begin{proof}  If $\alpha = 0$, the statement is clear, because $\nu^{\xi, 0} = 
  \mathrm{ord}^L$.
  
  Consider from now on the case $\alpha > 0$.
    Let us prove successively the three conditions (\ref{add}), (\ref{ineq}), (\ref{const})
    of Definition \ref{valdef}. 
    
    \medskip
    \noindent
       $\bullet$ {\bf Proof of condition (\ref{add})}.  Consider two functions $f, g \in \mathcal{O}$. 
           As $\nu_x$ is a valuation of $\mathcal\C[[ x^{1/ \N} ]]$, we have: 
            $$\nu_x(f(x, \eta) \cdot g(x, \eta)) =   \nu_x(f(x, \eta)) + \nu_x(g(x, \eta)) $$ 
            for all $\eta \in \mathcal{NP}_x(\xi, \alpha)$. But, by the definition of $\nu^{\xi, \alpha}$: 
                $\nu_x(f(x, \eta)) \geq \nu^{\xi, \alpha}(f)$ and 
                      $\nu_x(g(x, \eta)) \geq \nu^{\xi, \alpha}(g)$. 
             This implies that:
                $\nu_x(f(x, \eta) \cdot g(x, \eta)) \geq  \nu^{\xi, \alpha}(f) + \nu^{\xi, \alpha}(g)$.
             Passing to the infimum of the left-hand-sides over $\eta \in \mathcal{NP}_x(\xi, \alpha)$, 
             we get the inequality: 
                $$\nu^{\xi, \alpha}(f \cdot g) \geq  \nu^{\xi, \alpha}(f) + \nu^{\xi, \alpha}(g).$$
             We want now to show that in fact this is an equality. We will prove this  
             by showing that one has always also the converse inequality:
                \begin{equation} \label{converse}
                    \nu^{\xi, \alpha}(f \cdot g) \leq  \nu^{\xi, \alpha}(f) + \nu^{\xi, \alpha}(g).
                \end{equation}
              Let us consider $\eta_1,  \eta_2 \in \mathcal{NP}_x(\xi, \alpha)$. By Lemma \ref{corless}
              there exists a series  $\eta \in \mathcal{NP}_x(\xi, \alpha)$ such that 
                   \[
                                            \begin{array}{lcl}
                                        \nu_x(f(x, \eta))  & \leq  & \nu_x(f(x, \eta_1)),
                                         \\
                                          \nu_x(g(x, \eta))   & \leq &  \nu_x(g(x, \eta_2)).
                                    \end{array}
                \]
               By summing these  inequalities, we get:
                 $$\nu_x((f\cdot g)(x, \eta))  \leq \nu_x(f(x, \eta_1)) + \nu_x(g(x, \eta_2)). $$
               Therefore:
                  $$\nu^{\xi, \alpha}(f\cdot g)  \leq \nu_x(f(x, \eta_1)) + \nu_x(g(x, \eta_2)). $$
                This being true for all $\eta_1 , \eta_2 \in \mathcal{NP}_x(\xi, \alpha)$, 
                we may take the infimum over those choices, and get the desired 
                converse inequality (\ref{converse}). 
                        
     \medskip
     \noindent
       $\bullet$ {\bf Proof of condition (\ref{ineq})}. Consider again two functions 
            $f, g \in \mathcal{O}$. As $\nu_x$ is a valuation of $\mathcal\C[[ x^{1/ \N} ]]$, we have: 
            $$\nu_x(f(x, \eta) + g(x, \eta)) \geq \min \{ \nu_x(f(x, \eta)), \nu_x(g(x, \eta)) \}$$ 
            for all $\eta \in \mathcal{NP}_x(\xi, \alpha)$. 
             This implies, as in the previous reasoning, that:
                $$\nu_x(f(x, \eta) + g(x, \eta)) \geq \min 
                   \{ \nu^{\xi, \alpha}(f), \nu^{\xi, \alpha}(g) \}.$$
             Passing to the infimum of the left-hand-sides over $\eta \in \mathcal{NP}_x(\xi, \alpha)$, 
             we get the desired inequality: 
                $$\nu^{\xi, \alpha}(f + g) \geq \min \{ \nu^{\xi, \alpha}(f), \nu^{\xi, \alpha}(g) \}.$$
       
     \medskip
     \noindent
       $\bullet$ {\bf Proof of condition (\ref{const})}. This is immediate from the definition. 
       
      \medskip Finally notice that   $\nu^{\xi, \alpha}(x) = 1$, thus the semivaluation 
         $\nu^{\xi, \alpha}$ is normalized 
         relative to $L$.     
\end{proof}

\begin{remark} \label{especial}
     It is clear from the definition that if $0 < \alpha < \infty$, then the semivaluation 
     $\nu^{\xi, \alpha}$ is actually a valuation 
     centered at $O$ in the sense of Definition \ref{valdef}. 
     We know that $\nu^{\xi, 0} =\mathrm{ord}^L$, while 
     by Proposition \ref{valtree} one has 
     $\nu^{\xi, \infty }  = I_L^{C_\xi}$ (see Definition \ref{def:val} and Formula  (\ref{normaliznot})). 
      This is because $\mathcal{NP}_x (\xi, \infty) = \{ \xi \}$ and 
     for any irreducible element $f \in \mathcal{O}$, we have:
     \[
        \nu^{\xi, \infty} (f) = \nu_x (f(x, \xi))
         \stackrel{(\ref{geomval2})}{=} 
         (L \cdot  Z(f)) \, {\ic}_L( \langle L, C_{\xi}, Z(f) \rangle ) 
         \stackrel{(\ref{f-intcomp})}{=}
          \frac{(C_\xi \cdot Z(f))}{(L \cdot C_\xi) } = I_L^{C_\xi} (f). 
       \]
\end{remark}

\begin{definition}  \label{defemb}
Let $C$ be a (possibly reducible) reduced germ of curve 
    on $S$ and $L$ be a smooth branch. 
We define the map: 
  \begin{equation}   \label{embew}
             \begin{array}{cccc}
               V_L : &   \Theta_L(C) & \to & \mathcal{V}_L \\
                                      &   P              & \to &    V_L^P := \nu^{\xi, \alpha}
            \end{array}   
        \end{equation}
if $P$ is the point of exponent $\alpha$ in the segment $[L, C_\xi]$ of $\Theta_L(C)$,  
where $C_\xi$ is a component of $C$. 
\end{definition}

The map $V_L$ is well-defined, in the sense that it does not depend 
on the choice of a suitable component 
$C_\xi$.  This results from the following proposition which allows 
to compute the values taken by $V_L^P$ on any branch (hence 
on any divisor, by the additivity property (\ref{add}) in the Definition \ref{valdef} of valuations):

  \begin{proposition}  \label{reformpoint}
     Let $C$ be a reduced germ on $S$ and $A$ be any branch on $S$. Fix a smooth reference 
     branch $L$. If 
     $P \in \Theta_L(C)$, then:
        $$V_L^P(A) =   \left\{ 
                \begin{array}{ccl}
                          (L \cdot A) \cdot {\ic}_L(\min \{ P, \langle L, P, A \rangle \})  &  
                          \mbox{ if } & A \neq L, \\
                       1 &  \mbox{ if } & A = L.  
                  \end{array}  \right. 
         $$
  \end{proposition}
  
  \begin{proof} If $A =L$, this results from Proposition \ref{valtree}. 
  
      Assume now that $A \neq L$. 
         Choose a branch $C_i$ of $C$ such that $P \in [L, C_i]$. Apply Proposition \ref{valtree}
        to  $C_{\xi} = C_i$ and $P(\alpha) = P$. We get:
            \begin{equation} \label{interm}
                V_L^P(A) = (L \cdot A) \cdot {\ic}_L(\min \{ P, \langle L, C_i, A \rangle \}).
            \end{equation}
         Analysing both possibilities $\langle L, C_i, A \rangle \prec_L P$ and 
        $\langle L, C_i, A \rangle \succeq_L P$ (compare with Figure \ref{fig:Both1}), we see that:
             $$\min \{ P, \langle L, C_i, A \rangle \} = \min \{ P, \langle L, P, A \rangle \}$$
          always holds. Formula (\ref{interm}) implies then the desired equality. 
  \end{proof}

We state now the embedding theorem of the Eggers-Wall tree 
in the $\R$-tree  of normalized semivaluations:

  \begin{theorem} \label{embcurve}
    The map  $V_L$ 
     is an increasing embedding of rooted trees, which sends the root $L$ of 
     $\Theta_L(C)$ onto the root $\mbox{ord}^L$ of $\mathcal{V}_L$ 
     and the end $C_i$ of $\Theta_L(C)$ onto the end $I_L^{C_i}$ 
     of $\mathcal{V}_L$ for each branch $C_i$ of $C$. 
     Under this embedding, the function $1 + {\ex}_L$ is identified with 
     the relative log-discrepancy ${\ld}_L$, the denominator 
     function ${\de}_L$ with the relative multiplicity ${\mult}_L$ 
     and the contact complexity ${\ic}_L$ with the self-interaction ${\si}_L$. 
\end{theorem}

\begin{proof} We will prove successively the various statements of the theorem. 
        
       \medskip
       \noindent
      $\bullet$ {\bf The map $V_L$ is increasing}. Consider two points 
           $P, Q \in \Theta_L(C)$, with $P \prec_L Q$. Therefore, there exists 
           a branch $C_i$ of $C$ such that $P, Q \in \Theta_L(C_i)$. In order 
           to simplify the notations, let us denote it simply by $C$. 
            
            Consider an arbitrary function $f \in \mathcal{O}$. By the definition of the 
            order relation $\preceq_L$ on $\mathcal{V}_L$, we want to show that 
            $V_L^P(f) \leq V_L^Q(f)$. It is enough to prove this inequality when 
            $f$ is irreducible, because it extends then to arbitrary $f$ by the 
            additivity property (\ref{add}) in the Definition \ref{valdef} of semivaluations. 
            
            Assume therefore that $f$ is irreducible. Let $A$ be the branch defined by it. 
            By Proposition \ref{reformpoint}, the inequality is equivalent to 
            ${\ic}_L(\min \{ P, \langle L, P, A \rangle \}) \leq {\ic}_L(\min \{ Q, \langle L, Q, A \rangle \})$. 
            But this is obvious, as $P \preceq_L Q$ implies $\min \{ P, \langle L, P, A \rangle \} 
            \preceq_L \min \{ Q, \langle L, Q, A \rangle \}$, and the function ${\ic}_L$ is increasing. 
      
      \medskip
      \noindent
       $\bullet$ {\bf The map $V_L$ is injective}. Let us consider two \emph{distinct} 
          points $P, Q \in 
          \Theta_L(C)$. We want to show that there exists a branch $A$ such that 
          $V_L^P(A) \neq V_L^Q(A)$. We will consider two cases, according to the 
          comparability or incomparability of $P$ and $Q$ for the partial order relation 
          $\preceq_L$.

              -- Assume that $P$ and $Q$ are comparable for $\preceq_L$, say $P \prec_L Q$.
              
                 \noindent
                 By restricting $C$ to a suitable branch of it,  
                 we can suppose that $C$ is irreducible and that $P, Q \in \Theta_L(C)$.   
                  Let $T$ be a \emph{rational} point of the open segment 
                  $(P,Q)$ of $\Theta_L(C)$, and let $A$ be a branch on $S$ whose attaching 
                  point $\langle L, C, A \rangle$ in $\Theta_L(C)$ is $T$ (see Figure \ref{fig:Comp}).

\begin{figure}[h!] 
\vspace*{6mm}
\labellist \small\hair 2pt 
\pinlabel{$L$} at 50 4
\pinlabel{$C$} at -15 256
\pinlabel{$T = \langle L, C, A \rangle$} at 115 115
\pinlabel{$A$} at 87 256
\pinlabel{$P$} at 35 60
\pinlabel{$Q$} at 3 190
\endlabellist 
\centering 
\includegraphics[scale=0.5]{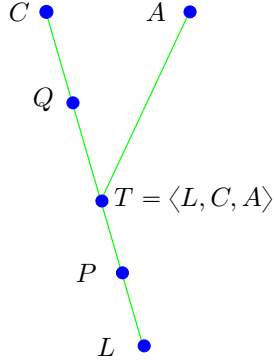} 
\caption{The case when $P$ and $Q$ are comparable}
\label{fig:Comp}
\end{figure}

\noindent
                     We have then: 
                     \[ 
                                    \min\{P, \langle L, P, A \rangle \} = P, \quad  
                                     \min\{Q, \langle L, Q, A \rangle \} = T.
                      \]
                    As the function ${\ic}_L$ is strictly increasing on $\Theta_L(C)$ and 
                    $P \prec_L  T$, we deduce that ${\ic}_L(P) < {\ic}_L(T)$. By Proposition 
                    \ref{reformpoint}, we conclude that $V_L^P(A) < V_L^Q(A)$.
               
              \medskip 
              
               -- Assume that $P$ and $Q$ are incomparable for $\preceq_L$.  
                     
                      \noindent
                   Denote $I := P \wedge_L Q = \langle L, P, Q \rangle$. We have the strict 
                   inequalities $I \prec_L P$, $I \prec_L Q$.  Choose a \emph{rational} point 
                   $T \in (I , Q)$.   
                   Therefore there exists a branch $A$ on $S$ such that its attaching point in 
                   $\Theta_L(C)$ is the point $T$ (see Figure \ref{fig:Incomp}). 

\begin{figure}[h!] 
\vspace*{6mm}
\labellist \small\hair 2pt 
\pinlabel{$L$} at 50 5
\pinlabel{$T$} at 105 195
\pinlabel{$I$} at 20 113
\pinlabel{$A$} at 63 275
\pinlabel{$P$} at 5 275
\pinlabel{$Q$} at 110 275
\endlabellist 
\centering 
\includegraphics[scale=0.50]{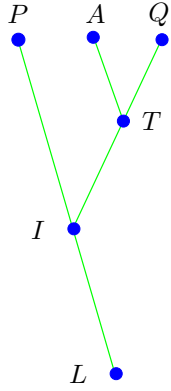} 
\caption{The case when $P$ and $Q$ are incomparable}
\label{fig:Incomp}
\end{figure}

We deduce that:
                     \[ 
                                    \min\{P, \langle L, P, A \rangle \} = \min\{P,  I \} = I, 
                                     \mbox{ and } 
                                     \min\{Q, \langle L, Q, A \rangle \} = \min \{Q, T \} = T.
                      \]
                     As the function ${\ic}_L$ is strictly increasing on $\Theta_L(C)$ and 
                    $I \prec_L  T$, we deduce that ${\ic}_L(I) < {\ic}_L(T)$. By Proposition 
                    \ref{reformpoint}, we conclude that $V_L^P(A) < V_L^Q(A)$.

         \medskip
         \noindent
       $\bullet$ {\bf The map $V_L$ is continuous}. It is enough to prove that 
           $V_L$ is continuous when $C$ is a branch. By the definition of the 
           weak topology on the semivaluation space $\mathcal{V}$, this amounts 
           to proving the continuity of the following map: 
                 $$\left\{ \begin{array}{ccc}
                                       \Theta_L(C)  & \to & [0, \infty] \\
                                            P  & \to & V_L^P(A)
                               \end{array} \right. $$
            for any fixed branch $A$. But this is an immediate consequence of 
            Proposition \ref{reformpoint}.

             \medskip
             \noindent
       $\bullet$ {\bf The map $V_L$ sends $L$ to $\mbox{ord}^L$ and $C_i$ to $I_L^{C_i}$}. 
       This follows from Remark \ref{especial}.

           \medskip
           \noindent
       $\bullet$ {\bf The map $V_L$ identifies  ${\ic}_L$ with ${\si}_L$}. We will prove 
         this in restriction to the rational points of $\Theta_L(C)$. Such a point 
         is the center $\langle L, C_i, A \rangle$ of a tripod, where 
         $C_i$ is a branch of $C$ and $A$ is a certain branch on $S$. 
         We may assume as before that $C$ is irreducible (therefore $C_i = C$), 
         and that we look at the point 
         $\langle L, C, A \rangle$. Then the fact that $V_L$ is continuous, injective and 
         increasing implies that $V_L^{\langle L, C, A \rangle} = 
         \langle \mbox{ord}^L, I^{C}_L, I^{A}_L  \rangle$. 
         
         By Theorem \ref{intcomp}, we have:
           $$  {\ic}_L (\langle L, C, A \rangle) = \frac{(C \cdot A)}{(L \cdot C) (L \cdot A)}.$$
           
         By Theorem \ref{tripodform}, we also have:
           $$ {\si}_L( \langle \mbox{ord}^L, I^{C}_L, I^{A}_L)  \rangle) = 
                   \frac{\bra{I^C , I^{A}}}{ \bra{\mbox{ord}^L, I^C} \bra{\mbox{ord}^L, I^{A}}}. $$
           Proposition \ref{geomint} shows then that the right-hand sides of the 
           two previous equalities coincide. 
           
           As the statement is true for the rational points, which are dense in 
           $\Theta_L(C)$, and both ${\ic}_L$ 
           and ${\si}_L$  are continuous, we deduce that the statement is true 
           for all points.

        \medskip
        \noindent
       $\bullet$ {\bf The map $V_L$ identifies ${\de}_L$ with ${\mult}_L$}. We reason analogously, 
         by first proving the statement for rational points of $\Theta_L(C)$. 
         Let $P$ be such a point. We may choose a branch $A$ such that 
         $P \in \Theta_L(A)$ and ${\de}_L(P) = {\de}_L(A) =  \bra{I^L, A}$. Moreover, in this case 
         $\bra{I^L, A} = \min \{ \bra{I^L, A'} \  |  \   P \preceq_L A' \}.$ By definition, this last 
         minimum is ${\mult}_L(V_L^P)$. The conclusion follows. 
            
          \medskip
          \noindent
       $\bullet$ {\bf The map $V_L$ identifies $1 + {\ex}_L$ with ${\ld}_L$}. 
           As a direct consequence of the differential relations 
           $d {\ld}_L = {\mult}_L \: d {\si}_L$ and  $d {\ex}_L = {\de}_L \: d {\ic}_L$, we see 
           that there exists a constant $a \in \R$ such that ${\ex}_L + a$ is sent to 
           ${\ld}_L$ by the map $V_L$. As ${\ex}_L(L) =0$ and ${\ld}_L(\mathrm{ord}^L) = 1$, 
           we deduce that $a =1$. 
\end{proof}

 \begin{remark} 
     As it was the case with the Eggers-Wall tree 
       $\Theta_L(C)$ itself, the map $V_L$ depends only on $C$ and on the smooth branch $L$ 
       defined by $x$. Namely, $V_L^P$ is the unique semivaluation of the segment 
       $[\mathrm{ord}^L, I^C_{L}] 
      \subset \mathcal{V}_L$ whose self-interaction is equal to ${\ic}_L(P)$. 
  \end{remark}

  \begin{remark}
      A variant of the map $V_L$ was already defined by Favre and Jonsson in 
      \cite[Prop. D1, page 223]{FJ 04}. They started from a generic Eggers-Wall tree 
      and a generic version of the exponent function. They associated to any 
      point of it of exponent $e$, situated on the segment $[O, C_i]$,  
      the unique point of the segment $[I^O, I^{C_i}_{O}]$ with log-discrepancy 
      $1 + e$ relative to $O$. They did not give another interpretation of that map, 
      for instance analogous to our definition (\ref{defsvser}). 
  \end{remark}

The following lemma proves that if 
$\nu$ and $\nu'$ are two different semivaluations in $\mathcal{V}_L$,  
then there exists a branch $A$ such that 
the attaching points of $\nu$ and $\nu'$ on the segment 
$[\mathrm{ord}^L, I^{A}_L]$ are different.

\begin{lemma} \label{separation} $\:$ 

     \begin{enumerate} 
          \item Let $\nu$ and $\nu'$ be two different semivaluations in $\mathcal{V}_L$. 
     Then, there exists a branch $A$ such that $\nu(A) \ne \nu'(A)$. 
            \item If $A$ is such a branch, 
     denote by $P =\langle \mathrm{ord}^L , I^A_L, \nu \rangle $ the center of the tripod 
     determined by the normalized semivaluations $ \mathrm{ord}^L , I^A_L $ and $ \nu$
     on the tree $\mathcal{V}_L$, and denote similarly $P' = \langle \mathrm{ord}^L , I^A_L, \nu' \rangle$. 
     Then, we have that $P \ne P'$. 
     \end{enumerate}
\end{lemma}

\begin{proof}
Assume that $\nu(A) = \nu'(A)$ for any branch $A$. 
Then, if  $h \in \mathcal{O}$, we have that $h = \prod_{j} h_j$ with $h_j$ irreducible. 
By hypothesis the semivaluations $\nu$ and $\nu'$ have the same value on the branch $A_j = Z(h_j)$. 
It follows that $\nu (h) = \sum_j \nu (A_j)  = \nu' (h)$. This proves the first statement.

By the tripod formula (\ref{trip2}) we get the relations:
\begin{equation} \label{trip3}
{\si}_L ( P ) = \frac{ \bra{\nu, I^A} }{ \bra{I^L, \nu} \: \bra{I^L, I^A} }
 \mbox{ and } {\si}_L ( P' ) = \frac{\bra{ \nu', I^A} }{\bra{I^L, \nu} \: \bra{I^L, I^A} }. 
\end{equation}
By Proposition \ref{geomint}, we have that:
\[
\begin{array}{lclclcl}
\bra{\nu, I^L} & =  & \nu(L),  &  \quad & \bra{\nu', I^L} & = &  \nu'(L),
\\
\bra{\nu, I^A} & =  & \nu (A), &  \quad &  \bra{\nu', I^A} & = & \nu' (A), 
\end{array}
\]
and $\bra{I^L, I^A} = (L \cdot A)$. 
In addition, $\nu(L) = \nu' (L) = 1$ since $\nu$ and $\nu'$ belong to $\mathcal{V}_L$. 
It follows that $L \ne A$ hence  $(L\cdot A)  \in \N^*$.
Since 
$\nu(A) \ne \nu' (A)$ it follows from (\ref{trip3}) that ${\si}_L ( P ) \ne {\si}_L ( P' ) $. 
Since the restriction to ${\si}_L$ to the segment $[\mathrm{ord}^L, I^A_L]$ is 
strictly increasing and $P, P'$ belongs to this segment it follows that $P \ne P'$, 
which proves the second statement.
\end{proof}

We prove now that the semivaluation space $\mathcal{V}_L$ is 
the projective limit of the Eggers-Wall trees 
$\Theta_L (C)$ of reduced plane curves, embedded 
by the map $V_L$. 
 
  \begin{theorem} \label{proj_lim}
Let us denote by $\mathcal{B}$ the set of branches at $O$ on the smooth surface $S$ and by 
$J$ the set consisting of finite subsets of $\mathcal{B}$. 
For any $j \in J$, we denote by $C_j$ the reduced plane curve singularity whose branches are
the elements of the set $j$. Denote by $\cV_{L,j}$ the subtree 
$\mathcal{V}_L  ( \Theta_L (C_j) )$ of $\mathcal{V}_L$. 
The collection  $(\cV_{L,j})_{j \in J}$ forms a projective system for the inclusion partial order.
If $\cV_{L, j} \subset \cV_{L,l}$, we denote by $\pi^l_{L,j} : \cV_{L,l} \to \cV_{L,j}$ 
the corresponding attaching map. Then: 
\begin{enumerate}
    \item The maps $\pi_{L,j}^l$ form a projective system of continuous maps.
      \item The attaching maps $\pi_{L,j}: \cV_L \to \cV_{L,j}$ glue into a homeomorphism 
               $\pi_L : \cV_L  \to \displaystyle{\lim_{\longleftarrow}} \: \cV_{L,j}$.            
\end{enumerate}              
\end{theorem}

\begin{proof}
The collection  $(\cV_{L,j})_{j \in J}$ form a projective
system for the inclusion partial order, since 
for any $j, k \in J$ there exists $l = j\cup k \in J$ such that 
$\cV_{L,j}\subset \cV_{L,l}$ and $\cV_{L,k} \subset \cV_{L,l}$. 

Notice that if $\cV_{L, j} \subset \cV_{L,l}$, we can understand
 the attaching map  $\pi^l_{L,j} : \cV_{L,l} \to \cV_{L,j}$ 
by using the embedding $V_L$, since for any $P \in \Theta_L (C_l)$ we have 
\[
\pi^l_{L, j} (V_L (P)) = V_L (\pi_{\Theta_L (C_j)} (P)),
\]
where $\pi_{\Theta_L (C_j)} : \Theta_L (C_l) \to \Theta_L (C_{j})$, is the surjective attaching map 
of Definition \ref{defatt} (whose image is  $\Theta_L (C_j)  \subset  \Theta_L (C_l)$). 
This implies that the maps $\pi_{L,j}^l$ form a projective system of continuous maps.
\medskip 

Now we apply Theorem \ref{R-theo} in this setting:
 
-- The only hypothesis we need to check is point (\ref{sufmany}) in Theorem \ref{R-theo}. 
    This hypothesis hold by Lemma \ref{separation}.

-- Recall that the semivaluation space $\cV_L$ is compact. Therefore, Theorem \ref{R-theo},  
    applied to the projective system $\pi_{L,j}^l$,  implies that the map
     $\pi_L : \cV_L  \to \displaystyle{\lim_{\longleftarrow}} \: \cV_{L,j}$ is a homeomorphism.
\end{proof}

\medskip 

  In order to be able to compare the points of Eggers-Wall trees of various curves relative to 
various smooth branches considered as their roots, we embed them also in the 
fixed projective semivaluation tree $\mathbb{P}(\mathcal{V})$, instead of doing it in the varying 
trees $\mathcal{V}_L$:

\begin{definition} \label{embranch} 
    The {\bf valuative embedding} 
     of the Eggers-Wall  tree $\Theta_L(C)$ is the map \linebreak 
      $ \Psi_L := \pi_L  \circ V_L : \Theta_L(C) \to \mathbb{P}(\mathcal{V})$.  
\end{definition}

\begin{section}{Change of observer on the semivaluation space} \label{change-obs}

It is important to know how to change coordinates when one changes the observer. 
The aim of this section is to prove formulae expressing the functions 
$({\ld}_{R'},  \mult_{R'} ,  {\si}_{R'} )$ 
in terms of the functions $({\ld}_R,  \mult_R ,  {\si}_R )$, whenever $R$ and $R'$ are two 
distinct observers of the valuative tree $\bP(\cV)$. Combined with the embedding 
theorem \ref{embcurve}, these formulae of changes of coordinates are the main 
ingredients of the proof of the generalized inversion theorem \ref{muthm}. 

\medskip

The following proposition is an immediate consequence of Definition \ref{relcoord}:

\begin{proposition}  \label{changels}
   Let $R, R'$ be two observers of $\mathbb{P}(\mathcal{V})$. Then one has 
   the following formulae of change of coordinates from $R$ to $R'$:
    \begin{equation}  \label{changelambda}
           {\ld}_{R'} = \gamma^R_{R'} \cdot {\ld}_R,
    \end{equation}
        \begin{equation}  \label{changesigma}
           {\si}_{R'} = (\gamma^R_{R'})^2 \cdot {\si}_R,
    \end{equation}
   where: 
     \begin{equation} 
            \gamma^R_{R'}(P) = \frac{\bra{\nu^P, I^R}}{ \bra{\nu^P, I^{R'}}}
     \end{equation} 
    for any projective semivaluation $P \in \mathbb{P}(\mathcal{V})$. Here  
   $\nu^P \in \mathcal{V}$ is an arbitrary semivaluation representing $P$.
\end{proposition}

\begin{remark} \label{reminv}
   Seen as functions on $\mathbb{P}(\mathcal{V})$, 
   one has $\gamma^R_{R'} \cdot \gamma_R^{R'} = 1$. 
\end{remark}

The function $\gamma^R_{R'}$  is 
  expressed in the following way in terms 
of the relative interaction and self-interaction functions:

\begin{proposition}  \label{changesm}
    Assume that $R, R'$ are distinct observers on $\mathbb{P}(\mathcal{V})$. Then:
      $$\gamma^R_{R'}(P) = ( \bra{I^R, I^{R'}} \cdot {\si}_R( \langle R, R', P \rangle) )^{-1}
        \mbox{ for any } P \in \mathbb{P}(\mathcal{V}).$$
     In particular, if $R$ and $R'$ are transversal smooth branches, we have:
         $$\gamma^R_{R'}(P) = {\si}_R( \langle R,R',  P \rangle)^{-1} = 
           ({\ld}_R( \langle R,R', P \rangle) -1)^{-1}
        \mbox{ for any } P \in \mathbb{P}(\mathcal{V}).$$        
\end{proposition}

\begin{proof}
    The first equality is an immediate consequence of the tripod formula (\ref{trip2}). The 
    last equality is a consequence of formula (\ref{relint2}) and of the fact that ${\mult}_R$ is 
    identically equal to $1$ on the segment $[L , L'] \subset \mathbb{P}(\mathcal{V})$, 
    when $R$ and $R'$ are transversal smooth branches. This last fact is a consequence 
    of the Definition \ref{relmult} of the relative multiplicity function. 
\end{proof}

There is also a formula of change of coordinates for the relative multiplicity functions: 

\begin{proposition}  \label{changecrd}
   Let $R, R'$ be two distinct observers on $\mathbb{P}(\mathcal{V})$. Then 
   (see Figure \ref{fig:Chmult}):
       \begin{equation}  \label{changerm}
           {\mult}_{R'} = 
             \left\{   \begin{array}{ll}
                    1 & \mbox{ on } [ R' , \  \langle R,R',  O \rangle ] ,\\
                    \bra{I^R, I^{R'}} & \mbox{ on } (\langle R,R',  O \rangle, \  R], \\
                        \gamma_R^{R'}  \cdot {\mult}_R & \mbox{ on }   \mathbb{P}(\mathcal{V}) 
                          \:   \setminus  \:  [R', R] .
                          \end{array} \right.
    \end{equation}  
\end{proposition}

\begin{figure}[h!] 
\labellist \small\hair 2pt 
\pinlabel{$R$} at 4 -15
\pinlabel{$R'$} at 310 -15
\pinlabel{$\langle L, L', O \rangle$} at 130 10
\pinlabel{$O$} at 163 110
\pinlabel{$1$} at 230 42
\pinlabel{$\bra{I^R, I^{R'}}$} at 46 44
\endlabellist 
\centering 
\includegraphics[scale=0.45]{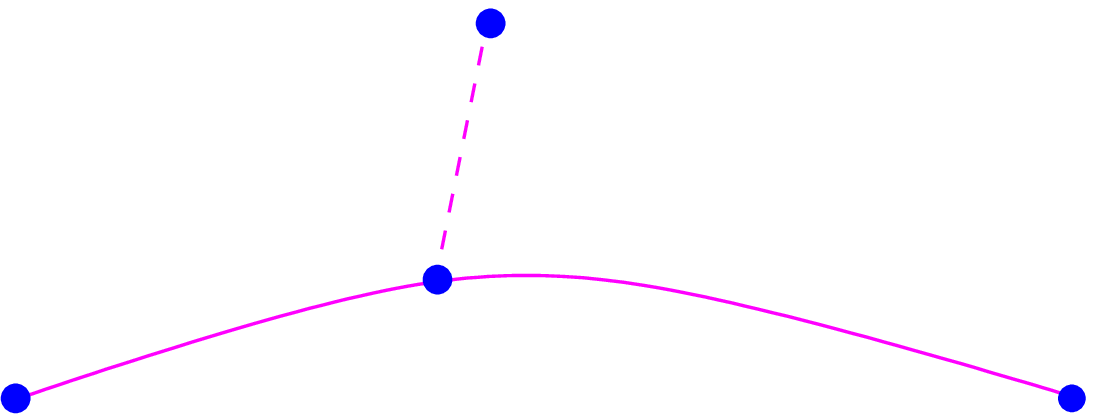} 
\caption{Some values of ${\mult}_{R'}$}
\label{fig:Chmult}
\end{figure}

\begin{proof}
    We prove the formulae when both observers are smooth branches $R = L, R'= L'$, leaving 
    to the reader the analogous reasoning in the remaining case when one of the observers 
    is the point $O$. We will consider successively the three possibilities listed 
    in the previous formula for the position of the point $P \in  \mathbb{P}(\mathcal{V})$ 
    relative to the tripod determined by $O, L, L'$. 
    
        \medskip
         \noindent
         $\bullet$ {\bf Assume that $P \in [L' ,\ \langle L, L', O \rangle ]$}. Consider a third smooth 
            branch $M$, \emph{transversal to $L'$} (see Figure \ref{fig:Segm1}). Then 
               $O \in [L', M]$ and ${\mult}_{L'}$ is 
            constantly equal to $1$ on $[L' , M]$. As $[L',  \  \langle L, L', O \rangle ] \subset [L', O]$, 
            we deduce the desired relation ${\mult}_{L'}(P) = 1$. 

                   \begin{figure}[h!] 
\vspace*{6mm}
\labellist \small\hair 2pt 
\pinlabel{$L$} at 4 67
\pinlabel{$L'$} at 310 66
\pinlabel{$\langle L, L', O \rangle$} at 172 140
\pinlabel{$O$} at 194 50
\pinlabel{$P$} at 244 84
\pinlabel{$M$} at 185 3
\endlabellist 
\centering 
\includegraphics[scale=0.45]{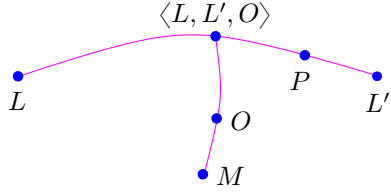} 
\caption{The case $P \in [L' , \ \langle L, L', O \rangle ]$}
\label{fig:Segm1}
\end{figure} 
    
          \medskip 
       \noindent
      $\bullet$ {\bf Assume that $P \in ( \langle L, L', O \rangle,  \ L]$}. 
  Apply then formula 
            (\ref{relderiv}) to ${\mult}_{L'}(P)$: 
               \begin{equation} \label{derex}
                  {\mult}_{L'}(P) = \lim_{P_- \to P ,  \, \, P_- \prec_{L'} P} \, \, \, 
                   \dfrac{{\ld}_{L'}(P) - {\ld}_{L'}(P_-)}{{\si}_{L'}(P) - {\si}_{L'}(P_-)} .
               \end{equation}
    In order to compute the limit (\ref{derex}) 
              we can assume that  $P_- \in ( \langle L, L', O \rangle,  \ L]$. 
Take then an auxiliary point $Q \in ( \langle L, L', O \rangle,  \ L]$ 
              (see Figure \ref{fig:Segm2}). 
              Our choice implies that  $\langle O, L', Q \rangle = \langle L, L', O  \rangle$.   
             
\begin{figure}[h!] 
\vspace*{6mm}
\labellist \small\hair 2pt 
\pinlabel{$L$} at 4 70
\pinlabel{$L'$} at 310 67
\pinlabel{$\langle L, L', O \rangle$} at 225 140
\pinlabel{$O$} at 185 7
\pinlabel{$P$} at 40 120
\pinlabel{$P_-$} at 90 130
\pinlabel{$Q$} at 140 140
\endlabellist 
\centering 
\includegraphics[scale=0.45]{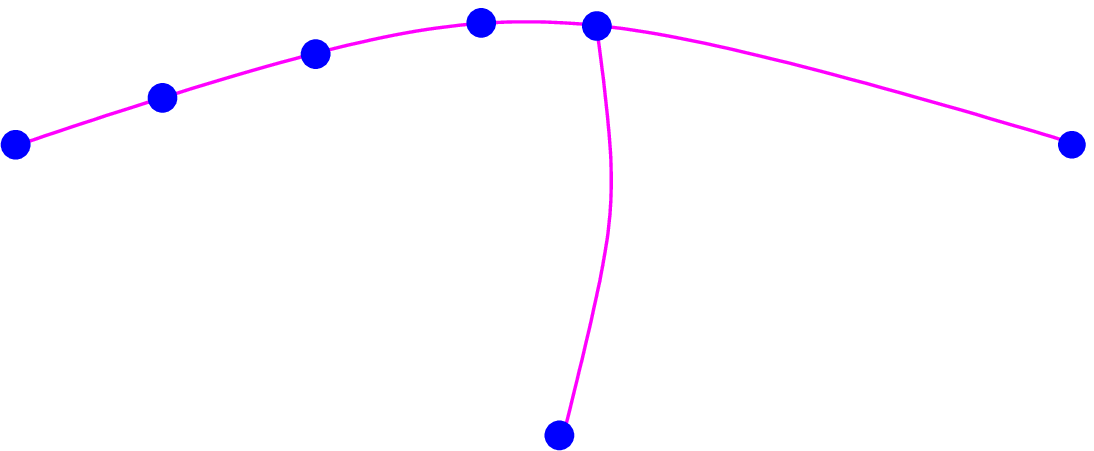} 
\caption{The case $P \in ( \langle L, L', O \rangle, \ L]$}
\label{fig:Segm2}
\end{figure}

\noindent       
By Proposition \ref{changesm} and the tripod formula 
              (\ref{trip2}) we deduce that: 
                \begin{equation} \label{derex2} 
                \gamma^O_{L'}(Q) = ( \bra{I^O, I^{L'}} \cdot {\si}_O( \langle O, L', Q \rangle) )^{-1} 
                     = {\si}_O( \langle  L, L', O \rangle) ^{-1} = (L \cdot  L')^{-1}. 
               \end{equation}
         We pass now from the observer $L'$ to $O$. That is, 
              we apply the formulae (\ref{changelambda}) and (\ref{changesigma})   
               to (\ref{derex}), with $R' = L'$ and $R= O$.              
             
               By (\ref{derex2}) the value of $\gamma^O_{L'}$ is constant on the segment  
               $( \langle L, L', O \rangle,  \ L]$ and we may factor it when computing the limit in (\ref{derex}). 
               We get: 
               \begin{equation} \label{derex3}
                 {\mult}_{L'}(P) =   (L \cdot  L')  \, \, 
                   \lim_{ P_- \to P ,   \, \, P_- \prec_{L'} P } \, \, \, \, 
                   \dfrac{{\ld}_{O}(P) - {\ld}_{O}(P_-)}{{\si}_{O}(P) - {\si}_{O}(P_-)}.                \end{equation}
               Since $P_- \in ( \langle L, L', O \rangle,  \ L]$, we have that 
            $P_- \prec_{L'} P $ is equivalent to $P_- \prec_{O} P$. 
            Therefore, the limit (\ref{derex3}) is equal to $\mult_O (P)$. 
             Notice that 
               $m_O$ is constantly equal to $1$ on $[O, L] \supset (  \langle L, L', O \rangle,  P)$, 
               thus $\mult_O (P) = 1$.  
              By Proposition \ref{geomint}, we get the desired 
                 equality ${\mult}_{L'}(P)  = (L \cdot  L') =  \bra{I^L, I^{L'}}$.

         \medskip
          \noindent
      $\bullet$
          {\bf Assume that $P \in \mathbb{P}(\mathcal{V}) \: \setminus \:  [L,  L']$}. Here the 
         reasoning is analogous to the one done in the previous case, but instead of changing 
         coordinates by replacing the observer $L'$ with $O$, one replaces it with $L$. 
         The main point is that one may compute the limit (\ref{derex}) by restricting 
         the points $P_-$ to the segment $(  \langle  L, L',  P \rangle, \  P)$. This implies that 
         $ \langle  L, L',  P_- \rangle =  \langle  L, L',  P \rangle$ (see Figure \ref{fig:Segm3}). 

\begin{figure}[h!] 
\vspace*{6mm}
\labellist \small\hair 2pt 
\pinlabel{$L$} at 4 70
\pinlabel{$L'$} at 308 70
\pinlabel{$\langle L, L', P \rangle$} at 120 140
\pinlabel{$P$} at 133 7
\pinlabel{$P_-$} at 150 60
\endlabellist 
\centering 
\includegraphics[scale=0.50]{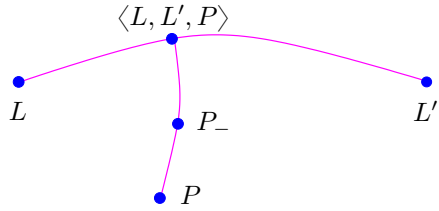} 
\caption{The case $P \in \mathbb{P}(\mathcal{V}) \setminus [L L']$ }
\label{fig:Segm3}
\end{figure} 

	\noindent
          In particular, by Proposition 
         \ref{changesm}, we get that $\gamma^L_{L'}(P_-) =  \gamma^L_{L'}(P) $. 
         Therefore one may factor $\gamma^L_{L'}(P)$ in the numerator 
         and $(\gamma^L_{L'}(P))^2$ in the denominator of the fraction in formula (\ref{derex}), 
         which implies by Remark \ref{reminv} that: 
            $$ {\mult}_{L'}(P) = \gamma_L^{L'}(P)  \cdot  \, \, 
                    \lim_{ P_- \to P , \,   P_- \prec_{L'} P} \, \, \, \, 
                    \dfrac{{\ld}_{L}(P) - {\ld}_{L}(P_-)}{{\si}_{L}(P) - {\si}_{L}(P_-)} .$$
           But one has also the inequality $P_- \prec_L P$, as $P_- \in (  \langle  L, L',  P \rangle,  P)$.  
           By (\ref{relderiv}), this implies that the last limit is equal to ${\mult}_L(P)$. We get the desired relation 
           ${\mult}_{L'}(P) = \gamma_L^{L'}(P)  \cdot  {\mult}_L(P)$. 
\end{proof}

\end{section}

\medskip
\end{document}